\documentclass[12pt]{amsart}
\usepackage[hmargin=0.8in,height=8.6in]{geometry}
\usepackage{amssymb,amsthm, times, enumerate}
\usepackage{delarray,verbatim} 
\usepackage{subcaption}
\usepackage{mathrsfs}
\usepackage{hyperref}
\usepackage{mathtools}
\usepackage{ifpdf}
\ifpdf
\usepackage[pdftex]{graphicx}
 \else
\usepackage[dvips]{graphicx}
\fi
\usepackage{xcolor}

\usepackage{float}
\newfloat{commentbox}{t}{loc}
\floatname{commentbox}{Algorithm}

\newcommand{\commentcontent}[1]{
  \fcolorbox{black}{white}{
    \parbox{0.95\linewidth}{#1}
  }
}

\usepackage{ifpdf}
\usepackage{color}

\numberwithin{equation}{section}

\newtheorem{theorem}{Theorem}[section]
\newtheorem{proposition}[theorem]{Proposition}
\newtheorem{corollary}[theorem]{Corollary}

\newtheorem{lemma}[theorem]{Lemma}

\newtheorem{assumption}{Assumption}[section]
\newtheorem{definition}{Definition}[section]

\newcommand {\E}{\mathbb{E}}
\newcommand{\diff}{{\rm d}}

\title{An Inventory System with Two Supply Modes and L\'evy Demand}

	\author[J. L. P\'erez]{Jos\'e-Luis P\'erez$^*$}

		\thanks{$*$\, Department of Probability and Statistics, Centro de Investigaci\'on en Matem\'aticasA.C. Calle Jalisco
		s/n. C.P. 36240, Guanajuato, Mexico. Email: jluis.garmendia@cimat.mx}

	\author[K. Yamazaki]{Kazutoshi Yamazaki$^\dagger$}
\thanks{$\dagger$\, School of 
Mathematics and Physics, The University of Queensland, St Lucia,
Brisbane, QLD 4072, Australia. Email: k.yamazaki@uq.edu.au,  qingyuan.zhang@uq.edu.au}

	\author[Q. Zhang]{Qingyuan Zhang$^\dagger$}

\date{\today}

\begin{document}

\maketitle
\noindent
\textbf{Abstract.} This study considers a continuous-review inventory model for a single item with two replenishment modes. Replenishments may occur continuously at any time with a higher unit cost, or at discrete times governed by Poisson arrivals with a lower cost. From a practical standpoint, the model represents an inventory system with random deal offerings. Demand is modeled by a spectrally positive L\'evy process (i.e., a L\'evy process with only positive jumps), which greatly generalizes existing studies. Replenishment quantities are continuous and backorders are allowed, while lead times, perishability, and lost sales are excluded. Using fluctuation theory for spectrally one-sided L\'evy processes, the optimality of a hybrid barrier policy incorporating both kinds of replenishments is established, and a semi-explicit expression for the associated value function is computed. Numerical analysis is provided to support the optimality result.
\\
		{\noindent  AMS 2020 Subject Classifications: 60G51, 93E20, 90B05. }

\section{Introduction}

In modern business environments, suppliers often offer price discounts at random times, with either fixed or variable discount amounts. For example, local supermarkets may run sales or special promotions due to temporary market imbalances---such as excess inventory---or as a result of promotional support from manufacturers. From the firm’s perspective, cost-effective supply chains can be developed by leveraging these periodic discounts through replenishment policies that combine regular and discounted replenishments. As noted in \cite{moinzadeh_discount_1997}, grocery retailers in the U.S. have invested millions of dollars to build warehouses specifically for storing excess inventory purchased at discounted prices. The prevalence of such scenarios is further supported by a motivating industry example presented in \cite{silver_random_1993}.

The scenario above highlights the relevance of inventory control models that incorporate two replenishment modes, a topic that has received considerable attention in the literature. Studies of such models in the context of price fluctuations include \cite{berling_stochastic_2011, hurter_regenerative_1967,  hurter_discount_1968, kalymon_stochastic_1971, moinzadeh_discount_1997}, among others. Of the aforementioned studies, Hurter and Kaminsky \cite{hurter_regenerative_1967, hurter_discount_1968} and Moinzadeh \cite{moinzadeh_discount_1997} focused on the setting where a fixed price discount is offered. Specifically, Hurter and Kaminsky \cite{hurter_regenerative_1967} studied replenishment policies for a system in which inventories can be replenished at either a regular unit price, available at all times, or a discounted unit price that arises randomly and expires immediately. Assuming stochastic demand and a specific replenishment policy, the authors analyzed both the behavior of the inventory system and the inventory-holding costs. This model was subsequently extended in \cite{hurter_discount_1968} to allow the discounted replenishment opportunities to have positive, random durations. In a related direction, Moinzadeh \cite{moinzadeh_discount_1997} analyzed a system in which demand is constant and price discounts are available at the arrival times of a Poisson process. Under specific assumptions on the inventory holding cost, analytical expressions are derived for the optimal policy parameters by solving a system of equations.

Kalymon \cite{kalymon_stochastic_1971} and Berling and Mart\'inez-de-Alb\'eniz \cite{berling_stochastic_2011} investigated replenishment policies in systems with stochastic supply prices. Kalymon \cite{kalymon_stochastic_1971} studied a periodic-review inventory model in which the unit price of supplies in each period follows a Markov process. Under certain conditions on the fixed ordering cost, the optimal policy is shown to be a price-dependent $(s, S)$ policy for both finite and infinite horizon cases. Berling and Mart\'inez-de Alb\'eniz \cite{berling_stochastic_2011} considered a setting closely related to that of \cite{kalymon_stochastic_1971} and extended the analysis to cases in which supply prices follow either a geometric Brownian motion or an Ornstein-Uhlenbeck process.

\subsection{Our model.} 
Motivated by the practical relevance of inventory systems with both regular and discounted replenishment modes, this study examines such a model. Specifically, in the present framework, inventory can be replenished either instantaneously at any time at a higher unit cost or instantaneously at random discrete times at a lower unit cost. Replenishment costs are proportional to the replenishment amount, and there are no fixed costs. Demand is modeled as a spectrally positive L\'evy process, and the random replenishment opportunities follow the arrival times of a Poisson process that is independent of the demand. The aim is to minimize the total cost, which consists of inventory-holding costs and replenishment costs. In view of existing studies, the present setting bears some similarity to that in \cite{moinzadeh_discount_1997}, which also assumes no lead time, perishability, or lost sales, and allows continuous replenishment quantities.

The present model is distinguished from existing inventory models with two replenishment modes by the following features. First, the inventory-holding cost under a replenishment policy is modeled by applying a function $f$ to the replenished inventory process. The function $f$ is assumed to be convex, as in classical inventory models such as that in \cite{benkherouf_optimality_2009, bensoussan_QVI_2005}. The set of functions satisfying this assumption is sufficiently rich and includes many commonly used forms, such as piecewise linear and quadratic functions. 

Second, we allow the demand process to be any spectrally positive L\'evy process, with either finite or infinite activity. This flexibility contrasts with existing models, which typically restrict the demand to follow processes of finite activity, such as a compound Poisson process or a Brownian motion. As noted in \cite{yamazaki_inventory_2017}, empirical studies of asset prices support the use of models involving processes with infinite activity. Since demand is often closely related to price dynamics, it is reasonable to model it using spectrally positive L\'evy processes with infinite activity.

Furthermore, the motivation for restricting the opportunities for discounted replenishments to independent Poisson arrival times is two-fold. To the best of our knowledge, Poisson arrival times are the only type of random discrete times for which semi-explicit expressions for the value function can be computed, owing to the memoryless property of the exponential inter-arrival times in a Poisson process. Moreover, as observed in \cite{leung_analytic_2014}, the Poisson inter-arrival model, where replenishments occur at the arrival times of an independent Poisson process, could potentially be used to approximate the constant inter-arrival model, in which replenishments occur at deterministic, uniformly spaced times. Hence, insights gained from analyzing the Poisson inter-arrival model may contribute to the understanding of the constant inter-arrival model.

\subsection{Contributions and method.}
In this study, we prove the optimality of a hybrid barrier replenishment policy. Given two barrier levels $a < b$, the policy maintains the inventory above $a$ via regular replenishment and, at each discounted replenishment opportunity, raises the inventory up to $b$ whenever it falls below $b$; see Figure \ref{Fig: processes} for an illustration. We explicitly identify a pair of optimal barriers as the solution to a pair of equations with semi-explicit expressions, and we obtain an analytic expression for the corresponding value function.

\begin{figure}[htbp]
\centering
\includegraphics[width=0.7\textwidth]{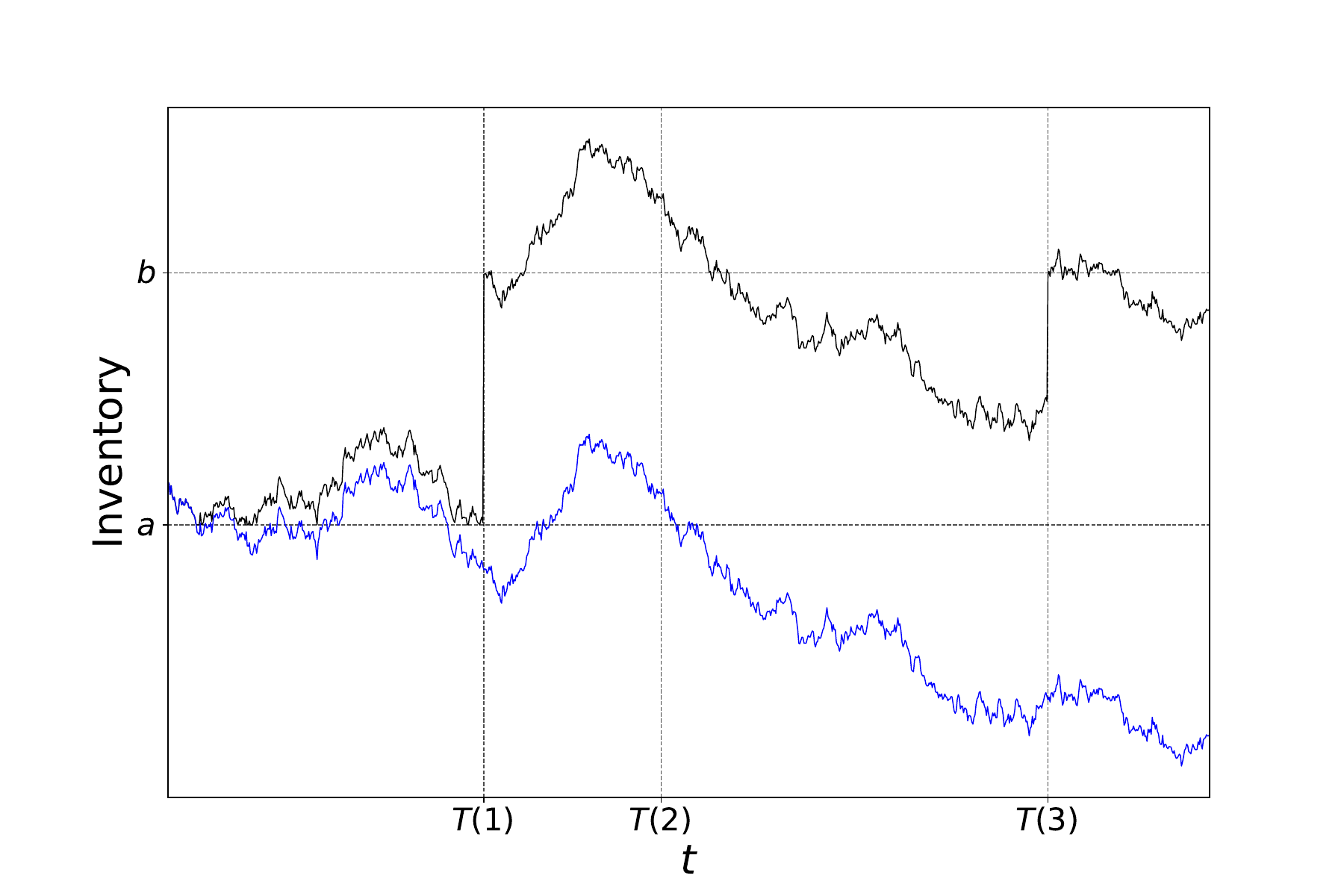}
\caption{\small Sample path of a standard Brownian motion (blue) and the replenished process (black). Discounted replenishment opportunities (denoted by $T(1)$, $T(2)$, $T(3)$) are shown as vertical dashed lines. At $T(1)$ and $T(3)$, the inventory level is below $b$, triggering replenishment. At $T(2)$ the inventory level exceeds $b$, so no replenishment occurs.}
\label{Fig: processes}
\end{figure}

Our analysis adopts the guess-and-verify approach, which can be applied effectively in this setting thanks to recent advances in the fluctuation theory of spectrally one-sided L\'evy processes. Specifically, our analysis of the model proceeds as follows:
\begin{enumerate}
    \item The present values of inventory-holding and replenishment costs under a hybrid barrier replenishment policy are computed using scale functions. The sum of these quantities yields the total cost of a hybrid barrier replenishment policy.   
    \item Using the analytical expression for the total cost, a pair of candidate barriers is selected and subsequently proven to exist. 
    \item The value function is shown to possess key properties such as smoothness and convexity. These properties, together with an application of a verification lemma, are used to establish rigorously the optimality of a candidate policy within the set of admissible policies.
    \item The optimality is illustrated numerically through an example and sensitivity analysis is performed.
\end{enumerate}

Our analysis overcomes several challenges not addressed in the existing literature. First, we derive the potential measure of the inventory process under a hybrid barrier replenishment policy, which corresponds to a spectrally negative L\'evy process subject to continuous reflection and reflection at Poisson arrival times (see Figure \ref{Fig: processes}). Second, we study the interaction between the two replenishment modes, which constitutes an interesting aspect of the model. Furthermore, we present an optimal solution, which is either (i) a hybrid barrier replenishment policy with finite barriers, or (ii) a pure discounted replenishment policy in which regular replenishment does not occur.

\subsection{Literature review}
Beyond inventory models with two replenishment modes, the inventory-holding cost in the inventory management literature is typically modeled as a convex function of the replenished inventory level, accumulated over time. The replenishment cost is often assumed to be proportional to the replenishment amount, with or without fixed costs incurred at each replenishment. Under suitable conditions, studies of inventory models with a single replenishment mode have established that the optimal policy takes the form of an $(s,S)$ policy in the impulse-control setting with fixed costs, and a barrier policy in the absence of fixed costs. Specifically, the compound Poisson model was analyzed by \cite{beckmann_compoundpoisson_1961}. The combined Brownian motion and compound Poisson model was subsequently investigated by \cite{benkherouf_optimality_2009, bensoussan_QVI_2005}, and later generalized to the spectrally positive L\'evy setting in \cite{yamazaki_inventory_2017}. Other versions of inventory models driven by L\'evy processes include \cite{yam_lostsales_2024}, among others.

In the context of inventory systems, the L\'evy process, a continuous-time analogue of the random walk, provides a natural model for aggregate demand. Spectrally positive L\'evy processes are particularly suitable, as they admit only positive jumps. For such processes, there exists a family of functions with semi-explicit expressions, known as scale functions, that yield compact formulations of key identities such as two-sided exit identities and resolvents. Although the application of scale functions in inventory models remains relatively underexplored, studies such as \cite{perez_optimal_2020, yamazaki_inventory_2017} have demonstrated their effectiveness in analyzing optimal replenishment policies. In these problems, scale functions allow for the direct computation of expected costs. The analytical properties of scale functions further make it possible to establish key properties of the total cost function, such as smoothness and convexity. Moreover, for certain spectrally one-sided L\'evy processes, explicit expressions for scale functions are available. For further reference on L\'evy processes and scale functions, see \cite{kuznetsov_theory_2013, kyprianou_fluctuations_2014}.

Our methodology is motivated by recent studies of L\'evy processes observed at Poisson arrival times. In \cite{ivanovs_exit_2016, landriault_potential_2018}, it was shown that the exit identities and potential measures of spectrally one-sided L\'evy processes under Poissonian observation can be expressed in terms of scale functions. Continuing in this line of research, L\'evy processes subject to interventions at Poisson arrival times have been investigated. In particular, the study in \cite{avram_parisian_2018} used scale functions to derive fluctuation identities for spectrally one-sided L\'evy processes reflected at Poisson arrival times. These fluctuation identities enabled the study of L\'evy models for stochastic control problems where controls can occur only at random discrete times. Such studies include de Finetti’s optimal dividend problem with periodic dividend payouts, as in \cite{noba_dividend_2018, zhao_dividend_2017}, which extend the classical Brownian motion model of \cite{avanzi_dividend_2014}.

Returning to the present model, we note that two related problems have been studied in the context of de Finetti’s optimal dividend problem, while another has been examined in the framework of singular control. In the context of dividend problems, the studies in \cite{avanzi_hybrid_2016, perez_dual_2020} examined models where the surplus follows a drifted Brownian motion and a spectrally positive L\'evy process, respectively, with dividends payable either continuously at any time at a higher per-dollar transaction cost, or at independent Poisson arrival times at a lower cost. Both problems aim to maximize the expected net present value of dividends distributed until ruin, which is the first time the surplus falls below zero. A fundamental difference between our problem and those in \cite{avanzi_hybrid_2016, perez_dual_2020} is that ours additionally incorporates an inventory-holding cost, which considerably complicates the analysis of the optimal policy. In another direction, the study in \cite{yamazaki_optimal_2025} considered a variant of our problem in which, at Poisson arrival times, the state process is decreased rather than increased.

\subsection{Outline}

The remainder of this paper proceeds as follows: Section 2 presents the mathematical formulation of our model. Section 3 introduces the scale functions and computes the costs under a hybrid barrier replenishment policy. Section 4 outlines the approach for selecting the barriers. Section 5 establishes, under certain conditions, the existence of a pair of candidate barriers, and Section 6 proves the optimality of the corresponding policy via a verification lemma. Finally, Section 7 concludes this study via numerical examples.

\section{Problem setting}
In the present single-item inventory system, the aggregate demand is modeled by a spectrally positive L\'evy process $D = (D(t); t \geq 0)$ with $D(0) = 0$. For completeness, we assume that $D$ is defined on a complete probability space $(\Omega, \mathcal{F}, \mathbb{P})$. Given an initial inventory level $x \in \mathbb{R}$, the inventory process $X = (X(t); t \geq 0)$ is a spectrally negative L\'evy process defined by
\[X(t) \coloneqq x - D(t), \quad t \geq 0.\]
To denote the law of $X$ conditional on the initial inventory level $x \in \mathbb{R}$, we adopt the standard notation $\mathbb{P}_x$, and use $\mathbb{E}_x$ for the corresponding expectation operator.

The sequence of random opportunities to replenish at the discounted price is denoted by $\mathcal{T} = (T(i); i \geq 0)$, where each $T(i)$ is an arrival time of a Poisson process $N = (N(t); t \geq 0)$ with intensity $\lambda > 0$, independent of the demand process $D$. Let $\mathbb{F} \coloneqq (\mathcal{F}(t); t \geq 0)$ denote the $\mathbb{P}$-enlargement of the filtration generated by the two-dimensional process $(D, N)$, representing the information from both the demand process and the random discounted replenishment opportunities.
A \emph{replenishment policy} $\pi = \{(R_c^\pi(t), R_p^\pi(t)); t \geq 0\}$ consists of a pair of processes such that, for any $t > 0$, the cumulative regular and discounted replenishments up to time $t$ are denoted by $R_c^\pi(t)$ and $R_p^\pi(t)$, respectively. Mathematically, both $R_c^\pi$ and $R_p^\pi$ are non-decreasing, c\`adl\`ag, and $\mathbb{F}$-adapted, with $R_p^\pi(0-) = R_c^\pi(0-) = 0$. To reflect the restriction on discounted replenishments, we assume that $R_p^\pi$ may be increased only at $t \in \mathcal{T}$, whereas $R_c^\pi$ may be increased at any $t \in \mathbb{R}$. Under $\pi$, the replenished inventory process $Y^\pi = (Y^\pi(t); t \geq 0)$ is given by
\[Y^\pi(t) := X(t) + R_c^\pi(t) + R_p^\pi(t), \quad t \geq 0.\]

For a positive discount factor $q$, an initial inventory level $x \in \mathbb{R}$, constants $K_c, K_p \in \mathbb{R}$, and $f: \mathbb{R} \to \mathbb{R}$, we define the net present value (NPV) of the total cost under $\pi$ as
\begin{equation}\label{Eq: total cost general form}
    v^\pi(x) \coloneqq \mathbb{E}_x \left[\int^\infty_0 e^{-qt} f(Y^\pi(t))\, \diff t + \int_{[0, \infty)} e^{-qt}(K_p \, \diff R_p^\pi(t) + K_c \, \diff R_c^\pi(t))\right].
\end{equation}
The value $f(x)$ represents the inventory-holding cost when $x > 0$, and the backorder cost when $x \leq 0$. The parameters $K_c$ and $K_p$ represent the unit costs, incurred by regular and discounted replenishments, respectively. We assume that
\begin{align}
    K_c > K_p. \label{Eq: C}
\end{align}
In the context of discount offerings, $K_c \leq K_p$ would imply that the discounted price is higher than the regular price, a case we can safely exclude from consideration. From a mathematical standpoint, dropping this assumption reduces the model to an uninteresting case, in which it is always optimal to incur continuous-time replenishments $R_c^\pi$ and disregard discrete-time replenishments $R_p^\pi$.

In \eqref{Eq: total cost general form}, the inventory-holding cost is modeled by the function $f$. For this function, we impose the following assumption, which is standard in the literature. 
\begin{assumption}\label{asm: f convexity}
    The function $f: \mathbb{R} \to \mathbb{R}$ is convex, piecewise continuously differentiable, and grows at most polynomially (i.e., there exist $a, M > 0$ and $N \in \mathbb{N}$ such that $|f(x)| \leq a |x|^N$ for $x \in (-\infty, -M) \cup (M, \infty)$) .
\end{assumption}

Define an affine perturbation of $f$ given by
\begin{equation}\label{Eq: tilde f}
    \tilde{f}(x) \coloneqq f(x) - \left(\frac{\lambda}{\lambda + q} K_p - K_c\right)(\lambda + q) x = f(x) + qK_c x + \lambda( K_c- K_p)x, \quad x\in \mathbb{R}.
\end{equation}
We impose the following assumption on the slope of $f$. 

\begin{assumption}\label{asm: f slope}
    The function $f: \mathbb{R} \to \mathbb{R}$ satisfies:
    \begin{itemize}
        \item[(1)] $\bar{a} \coloneqq \inf\{a \in \mathbb{R}: \tilde{f}'(a) \geq 0\}$ is finite.
        \item[(2)] $\bar{\bar{a}} \coloneqq \inf\{a \in \mathbb{R}: f'(a) + qK_p > 0\}$ is finite.
    \end{itemize}
\end{assumption}
We later show that Assumption \ref{asm: f slope} suffices to guarantee the existence of an optimal hybrid barrier replenishment policy. When Assumption \ref{asm: f slope}(1) is not satisfied, an optimal policy may still be identified; discussion of these cases is deferred to Section \ref{Sect: asm on f}. Assumption \ref{asm: f slope}(2) is typically imposed to exclude situations in which it is optimal to make arbitrarily large replenishments at the first exponential time; see \cite[Remark 1]{noba_singular_2023}.

Here and in the remainder of this paper, we use $f'(x)$ to denote the right-hand derivative of $f$ at $x$, whenever the standard derivative does not exist. Moreover, we use $x+$ and $x-$ to denote the right- and left-hand limits at $x \in \mathbb{R}$, respectively. The convexity of a function is always understood in the weak sense.

Under Assumptions \ref{asm: f slope}, we observe that
\[\bar{a} < \bar{\bar{a}},\]
as implied by the inequality $\tilde{f}'(x) = f'(x) + qK_c + \lambda(K_c - K_p) > f'(x) + qK_c > f'(x) + qK_p$.

Regarding the inventory process $X$, we make the following assumptions. First, recall that the \emph{Laplace exponent} of the spectrally negative L\'evy process $X$ takes the following form
\begin{align}\label{Eq: laplace exponent}
    \psi(s) \coloneqq \log\mathbb{E} [e^{s X(1)}] = \gamma s + \frac{\sigma^2}{2} s^2 + \int_{(-\infty, 0)} (e^{sz} - 1 - sz 1_{\{z > -1\}}) \,\mu(\diff z), \quad s \geq 0,
\end{align}
for some $\gamma \in \mathbb{R}$, $\sigma \geq 0$ and a \textit{L\'evy measure} $\mu$ on $(-\infty, 0)$ satisfying $\int_{(-\infty, 0)} (1 \wedge z^2) \, \mu(\diff z) < \infty$. When $X$ is of bounded variation, which occurs if and only if $\sigma = 0$ and $\int_{(-\infty, 0)} (1 \wedge |z|) \,\mu(\diff z) < \infty$, the Laplace exponent of $X$ simplifies to
\begin{equation}\label{Eq: drift finite variation}
    \psi(s) = \delta s + \int_{(-\infty, 0)} (e^{sz} - 1) \, \mu(\diff z), \quad s \geq 0, \quad \text{where} \quad \delta \coloneqq \gamma - \int_{(-1, 0)} z\, \mu(\diff z) \in \mathbb{R}. 
\end{equation}
We assume that $X$ is not the negative of a subordinator, which implies that $\delta > 0$ when $\sigma = 0$.

Additionally, we impose the following condition on the tail of the L\'evy measure. In conjunction with Assumption \ref{asm: f convexity}, it guarantees that $\mathbb{E}(X(1)) = \psi'(0+) \in (-\infty, \infty)$ and $\mathbb{E}_x[ \int^\infty_{0}e^{-qs} |f(X(s))|\, \diff s] < \infty$, for all $x \in \mathbb{R}$.
\begin{assumption} \label{asm: on X}
    There exists $\theta > 0$ such that $\int_{(-\infty, -1]} \exp(\theta |y|)\, \mu(\diff y) < \infty$.
\end{assumption}

Under Assumptions \ref{asm: f convexity} and \ref{asm: on X}, a replenishment policy $\pi$ is said to be \textit{admissible} if it satisfies the following integrability conditions: (1) $\mathbb{E}_x[\int^\infty_{0}e^{-qt} |f(Y^\pi(t))|\, \diff t] < \infty$, and (2) $\mathbb{E}_x[\int_{[0, \infty)}e^{-qt} (\diff R_p^\pi(t) + \diff R_c^\pi(t))] < \infty$, so that \eqref{Eq: total cost general form} is well-defined.  We use $\Pi$ to denote the set of admissible replenishment policies.

The objective of this problem is to derive the value function 
\begin{equation}\label{Eq: value}
    v(x) \coloneqq \inf_{\pi \in \Pi} v^\pi(x), \quad x \in \mathbb{R},
\end{equation}
and an associated optimal policy $\pi^* \in \Pi$ such that $v(x) = v^{\pi^*}(x)$, if such a policy exists.

\section{Hybrid barrier replenishment policies}
In this section, we introduce the class of hybrid barrier replenishment policies
\[\pi^{a,b} = \{(R_c^{a, b}(t), R_p^{a, b}(t)); t \geq 0\}, \quad a < b,
\]
within the framework of our model.

Fix two barriers $a < b$. Under $\pi^{a, b}$, the inventory process is replenished continuously at the lower barrier $a$, and at each time in $\mathcal{T}$, it is replenished up to the upper barrier $b$ whenever it lies below $b$. The aggregate regular and discounted replenishments up to time $t$ are denoted by $R_c^{a, b}(t)$ and $R_p^{a, b}(t)$, respectively. The replenished process is denoted by
\begin{equation}\label{Eq: controlled process}
    Y^{a, b}(t) \coloneqq X(t) + R_c^{a, b}(t) + R_p^{a, b}(t), \quad t \geq 0.
\end{equation}
Processes similar to $Y^{a, b}$ have been studied in \cite{avanzi_hybrid_2016, perez_dual_2020}, among others. For completeness, we present Algorithm \ref{Algo: construction} for constructing the replenishment processes $R_c^{a, b}$ and $R_p^{a, b}$.

\begin{commentbox}
    \caption{Construction of $R^{a, b}_c$ and $R^{a, b}_p$.}
    \label{Algo: construction}
    \commentcontent{
For any initial inventory level $X(0) = x \in \mathbb{R}$:
\begin{enumerate}
    \item [1.] Define $R(t) \coloneqq X(t) - \inf_{0 \leq s \leq t} \left(0 \wedge (X(s) - a)\right)$, the process $X$ reflected at $a$. Denote by $T^-_b \coloneqq \inf\{t \in \mathcal{T}: R(t) < b\}$. For $0 \leq t < \sigma \coloneqq T^-_b$, set $R_p^{a, b}(t) = 0$ and 
    \begin{equation*}
        R_c^{a, b}(t) = -\inf_{0 \leq s \leq t} \left(0 \wedge (X(s) - a)\right).
    \end{equation*}
    Additionally, set $\eta \coloneqq R(\sigma)$. Proceed to the next step.
    \item [2.] For $t \geq \sigma$, set $\tilde{X}(t) \coloneqq X(t) - X(\sigma)$ and $R(t) \coloneqq \tilde{X}(t) - \inf_{\sigma \leq s \leq t} \left(0 \wedge (\tilde{X}(s) + (b - a))\right)$. Denote by $\tilde{T}^{-}_b \coloneqq \inf\{t \in \mathcal{T},\, t > \sigma: R(t) < 0\}$. For $\sigma \leq t < \tilde{\sigma} \coloneqq \tilde{T}^{-}_b$, set 
    \[R_c^{a, b}(t) = R_c^{a, b}(\sigma-) - \inf_{\sigma \leq s \leq t} \left(0 \wedge (\tilde{X}(s) + (b - a))\right), \quad R_p^{a, b}(t) = R_p^{a, b}(\sigma-) + (b - \eta).\]
    Set $\sigma \coloneqq \tilde{\sigma}$ and $\eta \coloneqq b + R(\tilde{\sigma})$. Return to the beginning of this step.
\end{enumerate}
}
\end{commentbox}

\subsection{Expected costs.}
For a hybrid barrier replenishment policy $\pi^{a, b}$, let
\begin{equation}
    v_{a, b} (x) \coloneqq v^{\pi^{a,b}}(x) = v_{a, b}^{f}(x) + v_{a, b}^{r}(x), \quad x \in \mathbb{R}, \label{def_v_a_b} 
\end{equation}
where
\begin{equation*}
    v_{a, b}^{f}(x) \coloneqq \mathbb{E}_x\left[\int^\infty_0e^{-qt} f(Y^{a, b}(t))\, \diff  t\right], \quad v_{a, b}^{r}(x) \coloneqq \mathbb{E}_{x}\left[ \int_{[0, \infty)} e^{-qt} (K_c\, \diff R_c^{a, b}(t) + K_p \, \diff R_p^{a, b}(t))\right].
\end{equation*}
We will demonstrate that these expectations are finite under our assumptions, and thus the decomposition is well-defined.

To this end, we introduce the scale functions of a spectrally negative L\'evy process. For $q \geq 0$, denote by $W^{(q)}: \mathbb{R} \to [0 , \infty)$ the $q$-scale function of $X$. On the positive half-line, the scale function $W^{(q)}$ is a continuous and strictly increasing function that satisfies
\begin{equation}
    \int_0^\infty e^{-sx} W^{(q)}(x)\, \diff x = \frac{1}{\psi(s) - q}, \quad s > \Phi(q),
\end{equation}
where $\psi$ is the Laplace exponent as defined in \eqref{Eq: laplace exponent} and 
\[
\Phi(q) \coloneqq \sup \{s \geq 0: \psi(s) = q\}, \quad q \geq 0,
\]
is its right inverse. 
On the negative half-line, $W^{(q)}$ is set to be zero. Below, we introduce several auxiliary functions that will be used in subsequent computations. A detailed review of their properties and applications can be found in \cite{kuznetsov_theory_2013, kyprianou_fluctuations_2014}.

For $x \in \mathbb{R}$, define
\begin{align*}
    \overline{W}^{(q)}(x) &\coloneqq \int_0^{x} W^{(q)}(y)\, \diff y, \qquad \overline{\overline{W}}^{(q)}(x) \coloneqq \int_0^{x} \overline{W}^{(q)}(y)\, \diff y, \\
    Z^{(q)}(x) &\coloneqq 1 + q\overline{W}^{(q)}(x), \qquad  \overline{Z}^{(q)}(x) \coloneqq \int^x_0 Z^{(q)}(y)\, \diff y = x + q \overline{\overline{W}}^{(q)}(x).
\end{align*}
We note the following known limits, which will be applied in subsequent computations,
\begin{equation}\label{Eq: W, Z limits}
    \lim_{y\to\infty} \frac{W^{(q)}(y + x)}{W^{(q)}(y)} = e^{\Phi(q)x}, \qquad \lim_{y\to\infty} \frac{W^{(q)}(y)}{Z^{(q)}(y)} = \frac{\Phi(q)}{q}, \quad x \in \mathbb{R}. 
\end{equation}

We also use the \textit{second scale function}. For $x \in \mathbb{R}$ and $\lambda > 0$,
\begin{equation}\label{Eq: Z phi}
    Z^{(q + \lambda)}(x, \Phi(q)) \coloneqq e^{\Phi(q) x} \left(1 + \lambda\int_0^{x} e^{-\Phi(q) z} W^{(q + \lambda)}(z)\, \diff z \right).
\end{equation}
By differentiation, 
\begin{equation}\label{Eq: Z phi prime}
    Z^{(q + \lambda)'}(x, \Phi(q)) \coloneqq \frac{\partial}{\partial x} Z^{(q + \lambda)}(x,\Phi(q)) = \Phi(q) Z^{(q + \lambda)}(x,\Phi(q)) + \lambda W^{(q + \lambda)}(x), \quad x > 0.
\end{equation}

For $a \in \mathbb{R}$, $q > 0$, and $h: \mathbb{R} \to \mathbb{R}$, we define
\begin{equation}\label{Eq: rho}
    \rho_a^{(q)}(x; h) \coloneqq \int_a^x h(y) W^{(q)}(x - y) \, \diff y, \quad x \in \mathbb{R}.
\end{equation}
Additionally, for $b, y \in \mathbb{R}$ and $h: \mathbb{R} \to \mathbb{R}$, define the operator
\begin{equation*}
    (\mathcal{A}^{(q,\lambda, b, y)} h)(x) \coloneqq h(x - y) -\lambda\int_{b}^{x} W^{(q)}(x - z) h(z - y)\, \diff z, \quad x \in \mathbb{R},
\end{equation*}
along with
\begin{align}
\begin{split}\label{Eq: scale scr}
    &\mathscr{Z}^{(q,\lambda)}_{b}(x, y) \coloneqq (\mathcal{A}^{(q,\lambda, b, y)} Z^{(q + \lambda)})(x), \quad \overline{\mathscr{Z}}^{(q,\lambda)}_{b}(x, y) \coloneqq (\mathcal{A}^{(q,\lambda, b, y)} \overline{Z}^{(q + \lambda)})(x), \\
    &\mathscr{W}^{(q,\lambda)}_{b}(x, y) \coloneqq (\mathcal{A}^{(q,\lambda, b, y)} W^{(q + \lambda)})(x), \quad \overline{\mathscr{W}}^{(q,\lambda)}_{b}(x, y) \coloneqq (\mathcal{A}^{(q,\lambda, b, y)} \overline{W}^{(q + \lambda)})(x). 
\end{split}
\end{align}

Lastly, for $a < b$, $x \in \mathbb{R}$, and $h: \mathbb{R} \to \mathbb{R}$, we define
\begin{align}
    \mathcal{H}^{(q, \lambda)}_{b}(x, a; h) &\coloneqq \rho_b^{(q)}(x; h) + \int_a^b h(y) \mathscr{W}^{(q,\lambda)}_{b}(x, y)\, \diff y + \frac{\lambda}{q} \rho_a^{(q + \lambda)}(b; h) Z^{(q)}(x - b), \nonumber \\ 
    \mathcal{K}_{b}^{(q, \lambda)}(x, a) &\coloneqq \frac{1}{\lambda + q} \left(q\frac{\mathscr{Z}^{(q,\lambda)}_{b}(x, a)}{Z^{(q + \lambda)}(b - a)} + \lambda Z^{(q)}(x - b)\right),\label{Eq: script K}\\ 
    \Theta^{(q + \lambda)}(x, \Phi(q)) &\coloneqq 1 + \lambda \int_0^{x} e^{-\Phi(q)y} W^{(q + \lambda)}(y) \, \diff y. \notag
\end{align}

Lemmas \ref{Lemma: inventory cost} and \ref{Lemma: control costs} provide the analytical expressions for inventory-holding and replenishment costs under $\pi^{a, b}$. The proofs are presented in Appendices \ref{Appendix: proof of inventory cost} and \ref{Appendix: proof of control cost}, respectively.
\begin{lemma}[inventory-holding cost] \label{Lemma: inventory cost}
    For $a < b$ and $x \in \mathbb{R}$,
    \begin{equation}\label{Eq: inventory cost}
        v_{a, b}^{f}(x) = -\mathcal{H}^{(q, \lambda)}_{b}(x, a; f) + \mathcal{G}^f(b) \mathcal{K}_{b}^{(q, \lambda)}(x, a),
    \end{equation}
    where
    \begin{multline}\label{Eq: v^h(b)}
        \mathcal{G}^f(b) 
        \coloneqq \left[\frac{\Phi(q)}{q}\left( \int_0^\infty e^{-\Phi(q)y} f(y + a) \, \diff y + \lambda \int_0^{b - a} e^{-\Phi(q)y}\rho_a^{(q + \lambda)}(y + a; f)\, \diff y\right) + \frac{\lambda}{q} e^{-\Phi(q)(b - a)} \rho_a^{(q + \lambda)}(b; f) \right]\\
        \times \frac{Z^{(q + \lambda)}(b - a)}{\Theta^{(q + \lambda)}(b - a, \Phi(q))}, 
    \end{multline}
    which can also be written as
    \begin{align} \label{Eq: v^f(b)}
        \mathcal{G}^f(b) = \left(\frac{ \int_0^\infty e^{-\Phi(q)y} f'(y + a) \diff y + \lambda \int_0^{b - a} e^{-\Phi(q)y}\rho_a^{(q + \lambda)}(y + a; f') \, \diff y}{\Theta^{(q + \lambda)}(b - a, \Phi(q))} + f(a)\right) \frac {Z^{(q + \lambda)}(b - a)} q.
    \end{align}
\end{lemma}

\begin{lemma}[replenishment cost]\label{Lemma: control costs} For $a < b$ and $x \in \mathbb{R}$,
    \begin{align} \label{Eq: control costs}
    \begin{split}
        v^{r}_{a, b}(x) &= \mathcal{K}_{b}^{(q, \lambda)}(x, a)(K_p k_1 + K_c k_2) - \frac{\lambda}{\lambda + q} K_p \overline{Z}^{(q)}(x - b)\\
        &+ K_c \psi'(0+) \left(\overline{\mathscr{W}}^{(q,\lambda)}_{b}(x, a) + \frac{\lambda}{q} Z^{(q)}(x - b) \overline{W}^{(q + \lambda)}(b - a) \right)\\
        &+ \left(\frac{\lambda}{\lambda + q} K_p - K_c\right) \left(\overline{\mathscr{Z}}^{(q,\lambda)}_{b}(x, a) + \frac{\lambda}{q} Z^{(q)}(x - b) \overline{Z}^{(q + \lambda)}(b - a)\right),
        \end{split}
    \end{align}
    where 
    \begin{align}
        k_1 &\coloneqq \left(\frac{\frac{1}{\lambda + q}\left(1 + \frac{\lambda}{q} Z^{(q + \lambda)}(b - a)\right) e^{-\Phi(q)(b - a)}}{\Theta^{(q + \lambda)}(b - a, \Phi(q))} - \frac{1}{q}  \right) \lambda \frac {Z^{(q + \lambda)}(b - a)} {\Phi(q)}, \label{Eq: k_1}\\
        k_2 &\coloneqq \left(- \frac{\lambda Z^{(q + \lambda)}(b - a) e^{-\Phi(q)(b - a)}}{\Theta^{(q + \lambda)}(b - a, \Phi(q))} + \lambda + q - \Phi(q) \psi'(0+) \right) \frac {Z^{(q + \lambda)}(b - a)} {q\Phi(q)}. \label{Eq: k_2}
    \end{align}
\end{lemma}

For $a < b$, further define
\begin{align}\label{Eq: Gamma}
    \Gamma(a, b) &\coloneqq \int_0^\infty e^{-\Phi(q)y} \tilde{f'}(y + a) \, \diff y + \lambda \int_0^{b - a} e^{-\Phi(q)y}\rho_a^{(q + \lambda)}(y + a; \tilde{f}') \, \diff y + \frac{\lambda}{\Phi(q)} e^{-\Phi(q)(b - a)} \left(K_p - K_c\right).
\end{align}
By combining \eqref{Eq: inventory cost} and \eqref{Eq: control costs}, we establish the total cost corresponding to $\pi^{a, b}$. The proof of Lemma \ref{Lemma: NPV of costs} is presented in Appendix \ref{Appendix: proof of NPV of costs}. 
\begin{lemma}\label{Lemma: NPV of costs}
    For $a < b$ and $x \in \mathbb{R}$,
    \begin{align}\label{Eq: NPV of costs}
    \begin{split}
        v_{a, b}(x) &= -  \mathcal{H}^{(q, \lambda)}_{b}(x, a; f) + \frac{1}{q} \left(\frac{ \Gamma(a, b)}{\Theta^{(q + \lambda)}(b - a, \Phi(q))} + f(a) - K_c\psi'(0+) \right)  \mathcal{K}_{b}^{(q, \lambda)}(x, a) Z^{(q + \lambda)}(b - a) \\
        &- \frac{\lambda}{\lambda + q} K_p \overline{Z}^{(q)}(x - b) + K_c \psi'(0+) \left(\overline{\mathscr{W}}^{(q,\lambda)}_{b}(x, a) + \frac{\lambda}{q} Z^{(q)}(x - b) \overline{W}^{(q + \lambda)}(b - a) \right)\\
        &+ \left(\frac{\lambda}{\lambda + q} K_p - K_c\right) \left(\overline{\mathscr{Z}}^{(q,\lambda)}_{b}(x, a) + \frac{\lambda}{q} Z^{(q)}(x - b) \overline{Z}^{(q + \lambda)}(b - a)\right).
        \end{split}
    \end{align}
\end{lemma}

\section{Selection of barriers}\label{Sect: candidate barriers}
In this section, we outline the approach used to select the candidate barrier values, denoted by $(a^*, b^*)$. We begin by analyzing smoothness properties of the total cost  $v_{a,b}$ as in \eqref{def_v_a_b}, which play a critical role in establishing the optimality of a policy.

\subsection{Smoothness}
The following notion of smoothness is particularly useful for L\'evy-based models:
\begin{definition}\label{def: sufficiently smooth}
    A function is \emph{sufficiently smooth} if it is continuously differentiable on $\mathbb{R}$ when $X$ is of bounded variation and twice continuously differentiable on $\mathbb{R}$ when $X$ is of unbounded variation.
\end{definition}

The following lemma shows that $x \mapsto v_{a, b}(x)$ is sufficiently smooth everywhere on $\mathbb{R}$, except possibly at the lower barrier $a$. The proof is provided in Appendix \ref{Appendix: proof of NPV derivatives}.
\begin{lemma}\label{Lemma: NPV derivatives}
    For $a < b$ and $x \in \mathbb{R}\backslash \{a\}$, 
    \begin{multline}\label{Eq: NPV derivative}
        v'_{a, b}(x) = -\rho_b^{(q)}(x; \tilde{f}') - \int_a^b \tilde{f}'(y) \mathscr{W}^{(q,\lambda)}_{b}(x, y)\, \diff y + \frac{\Gamma(a, b) \mathscr{W}^{(q,\lambda)}_{b}(x, a)}{\Theta^{(q + \lambda)}(b - a, \Phi(q))} + \lambda (K_c - K_p) \overline{W}^{(q)}(x - b) - K_c,
    \end{multline}
    which can also be written as 
    \begin{equation}\label{Eq: NPV derivative 2}
        v'_{a, b}(x) = -\mathcal{H}_{b}^{(q, \lambda)}(x, a; \tilde{f}') + \left(\frac{\lambda}{\lambda + q} K_p - K_c\right)\frac{\lambda + q}{q} - \frac{\lambda}{q} \gamma(a, b) Z^{(q)}(x - b) + \frac{\Gamma(a, b) \mathscr{W}^{(q,\lambda)}_{b}(x, a)}{\Theta^{(q + \lambda)}(b - a, \Phi(q))}.
    \end{equation}
    If $X$ is of unbounded variation, for $x \in \mathbb{R}\backslash \{a\}$, 
    \begin{multline}\label{Eq: NPV 2nd derivative}
        v''_{a, b}(x) = -\int_b^x \tilde{f}'(y) W^{(q)'}(x - y) \, \diff y - \int_a^b \tilde{f}'(y) \frac{\diff}{\diff x}\mathscr{W}^{(q,\lambda)}_{b}(x, y)\, \diff y \\
        + \frac{\Gamma(a, b) \frac{\diff}{\diff x} \mathscr{W}^{(q,\lambda)}_{b}(x, a)}{\Theta^{(q + \lambda)}(b - a, \Phi(q))} + \lambda (K_c - K_p) W^{(q)}(x - b),
    \end{multline}
    where, when $X$ is of unbounded variation,
    \begin{equation*}
        \frac{\diff}{\diff x}\mathscr{W}^{(q,\lambda)}_{b}(x, y) = W^{(q + \lambda)'}(x - y) -\lambda\int_{b}^{x} W^{(q)'}(x - z)W^{(q + \lambda)}(z - y) \, \diff z, \quad x, y \in \mathbb{R}. 
    \end{equation*}
\end{lemma}

\subsection{Candidate barriers.}
We now specify the candidate barriers. Motivated by studies of similar stochastic control problems (see, e.g., \cite{perez_dual_2020}), we seek barriers $a^* < b^*$ that satisfy the following pair of equations:
\begin{equation*}
    v_{a^*, b^*}^{f'}(a^*) = -K_c, \quad 
    v_{a^*, b^*}^{f'}(b^*) = -K_p,
\end{equation*}
where 
\begin{equation*}
    v_{a, b}^{f'}(x) \coloneqq  \mathbb{E}_{x}\left[\int^\infty_0e^{-qt} f'(Y^{a, b}(t))\, \diff t\right], \quad a < b,\, x\in \mathbb{R}.
\end{equation*}
To this end, we compute $v_{a, b}^{f'}$. By differentiating $\Gamma$ as defined in \eqref{Eq: Gamma} with respect to the second argument, we obtain
\begin{equation}\label{Eq: partial Gamma}
    \frac{\partial }{\partial b} \Gamma(a, b) = \lambda e^{-\Phi(q)(b - a)} \left( \rho_a^{(q + \lambda)}(b; \tilde{f}') + \left(K_c - K_p\right) \right), \quad b > a. 
\end{equation}
For $b > a$, define
\begin{align} \label{Eq: gamma}
    \gamma(a, b) &\coloneqq -\frac{1}{\lambda} e^{\Phi(q)(b - a)}\frac{\partial}{\partial b} \Gamma(a, b) = -\rho_a^{(q + \lambda)}(b; \tilde{f}') + K_p - K_c. 
\end{align}

\begin{lemma}\label{Lemma: auxiliary result}
    For $a < b$ and $x \in \mathbb{R}$, we have
    \begin{equation}\label{Eq: v^f'}
        v_{a, b}^{f'}(x) = -\mathcal{H}_{b}^{(q, \lambda)}(x, a; \tilde{f}') + \left(\frac{\lambda}{\lambda + q} K_p - K_c\right)\frac{\lambda + q}{q} + \frac{\frac{\Phi(q)}{q} \Gamma(a, b) - \frac{\lambda}{q} e^{-\Phi(q)(b - a)} \gamma(a, b)}{\Theta^{(q + \lambda)}(b - a, \Phi(q))} Z^{(q + \lambda)}(b - a)\mathcal{K}_{b}^{(q, \lambda)}(x, a).
    \end{equation}
    In particular, 
    \begin{multline}\label{Eq: v^f' at a}
        v_{a, b}^{f'}(a) =  \left(\frac{\lambda}{q (\lambda + q)} \frac{\Phi(q) \Gamma(a, b) - \lambda e^{-\Phi(q)(b - a)} \gamma(a, b)}{\Theta^{(q + \lambda)}(b - a, \Phi(q))} Z^{(q + \lambda)}(b - a) + \frac{\lambda}{q} \gamma(a, b)\right)\\
        - K_c + \frac 1 {\lambda+q}\left(\frac{\Phi(q) \Gamma(a, b) - \lambda e^{-\Phi(q)(b - a)} \gamma(a, b)}{\Theta^{(q + \lambda)}(b - a, \Phi(q))}\right)
    \end{multline}
    and
    \begin{equation}\label{Eq: v^f' at b}
        v^{f'}_{a, b}(b) = \frac{\Phi(q) \Gamma(a, b) - \lambda e^{-\Phi(q)(b - a)} \gamma(a, b)}{\Theta^{(q + \lambda)}(b - a, \Phi(q))} \frac{Z^{(q + \lambda)}(b - a)}{q} + \frac{\lambda + q}{q} \gamma(a, b) - K_p.
    \end{equation}
\end{lemma}
The proof of Lemma \ref{Lemma: auxiliary result} is deferred to Appendix \ref{Appendix: proof of auxiliary result}. The next result is a direct consequence of this lemma.
\begin{proposition} \label{Prop: iff conditions} For $a < b$, the following statements are equivalent:  
\begin{itemize}
    \item[(1)] $v^{f'}_{a, b}(a) = -K_c$ and $v^{f'}_{a, b}(b) = -K_p$,
    \item[(2)] $\Gamma(a, b) = 0$ and $\gamma(a,b) = 0$,
    \item[(3)] $x \mapsto v_{a, b}(x)$ is sufficiently smooth, with $v'_{a, b}(a) = -K_c$ and $v'_{a, b}(b) = -K_p$.
\end{itemize}
\end{proposition}
\begin{proof}
Proof of (1) $\Longleftrightarrow$ (2): Suppose that (1) holds. From \eqref{Eq: v^f' at b}, we have 
\begin{equation}\label{Eq: iff misc 1}
    \frac{\Phi(q) \Gamma(a, b) - \lambda e^{-\Phi(q)(b - a)} \gamma(a, b)}{\Theta^{(q + \lambda)}(b - a, \Phi(q))} \frac{Z^{(q + \lambda)}(b - a)}{q} + \frac{\lambda + q}{q} \gamma(a, b) = 0.
\end{equation}
    Substituting \eqref{Eq: iff misc 1} into \eqref{Eq: v^f' at a} and using $v^{f'}_{a, b}(a) = -K_c$, we have
    \begin{equation}\label{Eq: iff misc 2}
       \frac{\Phi(q) \Gamma(a, b) - \lambda e^{-\Phi(q)(b - a)} \gamma(a, b)}{\Theta^{(q + \lambda)}(b - a, \Phi(q))} = 0.
    \end{equation}
    Substituting \eqref{Eq: iff misc 2} back into \eqref{Eq: iff misc 1}, we obtain $\gamma(a, b) = 0$. It remains to see from \eqref{Eq: iff misc 2} that if $\gamma(a, b) = 0$, then it must hold that $\Gamma(a, b) = 0$. Thus, (1) implies (2). The converse direction holds by substituting $\Gamma(a, b) = \gamma(a, b) = 0$ into \eqref{Eq: v^f' at a} and \eqref{Eq: v^f' at b}.
    
Proof of (2) $\Longleftrightarrow$ (3):
  Recall from Lemma \ref{Lemma: NPV derivatives} that $v_{a, b}$ is sufficiently smooth on $\mathbb{R}$ except possibly at $a$. Moreover, from \eqref{Eq: NPV derivative} with $x = a-$, we have $v'_{a, b}(a-) = -K_c$. Therefore, if $v'_{a, b}$ is continuous at $a$ then necessarily $v'_{a, b}(a) = -K_c$. Hence, it suffices to show that $v_{a, b}$ is sufficiently smooth at $a$ and that $v'_{a, b}(b) = -K_p$ together are equivalent to (2). 
    
In view of the behavior of $x \mapsto W^{(q)}(x)$ and $x \mapsto W^{(q)'}(x)$ at zero (see Lemmas 3.1 and 3.2 of \cite{kuznetsov_theory_2013}), $v_{a, b}$ is sufficiently smooth at $a$ if and only if $\Gamma(a, b) = 0$. Moreover, by \eqref{Eq: NPV derivative} with $x = b$ and \eqref{Eq: gamma}, we have
    \begin{align*}
        v_{a, b}'(b) &= -K_p + \gamma(a, b) + W^{(q + \lambda)}(b - a) \frac{\Gamma(a, b)}{Z^{(q + \lambda)}(b - a)}.
    \end{align*}
Thus, when $\Gamma(a, b) = 0$, $v_{a, b}'(b) = -K_p$ if and only if $\gamma(a, b) = 0$.
\end{proof}

In view of Proposition \ref{Prop: iff conditions}, we seek a pair $(a^*, b^*)$ satisfying the condition $\mathfrak{C}$: 
\begin{equation}\label{Eq: condition}
    \mathfrak{C}: \Gamma(a^*, b^*) = \gamma(a^*, b^*) = 0.
\end{equation}
In Sections \ref{Sect: existence} and \ref{Sect: optimality}, we establish the existence of $(a^*, b^*)$ and the optimality of $\pi^{a^*, b^*}$.

\section{Existence of candidate barriers}\label{Sect: existence}
In this section, we show that a pair of candidate barriers satisfying \eqref{Eq: condition} exists. In view of the expression of $\gamma$ in terms of the partial derivative of $\Gamma$ in \eqref{Eq: gamma}, we aim to find $(a^*,b^*)$ such that the curve $b \mapsto \Gamma(a^*, b)$ becomes tangent to the x-axis at $b^*$ (see Figure \ref{Fig: Gamma}).

As a first step towards proving the existence, we establish the starting and asymptotic values of $b \mapsto \Gamma(a, b)$. Recall from Assumption \ref{asm: f slope} that $-\infty < \bar{a} \coloneqq \inf\{a \in \mathbb{R}: \tilde{f}'(a) \geq 0\} < \bar{\bar{a}} \coloneqq \inf\{a \in \mathbb{R}: f'(a) + qK_p > 0\} < \infty$. The proof of Lemma \ref{Lemma: Gamma asymptotics} is presented in Appendix \ref{Appendix: proof of gamma asymptotics}. 
\begin{lemma} \label{Lemma: Gamma asymptotics} For $a \in \mathbb{R}$,
\begin{align}
    \Gamma_1(a) &\coloneqq \Gamma(a, a+) = \int_0^\infty e^{-\Phi(q)y} \tilde{f'}(y + a) \diff y + \frac{\lambda}{\Phi(q)} \left(K_p - K_c\right), \label{Eq: Gamma 1} \\
    \Gamma_2(a) &\coloneqq \lim_{b \to \infty} \frac{\Gamma(a, b)}{\Theta^{(q + \lambda)}(b - a, \Phi(q))} = \int_0^\infty e^{-\Phi(q + \lambda)y} \tilde{f}'(y + a)\, \diff y. \label{Eq: Gamma 2}
\end{align}
    Moreover, both $\Gamma_1$ and $\Gamma_2$ are continuous and strictly increasing on $(-\infty, \bar{\bar{a}}]$.
\end{lemma}

Define the left inverses of $\Gamma_1$ and $\Gamma_2$, 
\begin{align}\label{Eq: a_underline}
\begin{split}
    \underline{a}_i \coloneqq \inf\{a \in \mathbb{R}: \Gamma_i(a) \geq 0\}, \quad i = 1,2,
\end{split}
\end{align}
both of which exist under Assumption \ref{asm: f slope}. The following lemma establishes the upper bounds for $\underline{a}_1$ and $\underline{a}_2$; the proof is provided in Appendix \ref{Appendix: proof of bounds for a underline}. 
\begin{lemma}\label{Lemma: bounds for a underline}
    It holds that
    \[
    \underline{a}_1 \in (-\infty, \bar{\bar{a}}) \quad \textrm{and} \quad \underline{a}_2 \in (-\infty, \bar{a}).
    \]
\end{lemma}

In view of the bounds on $\underline{a}_1$ and $\underline{a}_2$ and the strict monotonicity of $\Gamma_1$ and $\Gamma_2$ on $(-\infty, \bar{\bar{a}}]$, we have
\begin{equation}
    \Gamma(a, a+) > 0 \iff a > \underline{a}_1 \label{Gamma_b_start}
\end{equation}
and
\begin{align}
    \lim_{b \to \infty} \Gamma(a, b) = \infty \quad \text{(resp. $-\infty$)} \quad  \text{if $a > \underline{a}_2$} \quad \text{(resp. $a < \underline{a}_2$)}. \label{Gamma_b_asymp}
\end{align}

Denote the supremum of $\Gamma$ with respect to the second argument by
\[
\overline{\Gamma}(a) \coloneqq \sup_{b > a} \Gamma(a, b), \quad a \in \mathbb{R}.
\] 
The function $a \mapsto \overline{\Gamma}(a)$ can be shown to be continuous and strictly monotone on $(-\infty, \underline{a}_2)$. First, we establish the following lemma; its proof is deferred to Appendix \ref{Appendix: proof of rho monotone}.

\begin{lemma}\label{Lemma: rho monotone}
    For $a < \underline{a}_2$, the function $b \mapsto \rho_{a}^{(q + \lambda)}(b; \tilde{f}')$ is strictly decreasing and negative on $(a, \infty)$. Moreover, the function $b \mapsto \rho_{\underline{a}_2}^{(q + \lambda)}(b; \tilde{f}')$ is non-increasing and non-positive on $(\underline{a}_2, \infty)$.
\end{lemma}

Combining Lemma \ref{Lemma: rho monotone} with \eqref{Eq: partial Gamma} and \eqref{Gamma_b_asymp} yields the following corollary.
\begin{corollary}\label{Corollary: critical point}
    For $a' < \underline{a}_2$, the maximum of $b \mapsto \Gamma(a', b)$ is attained at exactly one $b' > a'$ with $\frac{\partial }{\partial b} \Gamma(a', b') = 0$. In particular, it holds that  
    \begin{equation}\label{Eq: uniqueness misc 1}
        \rho_{a'}^{(q + \lambda)}(b'; \tilde{f}') = K_p - K_c \quad \text{and} \quad \gamma(a', b') = -\rho_{a'}^{(q + \lambda)}(b'; \tilde{f}') + K_p - K_c = 0.
    \end{equation}
\end{corollary}

Combining Lemma \ref{Lemma: rho monotone} with Corollary \ref{Corollary: critical point} gives the following proposition.
\begin{proposition} \label{Prop: monotonicity}
    The function $a \mapsto \overline{\Gamma}(a)$ is continuous and strictly increasing on $(-\infty, \underline{a}_2)$. 
\end{proposition}
\begin{proof}
    We will show that the function is strictly increasing, and its continuity will become apparent in the process. For $(a', b')$ and $(a'', b'')$ satisfying $a'' < a' < \underline{a}_2$, $\overline{\Gamma}(a') = \Gamma(a', b')$, and $\overline{\Gamma}(a'') = \Gamma(a'', b'')$, it suffices to show that $\Gamma(a'', b'') < \Gamma(a', b')$. Rewriting $\Gamma$ as given in \eqref{Eq: Gamma}, we obtain
    \begin{equation}\label{Eq: Gamma alt form}
        \Gamma(a, b) = \int_0^\infty e^{-\Phi(q)y} \tilde{f'}(y + a) \diff y + \lambda \int_0^{b - a} e^{-\Phi(q)y} (\rho_a^{(q + \lambda)}(y + a; \tilde{f}') + K_c - K_p)\, \diff y - \frac{\lambda}{\Phi(q)} (K_c - K_p). 
    \end{equation}
    For $y > 0$, by the convexity of $\tilde{f}$ and the fact that $a'' < a'$,
    \begin{equation*}
        \rho_{a''}^{(q + \lambda)}(y + a''; \tilde{f}') = \int_{0}^{y} \tilde{f}'(z + a'') W^{(q + \lambda)}(y - z) \, \diff z \leq \int_{0}^{y} \tilde{f}'(z + a') W^{(q + \lambda)}(y - z) \, \diff z = \rho_{a'}^{(q + \lambda)}(y + a'; \tilde{f}') \leq 0.
    \end{equation*}
    Hence, $y \mapsto \rho_{a''}^{(q + \lambda)}(y + a''; \tilde{f}')$ lower bounds $y \mapsto \rho_{a'}^{(q + \lambda)}(y + a'; \tilde{f}')$ on $(0, \infty)$. This observation, together with the fact that $\rho_{a'}^{(q + \lambda)}(b'; \tilde{f}') = \rho_{a''}^{(q + \lambda)}(b''; \tilde{f}') = K_p - K_c < 0$ as shown in \eqref{Eq: uniqueness misc 1}, yields $b'' - a'' \leq b' - a'$. Moreover, since
    \begin{equation*}
        0 \leq \rho_{a''}^{(q + \lambda)}(y + a''; \tilde{f}') + K_c - K_p \leq \rho_{a'}^{(q + \lambda)}(y + a'; \tilde{f}') + K_c - K_p, \quad 0 < y < b'' - a'',
    \end{equation*}
    and 
    \begin{equation*}
        0 \leq \rho_{a'}^{(q + \lambda)}(y + a'; \tilde{f}') + K_c - K_p, \quad b'' - a'' < y < b' - a',
    \end{equation*}
    we obtain
    \begin{equation*} 
        \lambda\int_0^{b'' - a''} e^{-\Phi(q)y} (\rho_{a''}^{(q + \lambda)}(y + a''; \tilde{f}') + K_c - K_p)\, \diff y \leq \lambda\int_0^{b' - a'} e^{-\Phi(q)y} (\rho_{a'}^{(q + \lambda)}(y + a'; \tilde{f}') + K_c - K_p)\, \diff y. 
    \end{equation*}
    Finally, by the convexity of $\tilde{f}$ and the existence of $\bar{a} > \underline{a}_2$ (see Lemma \ref{Lemma: bounds for a underline}) such that $\tilde{f}'(\bar{a}) > \tilde{f}'(a') \geq \tilde{f}'(a'')$, the following inequality holds strictly, $\int_0^\infty e^{-\Phi(q)y} \tilde{f'}(y + a'') \, \diff y < \int_0^\infty e^{-\Phi(q)y} \tilde{f'}(y + a') \, \diff y$. Substituting these inequalities into \eqref{Eq: Gamma alt form}, we obtain $\Gamma(a'', b'') < \Gamma(a', b')$.  
\end{proof}

The rest of this section is devoted to proving the following theorem.
\begin{theorem} \label{Thm: existence uniqueness}
    There exists a pair $(a^*, b^*)$ satisfying $\mathfrak{C}$.
\end{theorem}

In view of Corollary \ref{Corollary: critical point}, it suffices to find some $a^* < \underline{a}_2$ such that $\overline{\Gamma}(a^*) = 0$. With $a\mapsto\overline{\Gamma}(a)$ strictly increasing on $(-\infty,\underline{a}_2)$ (Proposition \ref{Prop: monotonicity}), $a^*$, if it exists, is found by starting at $a=\underline{a}_2$ and decreasing $a$. See Figure \ref{Fig: Gamma} for an illustration of this procedure.

We now prove the existence of $a^*$ for the case $\underline{a}_1 \leq \underline{a}_2$, and subsequently for the case $\underline{a}_1 > \underline{a}_2$.

\subsection{Proof of existence for the case $\underline{a}_1 \leq \underline{a}_2$.} \label{Sect: case 1}
By the definition of $\underline{a}_1$ as in \eqref{Eq: a_underline}, it holds that $\Gamma(\underline{a}_1, \underline{a}_1+) = 0$. By \eqref{Eq: partial Gamma}, we have $\frac{\partial }{\partial b} \Gamma(\underline{a}_1, \underline{a}_1+) = \lambda(K_c - K_p) > 0$. Hence, $\overline{\Gamma}(\underline{a}_1) > 0$ and $\overline{\Gamma}(\underline{a}_1-) > 0$. Indeed, if $\underline{a}_1 < \underline{a}_2$, then $\overline{\Gamma}(\underline{a}_1-) > 0$ follows from $\overline{\Gamma}(\underline{a}_1) > 0$ together with Proposition \ref{Prop: monotonicity}. If $\underline{a}_1 = \underline{a}_2$, the claim $\overline{\Gamma}(\underline{a}_1-) > 0$ follows from $\overline{\Gamma}(\underline{a}_1) > 0$ and the continuity of $a \mapsto \Gamma(a, b)$ on $(-\infty, b)$ for every $b \in \mathbb{R}$.

Define $a^\dagger$ as follows, whose existence is guaranteed by Assumption \ref{asm: f slope}(1):
\begin{equation}\label{Eq: a tilde}
     a^\dagger \coloneqq \inf \left\{a\in \mathbb{R}: \int_0^\infty e^{-\Phi(q)y} \tilde{f'}(y + a) \, \diff y \geq 0\right\} \in \mathbb{R}.
\end{equation}
Substituting $\int_0^\infty e^{-\Phi(q)y} \tilde{f}'(y + a^\dagger) \, \diff y = 0$ into \eqref{Eq: Gamma}, for $b > a^\dagger$,
\begin{equation}
    \Gamma(a^\dagger, b) =  \lambda \int_0^{b - a^\dagger} e^{-\Phi(q)y} \rho_{a^\dagger}^{(q + \lambda)}(y + a^\dagger; \tilde{f}') \, \diff y - \frac{\lambda}{\Phi(q)} e^{-\Phi(q)(b - a^\dagger)} \left(K_c - K_p\right).  \label{Gamma_tilde_a}
\end{equation}
Since $K_p < K_c$, we have $a^\dagger < \underline{a}_1 (\leq \underline{a}_2$). This observation, combined with \eqref{Gamma_tilde_a} and Lemma \ref{Lemma: rho monotone}, implies that $\Gamma(a^\dagger, b) < 0$ for all $b > a^\dagger$. Moreover, Equation \eqref{Gamma_b_asymp} yields $\lim_{b \to \infty}\Gamma(a^\dagger, b) = - \infty$. Hence, we conclude that
$\overline{\Gamma}(a^\dagger) < 0$.

Given $\overline{\Gamma}(\underline{a}_1-) > 0$ and $\overline{\Gamma}(a^\dagger) < 0$, the continuity of $\overline{\Gamma}$ on $[a^\dagger, \underline{a}_1) \subset (-\infty, \underline{a}_2)$ guarantees the existence of some $a^\dagger < a^* < \underline{a}_1$ such that the function $b \mapsto \Gamma(a^*, b)$ starts at $\Gamma(a^*, a^*+) < 0$ and maximizes at $\Gamma(a^*, b^*) = 0$. This implies that both $\Gamma(a^*, b^*)$ and $\gamma(a^*, b^*)$ are zero.

\subsection{Proof of existence for the case $\underline{a}_1 > \underline{a}_2$}\label{Sect: case 2}
Now, consider $\underline{a}_2 < \underline{a}_1$. We have $\Gamma(\underline{a}_2, \underline{a}_2+) < 0$ by \eqref{Gamma_b_start}. This implies that there are two potential cases to investigate: either $b \mapsto \Gamma(\underline{a}_2, b)$ upcrosses zero at some $b > \underline{a}_2$, or $b \mapsto \Gamma(\underline{a}_2, b)$ is uniformly non-positive. We show that the latter case cannot happen.

In view of the non-positivity and monotonicity of $b \mapsto \rho_{\underline{a}_2}^{(q + \lambda)}(b; \tilde{f}')$ (Lemma \ref{Lemma: rho monotone}), we have
\begin{equation}\label{Eq: rho limit}
    \rho_{\underline{a}_2}^{(q + \lambda)}(\infty; \tilde{f}') \coloneqq \lim_{b \to \infty} \rho_{\underline{a}_2}^{(q + \lambda)}(b; \tilde{f}') \in [-\infty, 0].
\end{equation} 
Moreover, in view of \eqref{Eq: Gamma}, the non-positivity and monotonicity of $b \mapsto \rho_{\underline{a}_2}^{(q + \lambda)}(b; \tilde{f}')$ further imply
\begin{equation}\label{Eq: Gamma limit}
    \Gamma(\underline{a}_2, \infty) \coloneqq \lim_{b \to \infty} \Gamma(\underline{a}_2, b) \in [-\infty, \infty).
\end{equation} 

To show $b \mapsto \Gamma(\underline{a}_2, b)$ upcrosses zero, we consider the following three (exhaustive) cases:
    \begin{enumerate}
        \item [\textbf{(1)}] $\rho_{\underline{a}_2}^{(q + \lambda)}(\infty; \tilde{f}') > -\infty$. 
        \item [\textbf{(2)}] $\rho_{\underline{a}_2}^{(q + \lambda)}(\infty; \tilde{f}') = -\infty$ and $\Gamma(\underline{a}_2, \infty) > -\infty$,
        \item [\textbf{(3)}] $\rho_{\underline{a}_2}^{(q + \lambda)}(\infty; \tilde{f}') = -\infty$ and $\Gamma(\underline{a}_2, \infty) = -\infty$.
    \end{enumerate}

We show in Lemma \ref{lemma_upcross} that $b \mapsto \Gamma(\underline{a}_2, b)$ upcrosses zero in cases \textbf{(1)} and \textbf{(2)}, and Lemma \ref{Lemma: case 3} establishes that case \textbf{(3)} cannot occur. To facilitate the analysis, we observe that by Assumption \ref{asm: f slope}(2) and monotone convergence,
    \begin{align*}
       -K_p < \frac{f'(\infty)}{q} = \lim_{b \to \infty} \mathbb{E}_{b}\left[\int^\infty_0e^{-qt} f'(Y^{\underline{a}_2, b}(t))\, \diff t\right]
       = \lim_{b \to \infty}  v^{f'}_{\underline{a}_2,b}(b).    \end{align*}
    Substituting in \eqref{Eq: v^f' at b} with $a = \underline{a}_2$, it holds that 
    \begin{equation}\label{Eq: second case misc 0}
        0 < \lim_{b \to \infty} v^{f'}_{\underline{a}_2, b}(b) + K_p = \lim_{b \to \infty} \left( \frac{\frac{\Phi(q)}{q} \Gamma(\underline{a}_2, b) - \frac{\lambda}{q} e^{-\Phi(q)(b - \underline{a}_2)} \gamma(\underline{a}_2, b)}{
        \Theta^{(q + \lambda)}(b - \underline{a}_2, \Phi(q))} Z^{(q + \lambda)}(b - \underline{a}_2) + \frac{\lambda + q}{q} \gamma(\underline{a}_2, b) \right). 
    \end{equation}
     
\begin{lemma} \label{lemma_upcross}
    In cases \textbf{(1)} and \textbf{(2)}, $\overline{\Gamma}(\underline{a}_2-) > 0$.
\end{lemma}
\begin{proof}
Suppose that case \textbf{(1)} or \textbf{(2)} holds. Using \eqref{Eq: Gamma} and because $\int_0^\infty e^{-\Phi(q) y} \tilde{f}'(y + \underline{a}_2) \, \diff y \in \mathbb{R}$, we have
    \begin{equation*}
        \Gamma(\underline{a}_2, \infty) > -\infty,  \quad \text{or equivalently,}  \quad \int_0^{\infty} e^{-\Phi(q)y} \rho_{\underline{a}_2}^{(q + \lambda)}(y + \underline{a}_2; \tilde{f}')\, \diff y > -\infty.
    \end{equation*}
    The latter is satisfied in case \textbf{(1)} and the former is satisfied in case \textbf{(2)}. Hence, for both cases, both sides of the implication hold. This observation, together with the fact that  $\rho_{\underline{a}_2}^{(q + \lambda)}(\infty; \tilde{f}') \leq 0$ as given in \eqref{Eq: rho limit}, implies that 
\begin{equation}    
    e^{-\Phi(q)(b - \underline{a}_2)} \rho_{\underline{a}_2}^{(q + \lambda)}(b; \tilde{f}') \xrightarrow{b \uparrow \infty} 0 \label{rho_asymp}
\end{equation}
in both cases. Furthermore, for any $a \in \mathbb{R}$, applying L'H\^opital's rule together with the limits in \eqref{Eq: W, Z limits} yields
\begin{multline}\label{Eq: second case misc 2}
        \lim_{b \to \infty} \frac{e^{-\Phi(q)(b - a)} Z^{(q + \lambda)}(b - a)}{\Theta^{(q + \lambda)}(b - a, \Phi(q))} = \lim_{b \to \infty} \frac{-\Phi(q) e^{-\Phi(q)(b - a)} Z^{(q + \lambda)}(b - a)}{\lambda e^{-\Phi(q)(b - a)} W^{(q + \lambda)}(b - a)}\\
        + \lim_{b \to \infty} \frac{e^{-\Phi(q)(b - a)} (q + \lambda) W^{(q + \lambda)}(b - a)}{\lambda e^{-\Phi(q)(b - a)} W^{(q + \lambda)}(b - a)} = \frac{\lambda + q}{\lambda} \frac{\Phi(q + \lambda) - \Phi(q)}{\Phi(q + \lambda)} > 0.
    \end{multline}

We now show that $\overline{\Gamma}(\underline{a}_2) > 0$. By the continuity of $a \mapsto \Gamma(a, b)$ for $b \in \mathbb{R}$ and $a < b$, it then follows that $\overline{\Gamma}(\underline{a}_2-) > 0$.

(i) Suppose that $\Gamma(\underline{a}_2, \infty) \neq 0$. Equation \eqref{Eq: gamma} together with \eqref{rho_asymp} implies that $e^{-\Phi(q)(b - \underline{a}_2)} \gamma(\underline{a}_2, b) \xrightarrow{b \uparrow \infty} 0$. Combining this observation with \eqref{Eq: second case misc 2} implies that for all sufficiently large $b$,
\begin{equation*}
   \mathrm{sgn} \left( \frac{\frac{\Phi(q)}{q} \Gamma(\underline{a}_2, b) - \frac{\lambda}{q} e^{-\Phi(q)(b - \underline{a}_2)} \gamma(\underline{a}_2, b)}{\Theta^{(q + \lambda)}(b - \underline{a}_2, \Phi(q))} Z^{(q + \lambda)}(b - \underline{a}_2) + \frac{\lambda + q}{q} \gamma(\underline{a}_2, b) \right) = \mathrm{sgn} (\Gamma(\underline{a}_2, b)).
\end{equation*}
By \eqref{Eq: second case misc 0}, $\Gamma(\underline{a}_2, b)$ must be positive for every sufficiently large $b$. Thus, $\overline{\Gamma}(\underline{a}_2) > 0$.
     
(ii) Suppose that $\Gamma(\underline{a}_2, \infty) = 0$. We further introduce two sub-cases. 

(ii)-(A) Suppose that $\rho_{\underline{a}_2}^{(q + \lambda)}(\infty; \tilde{f}') < K_p - K_c$. Then, by \eqref{Eq: partial Gamma}, the function $b \mapsto \Gamma(\underline{a}_2, b)$ is first increasing then decreasing (to zero by assumption). Hence, it must upcross zero. 

(ii)-(B) Suppose that $ \rho_{\underline{a}_2}^{(q + \lambda)}(\infty; \tilde{f}') \geq K_p - K_c$. We will show that this case is impossible in view of \eqref{Eq: second case misc 0}.

Recall that the function $b \mapsto \rho_{\underline{a}_2}^{(q + \lambda)}(b; \tilde{f}')$ is non-increasing. This observation, together with \eqref{Eq: gamma}, implies that $b \mapsto \gamma(\underline{a}_2, b)$ is non-decreasing for all $b > \underline{a}_2$. Moreover, since $\rho_{\underline{a}_2}^{(q + \lambda)}(\infty; \tilde{f}') \geq K_p - K_c$, $b \mapsto \gamma(\underline{a}_2, b)$ is also non-positive. Using \eqref{Eq: Gamma} with $\Gamma(\underline{a}_2, \infty) = 0$, which is the standing assumption in case (ii),
    \begin{multline*}
        \Gamma(\underline{a}_2, b) = \Gamma(\underline{a}_2, b)  - \Gamma(\underline{a}_2, \infty) = -\lambda \int_{b - \underline{a}_2}^\infty e^{-\Phi(q)y}\rho_{\underline{a}_2}^{(q + \lambda)}(y + \underline{a}_2; \tilde{f}') \, \diff y + \frac{\lambda}{\Phi(q)} e^{-\Phi(q)(b - \underline{a}_2)} \left(K_p - K_c\right)\\
        = \lambda \int_{b - \underline{a}_2}^\infty e^{-\Phi(q)y} \left(-\rho_{\underline{a}_2}^{(q + \lambda)}(y + \underline{a}_2; \tilde{f}')  + K_p - K_c\right) \, \diff y = \lambda \int_{b - \underline{a}_2}^\infty e^{-\Phi(q)y} \gamma(\underline{a}_2, y + \underline{a}_2) \, \diff y.
    \end{multline*}
    It follows that
    \begin{equation}\label{Gamma_gamma_as_gamma}
        \frac{\Phi(q)}{q} \Gamma(\underline{a}_2, b) - \frac{\lambda}{q} e^{-\Phi(q)(b - \underline{a}_2)} \gamma(\underline{a}_2, b) = \frac{\lambda}{q} \int_{b - \underline{a}_2}^\infty \Phi(q) e^{-\Phi(q)y} \left(\gamma(\underline{a}_2, y + \underline{a}_2) - \gamma(\underline{a}_2, b)\right) \, \diff y.
    \end{equation}
    
Since $b \mapsto \gamma(\underline{a}_2, b)$ is non-decreasing and non-positive, it follows that either  
\begin{itemize}
    \item[(a)] $\gamma(\underline{a}_2, b) < 0$ for all $b$, 
    \item[(b)] there exists $M \in \mathbb{N}$ such that $\gamma(\underline{a}_2, b) = 0$ for $b > M$.
\end{itemize}
    
(ii)-(B)-(a) By \eqref{Gamma_gamma_as_gamma}, for $b > \underline{a}_2$,
\begin{align*}
    e^{\Phi(q)(b - \underline{a}_2)} \left(\frac{\Phi(q)}{q} \Gamma(\underline{a}_2, b) - \frac{\lambda}{q} e^{-\Phi(q)(b - \underline{a}_2)} \gamma(\underline{a}_2, b)\right) = \frac{\lambda}{q} \int_{0}^\infty \Phi(q) e^{-\Phi(q)y} \left(\gamma(\underline{a}_2, y + \underline{a}_2) - \gamma(\underline{a}_2, b) \right) \, \diff y.
\end{align*}
In view of the above equation and the monotonicity and negativity of $b \mapsto \gamma(\underline{a}_2, b)$,
\begin{equation}\label{Eq: existence bound}
    e^{\Phi(q)(b - \underline{a}_2)} \left(\frac{\Phi(q)}{q} \Gamma(\underline{a}_2, b) - \frac{\lambda}{q} e^{-\Phi(q)(b - \underline{a}_2)} \gamma(\underline{a}_2, b)\right) \leq -\frac{\lambda}{q} \gamma(\underline{a}_2, b).
\end{equation}
Combining \eqref{Eq: existence bound} with \eqref{Eq: second case misc 2}, and then applying \eqref{Eq: second case misc 0}, we obtain
    \begin{multline*}
        0 < \lim_{b \to \infty} \left( \frac{\frac{\Phi(q)}{q} \Gamma(\underline{a}_2, b) - \frac{\lambda}{q} e^{-\Phi(q)(b - \underline{a}_2)} \gamma(\underline{a}_2, b)}{\Theta^{(q + \lambda)}(b - \underline{a}_2, \Phi(q))} Z^{(q + \lambda)}(b - \underline{a}_2) + \frac{\lambda + q}{q} \gamma(\underline{a}_2, b) \right)\\
        \leq \lim_{b \to \infty} \left(-\frac{\lambda}{q}\frac{\lambda + q}{\lambda} \frac{\Phi(q + \lambda) - \Phi(q)}{\Phi(q + \lambda)} \gamma(\underline{a}_2, b) + \frac{\lambda + q}{q} \gamma(\underline{a}_2, b) \right) = \lim_{b \to \infty} \frac{\lambda + q}{q} \frac{\Phi(q)}{\Phi(q + \lambda)}\gamma(\underline{a}_2, b),
    \end{multline*}
    which contradicts the negativity of $b \mapsto \gamma(\underline{a}_2, b)$.
    
(ii)-(B)-(b) 
For $b > M$, setting $\gamma(\underline{a}_2, b) = 0$ in \eqref{Eq: second case misc 0} yields $\Gamma(\underline{a}_2, b) > 0$. This again raises a contradiction. Indeed, by \eqref{Eq: gamma}, $\gamma(\underline{a}_2, b) = 0$ on $(M, \infty)$ implies that
\begin{equation*}
    \Gamma(\underline{a}_2, \infty) = \Gamma(\underline{a}_2, b) > 0, \quad b > M,
\end{equation*}
which contradicts with $\Gamma(\underline{a}_2, \infty) = 0$.
\end{proof}

Now, we show that case \textbf{(3)} cannot occur. The following lemma proves useful; its proof is provided in Appendix \ref{Appendix: proof of alt form gamma}.
    \begin{lemma}\label{Lemma: alt form gamma}
        For $a < b$, we have
        \begin{align}
            \Gamma(a, b) &= e^{-\Phi(q)(b - a)} \left(\int_0^{\infty} \tilde{f}'(z + a)Z^{(q + \lambda)}(b - a - z, \Phi(q)) \,\diff z + \frac{\lambda}{\Phi(q)}  \left(K_p - K_c\right)\right)\label{Eq: Gamma alt 2},\\
            -\lambda \gamma(a,b) &= e^{\Phi(q)(b - a)}\frac{\partial}{\partial b} \Gamma(a, b) = -\Phi(q) \tilde{\Gamma}(a, b) + \frac{\partial}{\partial b}\tilde{\Gamma}(a, b), \label{Eq: gamma alt 2}
        \end{align}
        where 
        \begin{equation}\label{Eq: Gamma tilde}
            \tilde{\Gamma}(a, b) \coloneqq e^{\Phi(q)(b - a)} \Gamma(a, b) = \int_0^{\infty} \tilde{f}'(z + a) Z^{(q + \lambda)}(b - a - z, \Phi(q))\,\diff z + \frac{\lambda}{\Phi(q)} \left(K_p - K_c\right).
        \end{equation}
    \end{lemma}

\begin{lemma}\label{Lemma: case 3}
    For $\underline{a}_1 > \underline{a}_2$, case \textbf{(3)} (i.e. $\Gamma(\underline{a}_2, \infty) = -\infty$ and $\rho_{\underline{a}_2}^{(q + \lambda)}(\infty; \tilde{f}') = -\infty$) cannot occur.
\end{lemma}
\begin{proof}
    Suppose $\Gamma(\underline{a}_2, \infty) = -\infty$ and $\rho_{\underline{a}_2}^{(q + \lambda)}(\infty; \tilde{f}') = -\infty$ for the sake of contradiction. By \eqref{Eq: second case misc 2} and \eqref{Eq: gamma alt 2}, we have
    \begin{align}\label{Eq: misc x}
    \begin{split}
        &~~ \lim_{b \to \infty} \left( \frac{\frac{\Phi(q)}{q} \Gamma(\underline{a}_2, b) - \frac{\lambda}{q} e^{-\Phi(q)(b - \underline{a}_2)} \gamma(\underline{a}_2, b)}{\Theta^{(q + \lambda)}(b - \underline{a}_2, \Phi(q))} Z^{(q + \lambda)}(b - \underline{a}_2) + \frac{\lambda + q}{q} \gamma(\underline{a}_2, b) \right) \\
        &= \lim_{b \to \infty} \left( \frac{\lambda + q}{\lambda} \frac{\Phi(q + \lambda) - \Phi(q)}{\Phi(q + \lambda)} \left(\frac{\Phi(q)}{q} \tilde{\Gamma}(\underline{a}_2, b) - \frac{\lambda}{q} \gamma(\underline{a}_2, b) \right) + \frac{\lambda + q}{q} \gamma(\underline{a}_2, b) \right)\\
        &= \lim_{b \to \infty} \frac{\lambda + q}{q\lambda} \left( -\frac{\Phi(q)}{\Phi(q + \lambda)} \frac{\partial}{\partial b} \tilde{\Gamma}(\underline{a}_2, b) + \Phi(q) \tilde{\Gamma}(\underline{a}_2, b) \right).
        \end{split}
    \end{align}
    Combining \eqref{Eq: misc x} with the inequality established in \eqref{Eq: second case misc 0}, we obtain
    \begin{equation} \label{inequality_gamma_tilde}
        0 < \lim_{b \to \infty} \frac{\lambda + q}{q\lambda} \left( -\frac{\Phi(q)}{\Phi(q + \lambda)} \frac{\partial}{\partial b} \tilde{\Gamma}(\underline{a}_2, b) + \Phi(q) \tilde{\Gamma}(\underline{a}_2, b) \right) \, \text{ and thus } \, 0 < \lim_{b \to \infty} \left(\tilde{\Gamma}(\underline{a}_2, b) -\frac{\frac{\partial}{\partial b} \tilde{\Gamma}(\underline{a}_2, b)}{\Phi(q + \lambda)} \right).
    \end{equation}

    Since $b \mapsto \rho_{\underline{a}_2}^{(q + \lambda)}(b; \tilde{f}')$ is non-increasing and non-positive on $(\underline{a}_2, \infty)$ (with $\rho_{\underline{a}_2}^{(q + \lambda)}(\infty; \tilde{f}') = -\infty$), it follows from \eqref{Eq: partial Gamma} that $b \mapsto \frac{\partial}{\partial b} \Gamma(\underline{a}_2, b)$ is positive at first and then negative. Moreover, because $\Gamma(\underline{a}_2, \infty) = -\infty$, there exists $b_0 > \underline{a}_2$ such that, for all $b > b_0$, the function $b \mapsto \Gamma(\underline{a}_2, b)$ is negative and strictly decreasing. Then, by the definition of $\tilde{\Gamma}$ in \eqref{Eq: Gamma tilde} as the product of $\Gamma$ and a positive function increasing in $b$, the function $b \mapsto \tilde{\Gamma}(\underline{a}_2, b)$ is monotonically decreasing on $(b_0,\infty)$ and tends to $-\infty$. Thus, for sufficiently large $b$, $\frac{\partial}{\partial b} \tilde{\Gamma}(\underline{a}_2, b) < 0$; combining this with \eqref{inequality_gamma_tilde} yields
    \begin{equation}\label{Eq: inequality x1}
        \frac{1}{\Phi(q + \lambda)} > \frac{\tilde{\Gamma}(\underline{a}_2, b)}{\frac{\partial}{\partial b} \tilde{\Gamma}(\underline{a}_2, b)}.
    \end{equation}
    Using the above inequality and that $b \mapsto \tilde{\Gamma}(\underline{a}_2, b)$ is decreasing on $(b_0, \infty)$ with $\tilde{\Gamma}(\underline{a}_2, b) \to -\infty$, we conclude that the function $b \mapsto |\tilde{\Gamma}(\underline{a}_2, b)|$ grows asymptotically no slower than an exponential function with parameter $\Phi(q + \lambda)$.

    The same growth behavior can be established for the function $b \mapsto |\gamma(\underline{a}_2, b)|$. To see this, by \eqref{Eq: gamma alt 2},
    \begin{align*}
        -\gamma(\underline{a}_2, b) = -\frac{\Phi(q)}{\lambda} \tilde{\Gamma}(\underline{a}_2, b) + \frac{\frac{\partial}{\partial b}\tilde{\Gamma}(\underline{a}_2, b)}{\lambda}, \quad b > \underline{a}_2.
    \end{align*}
    For sufficiently large $b$, combining \eqref{Eq: inequality x1} with $\tilde{\Gamma}(\underline{a}_2, b) < 0$, we obtain $\frac{\partial}{\partial b} \tilde{\Gamma}(\underline{a}_2, b) < \Phi(q + \lambda) \tilde{\Gamma}(\underline{a}_2, b) < 0$. Thus, 
    \begin{align*}
        -\gamma(\underline{a}_2, b) < -\frac{\Phi(q)}{\lambda} \tilde{\Gamma}(\underline{a}_2, b) + \frac{\Phi(q + \lambda)}{\lambda} \tilde{\Gamma}(\underline{a}_2, b) = \frac{\Phi(q + \lambda) - \Phi(q)}{\lambda} \tilde{\Gamma}(\underline{a}_2, b)< 0.
    \end{align*}
    The growth behavior of $b \mapsto |\gamma(\underline{a}_2, b)|$ follows from the above inequality and the growth of $b \mapsto |\tilde{\Gamma}(\underline{a}_2, b)|$.
    
    Note that the growth behavior of $b \mapsto |\gamma(\underline{a}_2, b)|$ leads to a contradiction. Indeed, as $\gamma(\underline{a}_2, b) = -\rho_{\underline{a}_2}^{(q + \lambda)}(b; \tilde{f}') + K_p - K_c$, we obtain 
    \begin{equation*}
        \lim_{b \to \infty} \frac{-\rho_{\underline{a}_2}^{(q + \lambda)}(b; \tilde{f}')}{W^{(q + \lambda)}(b - \underline{a}_2)} = \lim_{b \to \infty}  \frac{\gamma(\underline{a}_2, b)}{W^{(q + \lambda)}(b - \underline{a}_2)} > 0,
    \end{equation*}
    where the inequality follows as $W^{(q + \lambda)}(x) \sim \frac{e^{\Phi(q + \lambda)x}}{\psi'(\Phi(q + \lambda))}$ for large $x$ and $\gamma(\underline{a}_2, b) > 0$ for large $b$. This contradicts with the definition of $\underline{a}_2$, which requires (see \eqref{Eq: end value misc 1})
    \begin{equation*}
        0 = \int_0^\infty e^{-\Phi(q + \lambda)y} \tilde{f}'(y + \underline{a}_2)\, \diff y = \lim_{b \to \infty}  \frac{\rho_{\underline{a}_2}^{(q + \lambda)}(b; \tilde{f}')}{W^{(q + \lambda)}(b - \underline{a}_2)} = \lim_{b \to \infty}  \frac{-\gamma(\underline{a}_2, b)}{W^{(q + \lambda)}(b - \underline{a}_2)}.
    \end{equation*}
    Thus, we conclude that case \textbf{(3)}, where $\Gamma(\underline{a}_2, \infty) = -\infty$ and $\rho_{\underline{a}_2}^{(q + \lambda)}(\infty; \tilde{f}') = -\infty$, is impossible. 
    \end{proof}
    
    In summary, we have established that case \textbf{(3)} cannot occur, and that in cases \textbf{(1)} and \textbf{(2)}, $\overline{\Gamma}(\underline{a}_2-) > 0$. It remains to show that $\overline{\Gamma}$ is negative for some $a \in (-\infty, \underline{a}_2)$. To this end, recall $a^\dagger$ as defined in \eqref{Eq: a tilde}. We will show that $a^\dagger < \underline{a}_2$.

    Since $b \mapsto \Gamma(\underline{a}_2, b)$ upcrosses zero and $\rho_{\underline{a}_2}^{(q + \lambda)}(b; \tilde{f}')$ is non-positive, in view of \eqref{Eq: Gamma}, it must hold that $\int_0^\infty e^{-\Phi(q)y} \tilde{f'}(y + \underline{a}_2) \, \diff y > 0$. Hence, $a^\dagger < \underline{a}_2$. The argument in Section \ref{Sect: case 1} implies $\overline{\Gamma}(a^\dagger) < 0$, and consequently the existence of $a^\dagger < a^* < \underline{a}_2$ and $b^* > a^*$ such that $\Gamma(a^*, b^*) = \gamma(a^*, b^*) = 0$.

\section{Verification of optimality}\label{Sect: optimality}
In this section, we establish the optimality of the candidate policy $\pi^{a^*, b^*}$, which constitutes the main result of this study. Let $\mathcal{L}$ be the infinitesimal generator associated with the spectrally negative L\'evy process $X$, when applied to a sufficiently smooth function $h: \mathbb{R} \to \mathbb{R}$, it holds that
\begin{equation}\label{Eq: generator}
    \mathcal{L}h(x) \coloneqq \gamma h'(x) + \frac{\sigma^2}{2}h''(x) + \int_{(-\infty, 0)} \left[h(x + z) - h(x) - h'(x) z 1_{\{-1 < z < 0\}} \right]\,\mu(\diff z), \quad x\in \mathbb{R}. 
\end{equation}
Additionally, we define an operator $\mathcal{M}$ acting on a measurable function $h$ by
\begin{equation*}
    \mathcal{M}h(x) \coloneqq \inf_{l \geq 0} \{K_pl + h(x + l)\}.
\end{equation*} 

The following lemma outlines the sufficient conditions for the optimality of $\pi^{a^*, b^*}$. The proof follows a standard argument involving It\^o's lemma, as in \cite{avram_exit_2004, baurdoux_optimality_2015, perez_optimal_2020}, among others, and is therefore omitted.

\begin{lemma}\label{Lemma: verification lemma}
Suppose that $w: \mathbb{R} \to \mathbb{R}$ is the total cost under an admissible policy. Further suppose that the function $w$ is sufficiently smooth on $\mathbb{R}$, has polynomial growth, satisfies
    \begin{align}
        (\mathcal{L}-q)w + \lambda(\mathcal{M}w - w) + f &\geq 0,\label{Eq: vi condition 1}\\
        w' + K_c &\geq 0. \label{Eq: vi condition 2}
    \end{align}
    Additionally, suppose that
    \begin{equation}\label{Eq: limsup}
        \limsup_{t, n \uparrow \infty} \mathbb{E}_x\left[e^{-q(t\wedge \tau_n)} w(Y^\pi(t\wedge \tau_n))\right] \leq 0, \quad x \in \mathbb{R}, 
    \end{equation}
    for all admissible policies $\pi \in \Pi$, where $\tau_n \coloneqq \inf\{t \geq 0: |Y^{\pi}(t)| > n\}$, for $n \in \mathbb{N}$. Then, the function $w$ coincides with the value function, i.e., $w(x) = v(x) = \inf_{\pi \in \Pi} v^\pi(x)$ for all $x \in \mathbb{R}$.
\end{lemma}

In the remainder of this section, we show that the conditions in Lemma \ref{Lemma: verification lemma} are satisfied for the function $v_{a^*, b^*}$. First, we establish certain properties of the function $v_{a^*,b^*}$ for the computation of $\mathcal{M} v_{a^*,b^*}$.
\begin{proposition}\label{Prop: convexity}
    It holds that:
\begin{enumerate}
    \item[(1)]  $v^{f'}_{a^*, b^*}(x) = v'_{a^*, b^*}(x)$ for all $x \in \mathbb{R}$, 
    \item[(2)] $x \mapsto v_{a^*, b^*}(x)$ is convex on $\mathbb{R}$,
    \item[(3)] $v'_{a^*, b^*}(x) + K_c \geq 0$ for all $x \in \mathbb{R}$.
\end{enumerate}
\end{proposition}

\begin{proof}
    (1) Since $(a^*,b^*)$ satisfies $\mathfrak{C}$ (which yields $\Gamma(a^*, b^*) = \gamma(a^*, b^*) = 0$), we obtain the desired equality after substituting this condition into \eqref{Eq: NPV derivative 2} and \eqref{Eq: v^f'}, with $a = a^*$ and $b = b^*$.
    \\
    (2) By the convexity of $f$ and the monotonicity of $Y^{a^*, b^*}$ in the starting point almost surely, the function $x \mapsto v^{f'}_{a^*, b^*}(x)$ is non-decreasing. Hence, by (1), $x \mapsto v'_{a^*, b^*}(x)$ is non-decreasing, which implies $v_{a^*, b^*}$ is convex.
    \\
    (3) By (2) and the fact that $v'_{a^*, b^*}(x) = v'_{a^*, b^*}(a^*) = -K_c$ for $x \in (-\infty, a^*]$ (see Equation \eqref{Eq: NPV derivative 2}), statement (3) holds. 
\end{proof}

The following result is a direct consequence of the convexity of $x \mapsto v_{a^*, b^*}(x)$ and the fact that $v'_{a^*, b^*}(b^*) = -K_p$, which are due to Proposition \ref{Prop: convexity}(2) and Proposition \ref{Prop: iff conditions}(3), respectively.
\begin{corollary} \label{Corollary: generator} 
It holds that
\begin{align*}
    \mathcal{M}v_{a^*, b^*}(x) - v_{a^*, b^*}(x) &= \begin{cases}
        K_p(b^* - x) + v_{a^*, b^*}(b^*) - v_{a^*, b^*}(x), & x < b^*,\\
        0, & x \geq b^*.\\
    \end{cases}
\end{align*}
\end{corollary}

By standard computations, we obtain Lemma \ref{Lemma: generator}. The proof is provided in Appendix \ref{Appendix: proof of generator}.
\begin{lemma} \label{Lemma: generator}
It holds that
\begin{align*}
    (\mathcal{L} - q) v_{a^*, b^*}(x) + f(x) &= \begin{cases}
        -\lambda\left(v_{a^*, b^*}(b^*) - v_{a^*, b^*}(x) + K_p(b^* - x)\right), &  a^* < x < b^*,\\
        0, & x \geq b^*.\\
    \end{cases}
\end{align*}
For $x \leq a^*$, 
\begin{equation*}
    (\mathcal{L}-q)v_{a^*, b^*}(x) + \lambda(\mathcal{M}v_{a^*, b^*}(x) - v_{a^*, b^*}(x)) + f(x) = \tilde{f}(x) - \tilde{f}(a^*) \geq 0. 
\end{equation*}
\end{lemma}

The polynomial growth of $v_{a^*, b^*}$ may also be established through a standard argument. It is a direct consequence of the polynomial growth of $f$ and the $\mathbb{P}$-almost sure bound:
\begin{equation*}
    \left|Y^{a^*, b^*}_y(t) - Y^{a^*, b^*}_x(t)\right| \vee \left|R^{a^*, b^*}_{p,y}(t) - R^{a^*, b^*}_{p,x}(t)\right| \vee \left|R^{a^*, b^*}_{c,y}(t) - R^{a^*, b^*}_{c,x}(t)\right| \leq y - x, \quad y > x, \, t \geq 0,
\end{equation*}
where $Y^{a^*, b^*}_y$ and $Y^{a^*, b^*}_x$ denote the $\pi^{a^*, b^*}$-controlled processes driven by $(y + X(t); t \geq 0)$ and $(x + X(t); t \geq 0)$, respectively, and the other terms are defined analogously. Thus, we state Lemma \ref{Lemma: polynomial growth} without proof.
\begin{lemma}\label{Lemma: polynomial growth}
    The function $x \mapsto v_{a^*, b^*}(x)$ is of polynomial growth.
\end{lemma}

The proof of Lemma \ref{Lemma: verification limit} is presented in Appendix \ref{Appendix: proof of verification limit}.
\begin{lemma}\label{Lemma: verification limit}
    For all admissible policies $\pi \in \Pi$ and $x \in \mathbb{R}$, condition \eqref{Eq: limsup} holds with $w = v_{a^*, b^*}$. 
\end{lemma}

We are now ready to establish the optimality of a candidate policy, as stated in Theorem \ref{Thm: optimal policy}.
\begin{theorem}\label{Thm: optimal policy}
The policy $\pi^{a^*,b^*}$ is optimal within the set of admissible policies $\Pi$, with 
\[v_{a^*, b^*}(x) = \inf_{\pi \in \Pi} v^\pi(x), \quad x\in \mathbb{R}.\]
\end{theorem}
\begin{proof}
    By Propositions \ref{Prop: iff conditions}(3) and \ref{Prop: convexity}(3), $v_{a^*, b^*}$ is sufficiently smooth on $\mathbb{R}$ with $v_{a^*, b^*}'\geq -K_c$. By Corollary \ref{Corollary: generator} and Lemma \ref{Lemma: generator}, the variational inequality \eqref{Eq: vi condition 1} is satisfied for all $x \in \mathbb{R}$. It suffices to note that $v_{a^*, b^*}$, the total cost under the hybrid policy $\pi^{a^*, b^*}$, is of polynomial growth (as shown in Lemma \ref{Lemma: polynomial growth}), and that Lemma \ref{Lemma: verification limit} implies condition \eqref{Eq: limsup} is fulfilled. Thus, applying Lemma 
    \ref{Lemma: verification lemma}, $v_{a^*, b^*}(x) = v(x)$ for $x \in \mathbb{R}$.
\end{proof}

\subsection{Significance of Assumption \ref{asm: f slope}(1).}\label{Sect: asm on f}
We conclude the analysis by examining cases in which Assumption \ref{asm: f slope}(1) is violated, with all other assumptions satisfied. For these cases, we show that the pure discounted replenishment policy, with a suitable barrier $b^* \in \mathbb{R}$ as studied in \cite{perez_optimal_2020}, is optimal. We note that a pure discounted replenishment policy with barrier $b\in \mathbb{R}$ is a special case of the hybrid barrier policy, corresponding to a lower barrier $a = -\infty$ and an upper barrier $b$.

Since Assumption \ref{asm: f slope}(1) is violated, we must have
\begin{equation}\label{Eq: appen misc 1}
    \tilde{f}'(x) \geq 0 \implies f'(x) \geq -qK_c + \lambda(K_p - K_c), \quad x \in \mathbb{R}.
\end{equation}
Denote by $Y^{-\infty, b}$ the replenished inventory process under the pure discounted replenishment policy with barrier $b \in \mathbb{R}$. The processes $R_c^{-\infty, b}$ and $R_p^{-\infty, b}$ are defined analogously, with $R_c^{-\infty, b}$ almost surely uniformly zero. For $x \in \mathbb{R}$, we define
\begin{align*}
    v_{-\infty, b}(x) &\coloneqq \mathbb{E}_{x}\left[\int^\infty_0 e^{-qt} f(Y^{-\infty, b}(t))\, \diff  t\right] + K_p\mathbb{E}_{x}\left[ \int_{[0, \infty)} e^{-qt} \, \diff R_p^{-\infty, b}(t)\right],\\
    v_{-\infty, b}^{f'}(x) &\coloneqq \mathbb{E}_{x}\left[\int^\infty_0e^{-qt} f'(Y^{-\infty, b}(t))\, \diff  t\right].
\end{align*}

As Assumption \ref{asm: f slope}(2) holds, Lemma 5.5 together with Proposition 5.8 in \cite{perez_optimal_2020} shows that there exists some $b^* \in \mathbb{R}$ such that $v_{-\infty, b^*}' = v_{-\infty, b^*}^{f'}$ and
\begin{equation}\label{Eq: pure periodic misc 1}
    v_{-\infty, b^*}^{f'}(b^*) = \mathbb{E}_{b^*}\left[\int^\infty_0e^{-qt} f'(Y^{-\infty, b^*}(t))\, \diff  t\right] = -K_p.
\end{equation}
The following decomposition proves useful, 
\begin{equation}\label{Eq: pure periodic misc 0}
    v_{-\infty, b^*}^{f'}(x) = \mathbb{E}_{x}\left[\int^{T(1)}_0 e^{-qt} f'(Y^{-\infty, b^*}(t))\, \diff  t + \int^\infty_{T(1)} e^{-qt} f'(Y^{-\infty, b^*}(t))\, \diff  t\right],
\end{equation}
where $T(1)$ is an independent exponential time with rate $\lambda$. For the first term on the right-hand side of \eqref{Eq: pure periodic misc 0}, by \eqref{Eq: appen misc 1}, 
\begin{equation*}
    \mathbb{E}_{x}\left[\int^{T(1)}_0 e^{-qt} f'(Y^{-\infty, b^*}(t))\, \diff  t \right] \geq \frac{-qK_c + \lambda(K_p - K_c)}{q}\mathbb{E}_{x}\left[1 - e^{-qT(1)}\right] = \frac{-qK_c + \lambda(K_p - K_c)}{\lambda + q}. 
\end{equation*}

Note that $Y^{-\infty, b^*}$ is a strong Markov process with $Y^{-\infty, b^*}(T(1)) \geq b^*$ almost surely. Together with the independence of $T(1)$, these observations imply that the following inequality holds for the second term in \eqref{Eq: pure periodic misc 0}:
\begin{equation*}
    \mathbb{E}_{x}\left[\int^\infty_{T(1)} e^{-qt} f'(Y^{-\infty, b^*}(t))\, \diff  t\right] \geq \mathbb{E}_{x}\left[e^{-qT(1)}\right](-K_p) = -K_p\frac{\lambda}{\lambda + q}. 
\end{equation*}
Hence, we obtain the inequality
\begin{equation*}
    v_{-\infty, b^*}^{f'}(x) \geq \frac{-qK_c + \lambda(K_p - K_c)}{\lambda + q} - \frac{\lambda}{\lambda + q} K_p = -K_c,
\end{equation*}
from which it follows that $v_{-\infty, b^*}'(x) = v_{-\infty, b^*}^{f'}(x) \geq -K_c$, for $x \in \mathbb{R}$. Note that this implies \eqref{Eq: vi condition 2} is satisfied for the function $v_{-\infty, b^*}$. The other conditions, including smoothness, growth rate, and \eqref{Eq: vi condition 1}, were established in \cite{perez_optimal_2020}. Then, condition \eqref{Eq: limsup} follows as in the proof of Lemma \ref{Lemma: verification limit}. By applying Lemma \ref{Lemma: verification lemma}, we establish the optimality of the pure discounted replenishment policy with barrier $b^*$. 

\section{Numerical study}
This section presents a numerical study to illustrate the main results of the paper. We begin by identifying a pair of candidate barriers, followed by an illustration of the optimality result stated in Theorem \ref{Thm: optimal policy}. We then perform a sensitivity analysis with respect to the unit costs $K_c$ and $K_p$, and conclude with an examination of the cost savings achieved by an optimal hybrid barrier policy relative to pure discounted and pure regular replenishment policies. 

\subsection{Setting.}\label{Sect: numerical setting}
Suppose that the inventory-holding cost function $f$ is given by $f(x) = x^2$. This choice of the function $f$ guarantees that Assumption \ref{asm: f slope} is satisfied for any pair of $K_c > K_p$. Unless otherwise specified, we fix $K_c = 10$ and $K_p = 2$ throughout the following analysis. The uncontrolled inventory process $X$ is modeled by 
\begin{equation*}
    X(t) \coloneqq x + t + B(t) - \sum^{N(t)}_{n = 1} Z_n,   
\end{equation*}
where $B = (B(t); t\geq 0)$ is a standard Brownian motion, $N = (N(t); t\geq 0)$ is a Poisson process with an arrival rate of $0.2$, and $\{Z_n\}_{n \geq 1}$ is a sequence of i.i.d. exponential random variables with rate $1$. It is known from Example 3.1 of \cite{egami_phase-type_2014} that the $q$-scale function of $X$ is
\begin{equation}\label{Eq: W numerical}
    W^{(q)}(x) = \frac{e^{\Phi(q) x}}{\psi'(\Phi(q))} - \sum^{2}_{i = 1} B_{i, q} e^{-\xi_{i, q} x}, 
\end{equation}
where, for $i = 1, 2$, we have $B_{i, q} = -1/\psi'(-\xi_{i, q})$, with $\xi_{i, q}$ defined such that $\psi(-\xi_{i, q}) = q$ and $-\xi_{i, q} < 0$. We further fix $q = 0.05$ and allow $\lambda$ to vary, where $\lambda$ is the Poisson arrival rate.

\subsection{Optimal barriers.}\label{Sect: numerical barriers}
By setting $\lambda = 2$ along with the parameters specified in Section \ref{Sect: numerical setting}, we obtain an example of the case $\underline{a}_1 \leq \underline{a}_2$. From the analysis in Section \ref{Sect: case 1}, we know that $a^\dagger < a^* < \underline{a}_1$. Hence, $a^*$ can be identified by finding a root to the function $a \mapsto \overline{\Gamma}(a)$ within the interval $(a^\dagger, \underline{a}_1)$. Figure \ref{Fig: Gamma} illustrates the behavior of the function $b \mapsto \Gamma(a, b)$ for values of $a$ near $a^*$.

\begin{figure}[ht]
\centering
\includegraphics[width=0.6\textwidth]{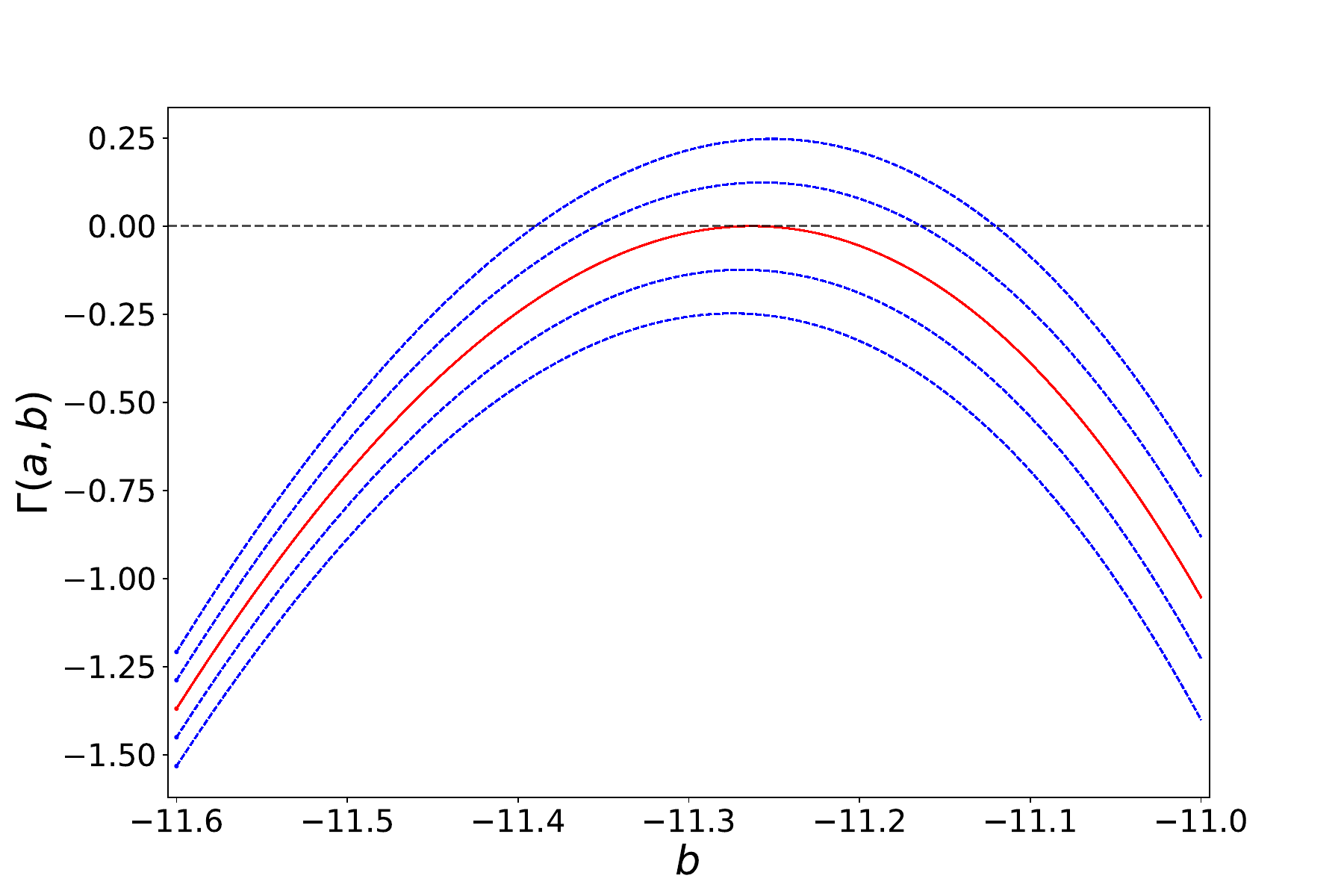}
\caption{\small Plot of $b \mapsto \Gamma(a, b)$ for $a = a^* - 0.010, a^* - 0.005, a^* + 0.005, a^* + 0.010$ (blue dashed) and $b \mapsto \Gamma(a^*, b)$ (red solid).}
\label{Fig: Gamma}
\end{figure}
Using the values $(a^*, b^*)$ identified in Figure \ref{Fig: Gamma}, Figure \ref{Fig: first case} displays the value function $v_{a^*, b^*}$, along with the total cost functions $v_{a, b}$ corresponding to hybrid barrier replenishment policies with $(a, b) \neq (a^*, b^*)$. We can confirm in Figure \ref{Fig: first case} that $v_{a^*, b^*}(x) \leq v_{a, b}(x)$ for the displayed range of $x$. 

Now, setting $\lambda = 12$, we obtain an instance of the case $\underline{a}_2 < \underline{a}_1$. By repeating the above procedure, we identify a pair of candidate barriers $(a^*, b^*)$ for this case. Figure \ref{Fig: second case} illustrates the corresponding value function, along with the total cost functions of other (suboptimal) hybrid barrier replenishment policies.

\begin{figure}[ht]
\centering
\begin{subfigure}[b]{0.49\textwidth}
    \centering
    \includegraphics[width=\textwidth]{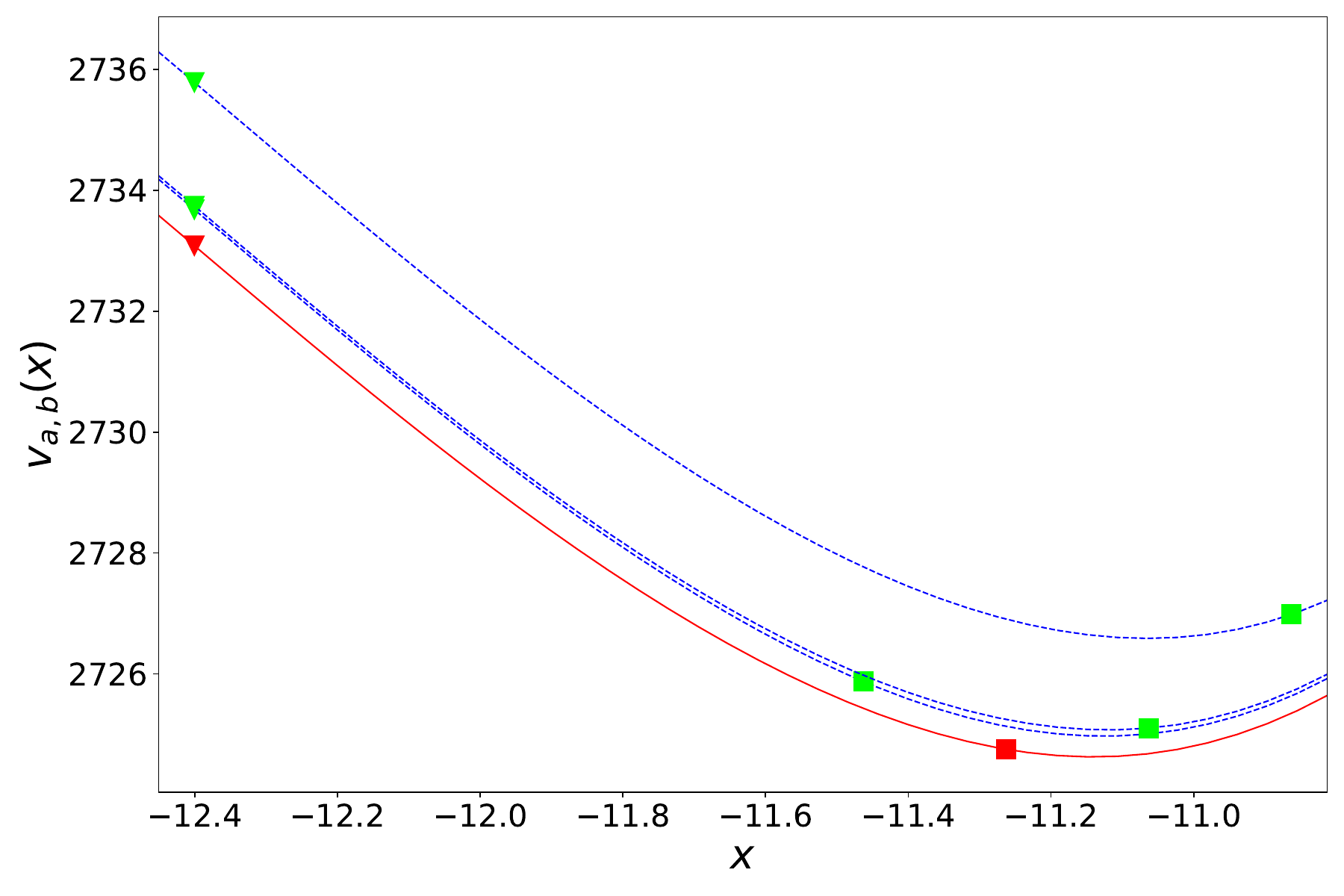}
\end{subfigure}
\vspace{1em}
\begin{subfigure}[b]{0.49\textwidth}
    \centering
    \includegraphics[width=\textwidth]{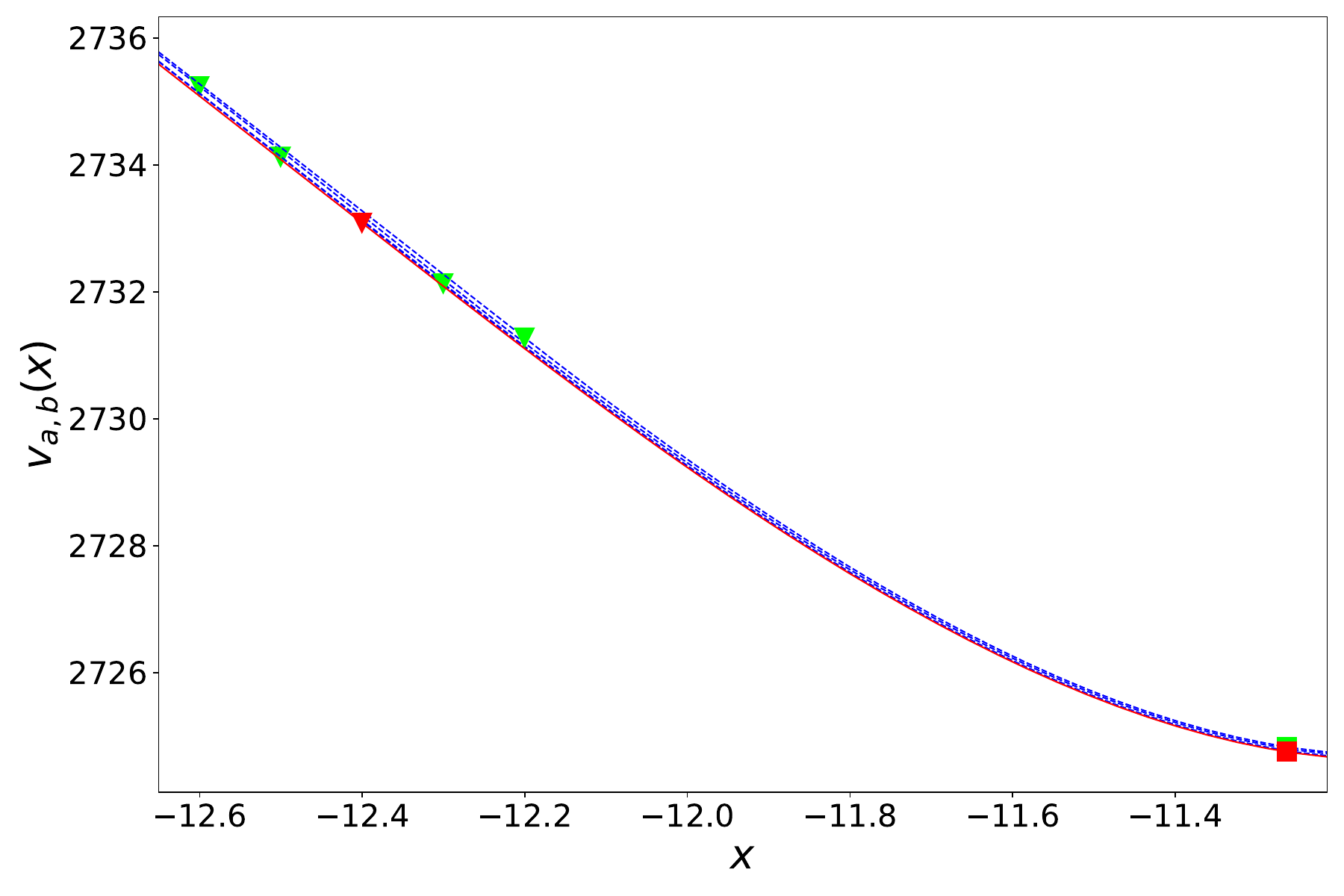}
\end{subfigure}

\caption{\small Left: Plot of $v_{a^*, b}$ (blue dashed) for $b = b^* - 0.2, b^*, b^* + 0.2, b^* + 0.4$, with $(a^*, v_{a^*, b}(a^*))$ (lime triangle), $(a^*, v_{a^*, b^*}(a^*))$ (red triangle), $(b, v_{a^*, b}(b))$ (lime square), and $(b^*, v_{a^*, b^*}(b^*))$ (red square). Right: Plot of $v_{a, b^*}$ (blue dashed) for $a = a^* - 0.2, a^* - 0.1, a^*, a^* + 0.1, a^* + 0.2$, with $(a, v_{a, b^*}(a))$ (lime triangle), $(a^*, v_{a^*, b^*}(a^*))$ (red triangle), $(b^*, v_{a, b^*}(b^*))$ (lime square), and $(b^*, v_{a^*, b^*}(b^*))$ (red square). The optimal total cost function $v_{a^*, b^*}$ is indicated by red curves.}
\label{Fig: first case}
\end{figure}

\begin{figure}[ht]
\centering
\begin{subfigure}[b]{0.49\textwidth}
    \centering
    \includegraphics[width=\textwidth]{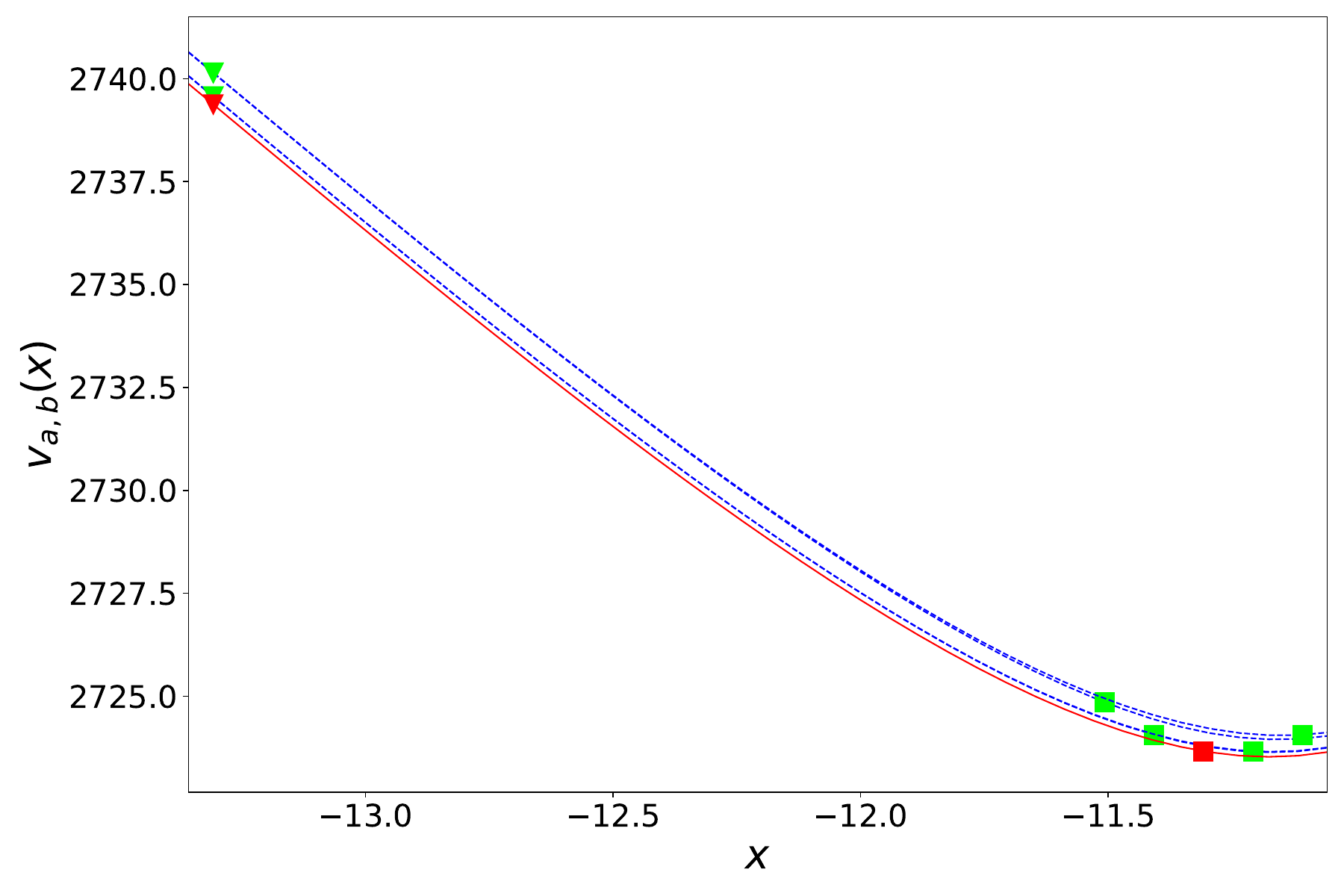}
\end{subfigure}
\vspace{1em}
\begin{subfigure}[b]{0.49\textwidth}
    \centering
    \includegraphics[width=\textwidth]{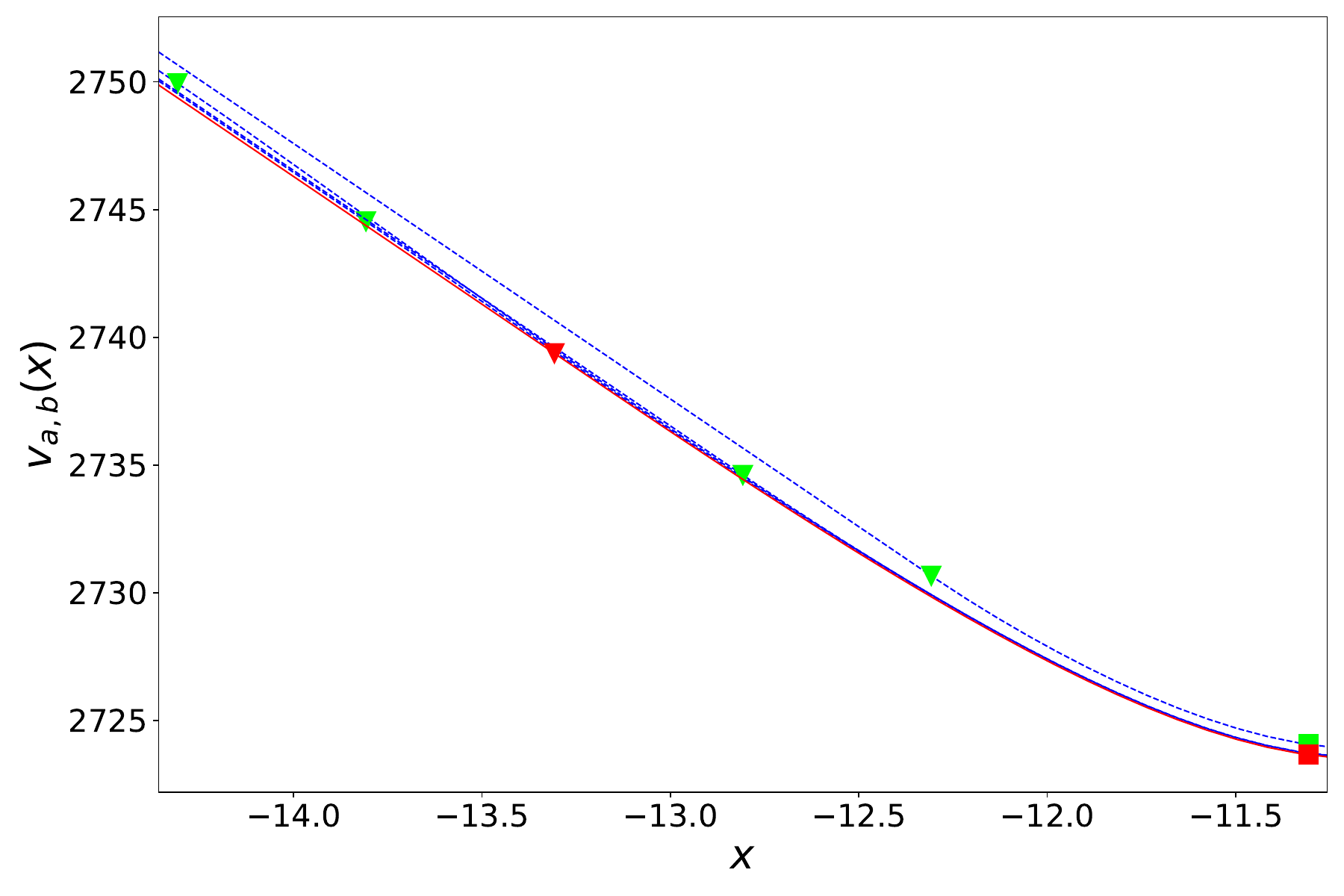}
\end{subfigure}

\caption{\small Left: Plot of $v_{a^*, b}$ for $b = b^* - 0.2, b^* - 0.1, b^*, b^* + 0.1, b^* + 0.2$. Right: Plot of $v_{a, b^*}$ for $a = a^* - 1, a^* - 0.5, a^*, a^* + 0.5, a^* + 1$.}
\label{Fig: second case}
\end{figure}

\subsection{Sensitivity analysis.}
We examine the sensitivity of the optimal solution, including the optimal barrier values identified in Theorem \ref{Thm: existence uniqueness} and the value function, with respect to the unit cost of replenishment. Specifically, for a fixed discounted replenishment cost $K_p$, we examine the behavior of the optimal solution as the regular replenishment cost $K_c$ approaches $K_p$.

First, we recall the model studied in Section 7 of \cite{yamazaki_inventory_2017}, which is a variation of the present model allowing only regular replenishment at a unit cost $C$. In \cite{yamazaki_inventory_2017}, it was established that, if $a^\ddagger \coloneqq \inf \{a\in \mathbb{R}: \int_0^\infty e^{-\Phi(q)y} (f'(y + a) + C) \, \diff y \geq 0\}$ exists, the barrier policy that continuously replenishes the inventory process to the level $a^\ddagger$ is optimal. Set $C = K_p$, for $f: x \mapsto x^2$, we have that $a^\ddagger$ is finite. Denote the corresponding cost function by $v_{a^\ddagger}$, whose semi-explicit form is given in Equation (49) of \cite{yamazaki_inventory_2017}.

As $K_c \to K_p = 2$, the gap between the optimal barriers $(a^*, b^*)$ is expected to shrink, with both barriers converging to $a^\ddagger$. Moreover, the functions $v_{a^*, b^*}$ are expected to converge pointwise to the value function $v_{a^\ddagger}$. These observations are confirmed numerically in Figure \ref{Fig: sensitivity}, which shows that both the functions $v_{a^*, b^*}$ and the optimal barriers $a^*$ and $b^*$ converge monotonically to $v_{a^\ddagger}$ and $a^\ddagger$, respectively.
\begin{figure}[ht]
\centering
    \includegraphics[width=0.6\textwidth]{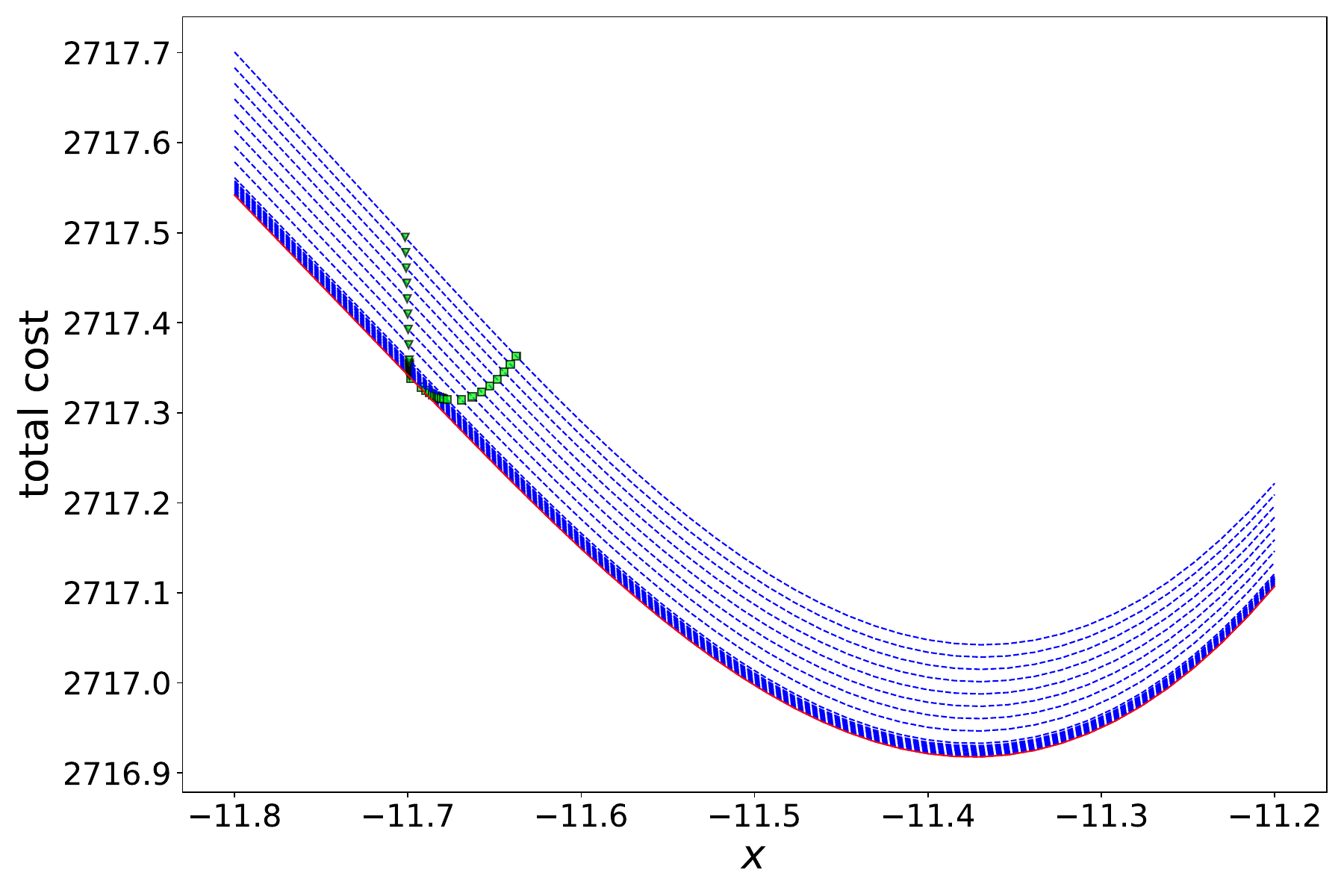}

\caption{\small Value functions $v_{a^*, b^*}$ for $K_c \to K_p$ (blue dashed), with $(a^*, v_{a^*, b^*}(a^*))$ and $(b^*, v_{a^*, b^*}(b^*))$ marked by lime triangles and squares, respectively. The classical value function $v_{a^\ddagger}$ from \cite[Section 7]{yamazaki_inventory_2017} is shown in red.}
\label{Fig: sensitivity}
\end{figure}

\subsection{Cost savings.}
Figure \ref{Fig: savings} displays the value function $v_{a^*, b^*}$, along with the cost functions $v_{a^\ddagger}$ and $v_{b^\ddagger}$, where the latter corresponds to the optimal pure discounted replenishment policy studied in \cite{perez_optimal_2020}. In particular, the left-hand plot corresponds to the case $\lambda = 2$, and the right-hand plot to $\lambda = 0.2$, with all other parameters as specified in Section \ref{Sect: numerical setting}.

We can first confirm in Figure \ref{Fig: savings} that $v_{a^*, b^*}(x) \leq \min(v_{a^\ddagger}(x), v_{b^\ddagger}(x))$ for all $x \in \mathbb{R}$. For the case $\lambda = 2$, a noticeable gap is observed between the function $v_{a^*, b^*}$  and both $v_{a^\ddagger}$ and $v_{b^\ddagger}$, suggesting that an optimal hybrid barrier replenishment policy yields cost savings compared to either a pure regular or pure discounted replenishment policy.

In contrast, for the case $\lambda = 0.2$, a visible gap remains between $v_{a^*, b^*}$ and $v_{b^\ddagger}$, whereas the gap between $v_{a^*, b^*}$ and $v_{a^\ddagger}$ is barely visible. In this case, although the hybrid policy provides substantial cost savings over the pure discounted replenishment policy, its benefit over the pure regular replenishment policy is marginal. 

To explain these observations, recall that $\lambda$ determines the frequency of the Poisson arrival times. When $\lambda$ is small, arrivals are infrequent, making regular replenishment more important for minimizing costs. Hence, an optimal hybrid policy relies more heavily on regular replenishment than on discounted replenishment, resulting in a cost function that closely resembles that of the pure regular replenishment policy. On the other hand, when $\lambda$ is large, Poisson arrivals occur more frequently, reducing the need for regular replenishment and making the advantage of the hybrid policy more evident.

\begin{figure}[ht]
\centering
\begin{subfigure}[b]{0.45\textwidth}
    \centering
    \includegraphics[width=\textwidth]{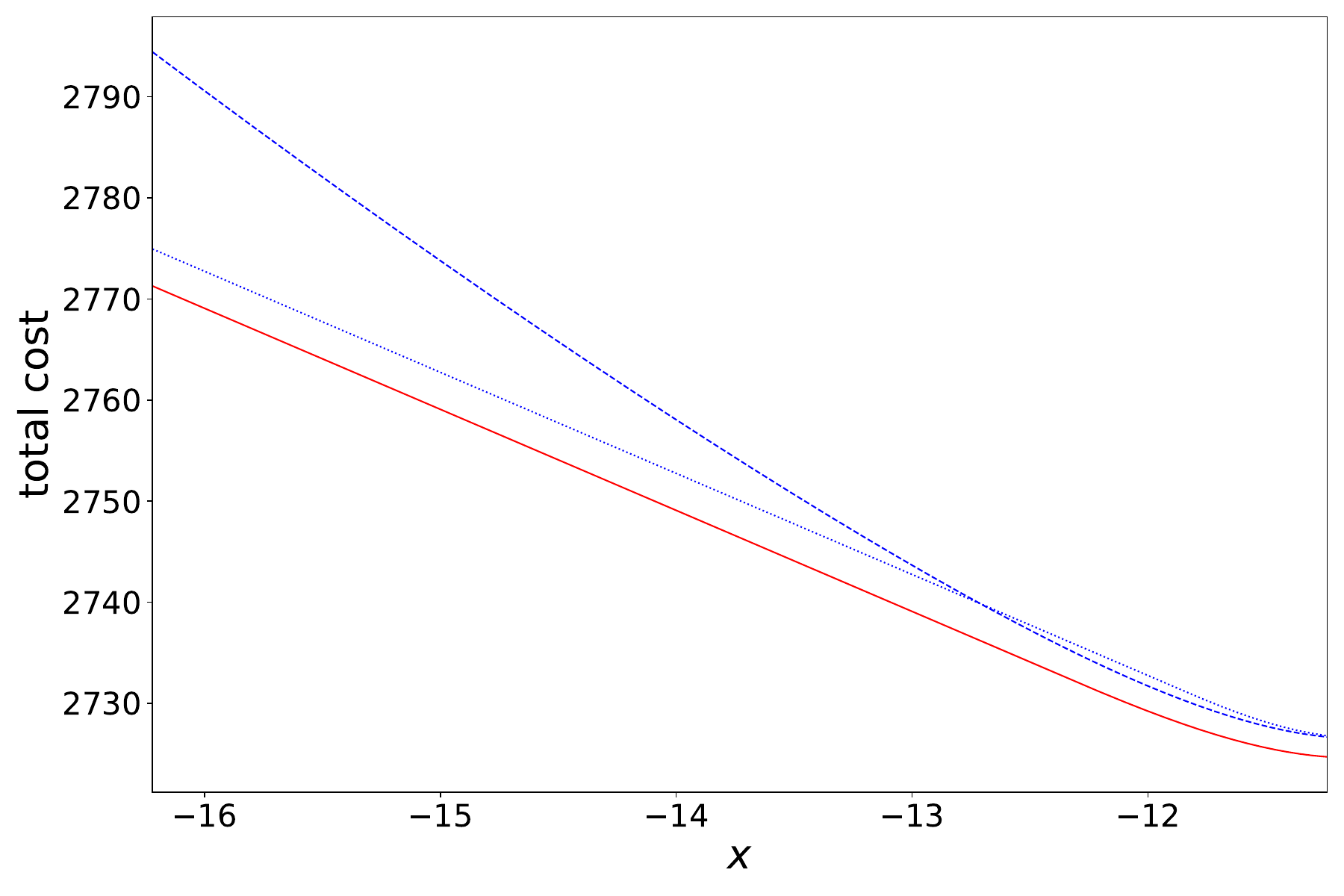}
\end{subfigure}
\vspace{1em}
\begin{subfigure}[b]{0.45\textwidth}
    \centering
    \includegraphics[width=\textwidth]{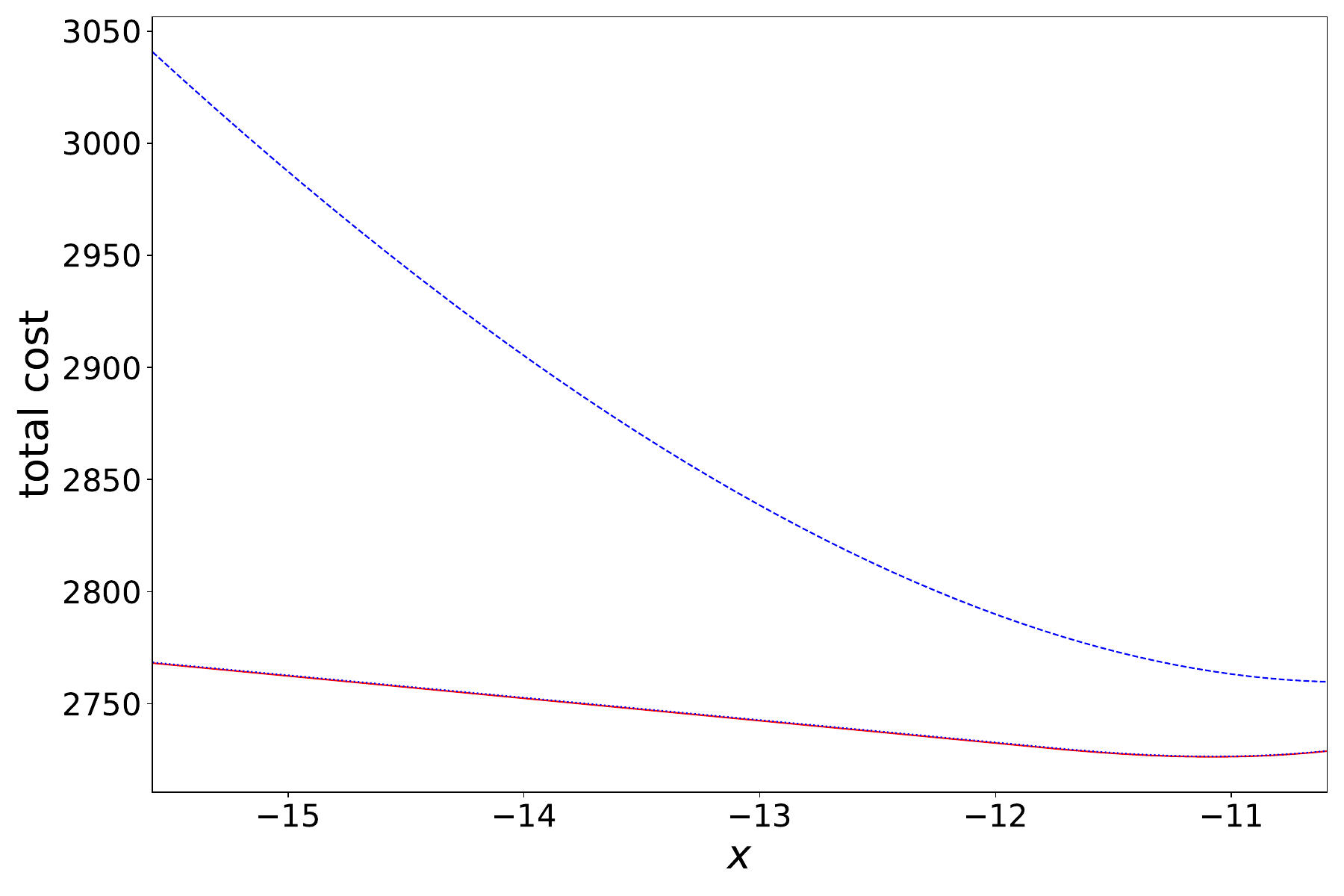}
\end{subfigure}

\caption{\small The functions $v_{a^*, b^*}$, $v_{b^\ddagger}$, and $v_{a^\ddagger}$ are shown as a red solid line, a blue dashed line, and a blue dotted line, respectively. Left: $\lambda = 2$. Right: $\lambda = 0.2$. }
\label{Fig: savings}
\end{figure}
\section*{Acknowledgements}
J. L. P\'erez gratefully acknowledges the support of the Fulbright Program during his sabbatical stay at Arizona State University. He also wishes to thank the faculty and staff of the School of Mathematical and Statistical Sciences at ASU for their hospitality.
\appendix

\section{Proofs}

We first summarize the known results that are used in the proofs in this section. Let $\tau_c^+ \coloneqq \inf\{t \geq 0: X(t) > c\}$ and $\tau_b^- \coloneqq \inf\{t \geq 0: X(t) < b\}$. Fix $b < c$ and $x \leq c$. By  Theorems 2.6(iii) and 2.7(i) in \cite{kuznetsov_theory_2013}, we have
\begin{align}\label{pm}
    \mathbb{E}_x \left[\int_0^{\tau_c^+ \wedge \tau_b^-} e^{-qt} h(X(t))\, \diff t\right]&=  \frac{W^{(q)}(x - b)}{W^{(q)}(c - b)} \rho_b^{(q)}(c; h) -  \rho_b^{(q)}(x; h), \\
    \E_x\left[e^{-q\tau_b^-};\tau_b^-<\tau_c^+\right] &= Z^{(q)}(x-b)-\frac{Z^{(q)}(c-b)}{W^{(q)}(c-b)}W^{(q)}(x-b). 
\end{align}
On the other hand, by Lemma 2.1 in \cite{loeffen_occupation_2014}, we have
\begin{align}
    \E_x\left[e^{-q \tau_b^-} Z^{(q + \lambda)}(X(\tau_b^-) - a); \tau_b^- < \tau_c^+ \right] &= \mathscr{Z}^{(q,\lambda)}_{b}(x, a) - \frac{W^{(q)}(x - b)}{W^{(q)}(c - b)} \mathscr{Z}^{(q,\lambda)}_{b}(c, a), \\
    \E_x\left[e^{-q\tau_b^-} \rho_a^{(q + \lambda)}(X(\tau_b^-); h); \tau_b^- < \tau_c^+\right] &= \int_a^b h(y) \left(\mathscr{W}^{(q,\lambda)}_{b}(x, y) - \frac{W^{(q)}(x - b)}{W^{(q)}(c - b)} \mathscr{W}^{(q,\lambda)}_{b}(c, y)\right)\, \diff y. \label{8}
\end{align}

Denote by $U^{(a)} = (U^{(a)}(t); t \geq 0)$ the spectrally negative L\'evy process with classical reflection from below at $a$ and $\kappa_c^+ \coloneqq \inf\{t \geq 0: U^{(a)}(t) > c\}$. By Theorem 2.8(i), (iii) in \cite{kuznetsov_theory_2013}, for $x \leq c$, we have
\begin{equation}\label{tsep_r}
    \E_x\left[e^{-q \kappa_c^+}\right] = \frac{Z^{(q)}(x - a)}{Z^{(q)}(c - a)}, \quad
    \E_x\left[\int_0^{\kappa_c^+} e^{-qt} h(U^{(a)}(t))\, \diff t\right] = 
    \frac{Z^{(q)}(x - a)}{Z^{(q)}(c - a)} \rho_a^{(q)}(c; h) - \rho_a^{(q)}(x; h).
\end{equation}

By identity (6) in \cite{loeffen_occupation_2014}, for $b, c \in \mathbb{R}$, we have 
\begin{align} \label{LRZ}
\begin{split}
    \lambda\int_b^c W^{(q + \lambda)}(c - u) W^{(q)}(u - b)\, \diff u &= W^{(q + \lambda)}(c - b) - W^{(q)}(c - b),\\
    \lambda\int_b^c W^{(q + \lambda)}(c - u) Z^{(q)}(u - b)\, \diff u &= Z^{(q + \lambda)}(c - b) - Z^{(q)}(c - b).
\end{split}
\end{align}
Similarly, one can verify that $\lambda\int_b^c W^{(q + \lambda)}(c - u) \overline{Z}^{(q)}(u - b)\, \diff u = \overline{Z}^{(q + \lambda)}(c - b) - \overline{Z}^{(q)}(c - b)$ by taking Laplace transforms of both sides. For $h: \mathbb{R} \to \mathbb{R}$ and $b < c$, Fubini's theorem together with \eqref{LRZ} yields
\begin{equation}\label{Eq: g(x; c) misc 1}
    \lambda \int_b^c W^{(q + \lambda)}(c - y) \int_b^y h(z) W^{(q)}(y - z)\, \diff z\, \diff y = \int_b^c h(z) \left(W^{(q + \lambda)}(c - z) - W^{(q)}(c - z) \right)\, \diff z.
\end{equation}
Moreover, for $a, b, c \in \mathbb{R}$, 
\begin{align}
\begin{split}\label{Eq: scale identity}
    \mathscr{W}^{(q,\lambda)}_{b}(c, a) &= W^{(q)}(c - a) + \lambda\int_{a}^{b} W^{(q)}(c - u)W^{(q + \lambda)}(u - a)\, \diff u,\\ 
    \overline{\mathscr{W}}^{(q,\lambda)}_{b}(c, a) &=  \overline{W}^{(q)}(c - a) + \lambda\int_{a}^{b} W^{(q)}(c - u)\overline{W}^{(q + \lambda)}(u - a)\, \diff u, \\ 
    \mathscr{Z}^{(q,\lambda)}_{b}(c, a) &= Z^{(q)}(c - a) + \lambda\int_{a}^{b} W^{(q)}(c - u)Z^{(q + \lambda)}(u - a)\, \diff u, \\
    \overline{\mathscr{Z}}^{(q,\lambda)}_{b}(c, a) &= \overline{Z}^{(q)}(c - a) + \lambda\int_{a}^{b} W^{(q)}(c - u) \overline{Z}^{(q + \lambda)}(u - a)\, \diff u.
\end{split}
\end{align}

The following identities can be derived from those listed above. 
\begin{lemma} For $a < b < c$,
\begin{align}\label{Eq: P1 misc 1}
    \int_b^c W^{(q + \lambda)}(c - y) \mathscr{Z}^{(q,\lambda)}_{b}(y, a) \, \diff y 
    &= \frac{Z^{(q+\lambda)}(c - a) - \mathscr{Z}^{(q,\lambda)}_{b}(c, a)} \lambda, \\
\label{Eq: P1 misc 3}
    \int_b^c W^{(q + \lambda)}(c - y) \int_a^b h(z) \mathscr{W}^{(q,\lambda)}_{b}(y, z)\, \diff z\, \diff y &= \int_a^b h(z) \int_b^c W^{(q)}(c - y) W^{(q + \lambda)}(y - z)\, \diff y\, \diff z.
\end{align}
\end{lemma}
\begin{proof}
(i) By \eqref{Eq: g(x; c) misc 1} with $h(z) = Z^{(q + \lambda)}(z - a)$, the left-hand side of \eqref{Eq: P1 misc 1} equals
\begin{multline*}
    \int_b^c W^{(q + \lambda)}(c - y) Z^{(q + \lambda)}(y - a) \diff y -\lambda \int_b^c W^{(q + \lambda)}(c - y)  \int_{b}^{y} W^{(q)}(y - z) Z^{(q + \lambda)}(z - a)\, \diff z \diff y \\
    = \int_b^c Z^{(q + \lambda)}(y - a) W^{(q)}(c - y)\, \diff y = \frac {Z^{(q+\lambda)}(c-a) - \mathscr{Z}^{(q,\lambda)}_{b}(c, a) } \lambda.
\end{multline*}
(ii) From \eqref{Eq: scale scr}, the left-hand side of \eqref{Eq: P1 misc 3} equals
\begin{equation}\label{Eq: lemma B1 misc 1}
     \int_b^c W^{(q + \lambda)}(c - y) \int_a^b h(z) W^{(q + \lambda)}(y - z)\, \diff z\, \diff y - H,
\end{equation}
where, by Fubini's theorem and \eqref{LRZ}, 
\begin{align*}
    H &\coloneqq \int_a^b h(z) \left( \int_b^c W^{(q + \lambda)}(u - z) \left(W^{(q + \lambda)}(c - u) - W^{(q)}(c - u) \right) \, \diff u \right)\, \diff z.
\end{align*}
Upon substituting $H$ into \eqref{Eq: lemma B1 misc 1}, several terms cancel, leaving exactly the right-hand side of \eqref{Eq: P1 misc 3}.

\end{proof}

\begin{lemma}\label{Lemma: B2} For $a < b < c$,
\begin{align}\label{Eq: control misc 1}
    \lim_{c \to \infty} \frac{\mathcal{K}_{b}^{(q, \lambda)}(c, a)}{Z^{(q)}(c - a)} &= \frac{\Theta^{(q + \lambda)}(b - a, \Phi(q))}{Z^{(q + \lambda)}(b - a)}. 
\end{align}
\end{lemma}
\begin{proof}
Identity \eqref{Eq: scale identity} together with \eqref{Eq: W, Z limits} yields
\begin{align*}
    &~~ \frac{\mathcal{K}_{b}^{(q, \lambda)}(c, a)}{Z^{(q)}(c - a)} = \frac{q}{\lambda + q} \frac{1}{Z^{(q + \lambda)}(b - a)} \left(\frac{Z^{(q)}(c - a)}{Z^{(q)}(c - a)} + \lambda\int_{a}^{b}\frac{W^{(q)}(c - y)}{Z^{(q)}(c - a)} Z^{(q + \lambda)}(y - a)\, \diff y\right) + \frac{\lambda}{\lambda + q} \frac{Z^{(q)}(c - b)}{Z^{(q)}(c - a)}\\
    &\xrightarrow{c \to \infty} \frac{q}{\lambda + q} \frac{1}{Z^{(q + \lambda)}(b - a)} \left(1 + \lambda \frac{\Phi(q)}{q}\int_{0}^{b - a} e^{-\Phi(q)y} Z^{(q + \lambda)}(y)\, \diff y\right) + \frac{\lambda}{\lambda + q} e^{-\Phi(q)(b - a)}.
\end{align*}
Simplification with integration by parts yields \eqref{Eq: control misc 1}.
\end{proof}

\subsection{Proof of Lemma \ref{Lemma: inventory cost}}\label{Appendix: proof of inventory cost}
For $c \in \mathbb{R}$, define
\begin{equation}
    \eta_c \coloneqq \inf \{t > 0: Y^{a, b}(t) > c\}. \label{def_eta}
\end{equation}
We first obtain the following result:
\begin{proposition} \label{prop_resolvent_capped}
Suppose $h$ is a positive, bounded, measurable function on $\mathbb{R}$ with compact support. For $a < b < c$,
\begin{align}\label{9.1}
\begin{split}
    g(x; c) &\coloneqq  \E_x\left[\int_0^{\eta_c} e^{-qt} h(Y^{a,b}(t))\, \diff t\right] =
   -  \mathcal{H}^{(q, \lambda)}_{b}(x, a; h) +  \frac {\mathcal{H}^{(q, \lambda)}_{b}(c, a; h)} {\mathcal{K}_{b}^{(q, \lambda)}(c, a)} \mathcal{K}_{b}^{(q, \lambda)}(x, a), \quad x \leq c.
\end{split}
\end{align}
\end{proposition}

\begin{proof}
(i):  Fix $x \leq b$. Since $Y^{a,b} = U^{(a)}$ on $[0, \kappa_b^+ \wedge T(1)]$ and $Y^{a,b} (\kappa_b^+ \wedge T(1)) = U^{(a)} (\kappa_b^+ \wedge T(1)) = b$ (where $T(1)$ is the first discounted replenishment opportunity), we can apply the strong Markov property to obtain
\begin{align}\label{0}
    g(x; c) = \E_x\left[\int_0^{\kappa_b^+ \wedge T(1)} e^{-qt} h(U^{(a)}(t)) \, \diff t\right] + \E_x \left[e^{-q(\kappa_b^+ \wedge T(1))}\right] g(b; c).
\end{align}
By first conditioning on $T(1)$, and then using \eqref{tsep_r}, we obtain 
\begin{equation*}
    \E_x\left[\int_0^{\kappa_b^+ \wedge T(1)} e^{-qt} h(U^{(a)}(t)) \, \diff t\right] = \E_x\left[\int_0^{\kappa_b^+} e^{-(q + \lambda)t} h(U^{(a)}(t))\, \diff t\right]
    = \frac{Z^{(q + \lambda)}(x - a)}{Z^{(q + \lambda)}(b - a)}\rho_a^{(q + \lambda)}(b; h) -\rho_a^{(q + \lambda)}(x; h),
\end{equation*}
\begin{equation*}
    \E_x \left[e^{-q(\kappa_b^+ \wedge T(1))}\right] = \mathbb{E}_x \left[e^{-qT(1)}; T(1) < \kappa_b^+\right] + \mathbb{E}_x \left[e^{-q\kappa_b^+}; \kappa_b^+< T(1)\right] \nonumber = \frac{\lambda}{\lambda + q}+\frac{q}{\lambda + q}\frac{Z^{(q + \lambda)}(x-a)}{Z^{(q + \lambda)}(b - a)}.
\end{equation*}
Substituting these identities into \eqref{0},
\begin{equation}\label{4}
    g(x; c) = \frac{Z^{(q + \lambda)}(x - a)}{Z^{(q + \lambda)}(b - a)} \frac q {\lambda + q} \mathcal{G}^h(b;c) + \frac{\lambda}{q + \lambda} g(b; c) - \rho_a^{(q + \lambda)}(x; h),
\end{equation}
where
\begin{equation}
    \mathcal{G}^h(b;c) \coloneqq \frac{\lambda + q} q \rho_a^{(q + \lambda)}(b; h) +  g(b; c). \label{G_expression_recursive}
\end{equation}

Now, fix $b < x \leq c$. Using the strong Markov property in conjunction with \eqref{4}, we obtain
\begin{align}
\begin{split}
    g(x; c) &= \E_x\left[\int_0^{\tau_b^-\wedge\tau_c^+} h(X(t))\, \diff t\right] + \E_x\left[e^{-q\tau_b^-} g(X(\tau_b^-); c); \tau_b^- < \tau_c^+\right] \\
    &= g_1(x;c) + \frac{\lambda}{\lambda + q} g(b; c) g_2(x; c) + \frac{q}{\lambda + q} \frac{\mathcal{G}^h(b;c)} {Z^{(q + \lambda)}(b - a)} g_3(x;c) - g_4(x;c),
    \end{split} \label{9}
\end{align}
where, using \eqref{pm}--\eqref{8},
\begin{align*}
    g_1(x; c) &\coloneqq \frac{W^{(q)}(x - b)}{W^{(q)}(c - b)} \rho_b^{(q)}(c; h) -  \rho_b^{(q)}(x; h),\quad g_2(x; c) \coloneqq Z^{(q)}(x - b) - \frac{W^{(q)}(x - b)}{W^{(q)}(c - b)} Z^{(q)}(c - b), \\
    g_3(x;c) &\coloneqq \mathscr{Z}^{(q,\lambda)}_{b}(x, a) - \frac{W^{(q)}(x - b)}{W^{(q)}(c - b)} \mathscr{Z}^{(q,\lambda)}_{b}(c, a), \\
    g_4(x;c) &\coloneqq \int_a^b h(y) \left(\mathscr{W}^{(q,\lambda)}_{b}(x, y) - \frac{W^{(q)}(x - b)}{W^{(q)}(c - b)} \mathscr{W}^{(q,\lambda)}_{b}(c, y)\right)\, \diff y.
\end{align*}

(ii): Now, we compute $\mathcal{G}^h(b; c)$. Recall $\kappa_c^+ \coloneqq \inf\{t \geq 0: U^{(a)}(t) > c\}$, using the strong Markov property,
\begin{equation}\label{g(b; c)}
    g(b; c) = B_1 + B_2 g(b; c) + B_3.
\end{equation} 
In the above, by \eqref{tsep_r}, we have
\begin{align}\label{15}
\begin{split}
    B_1 & \coloneqq \E_b \left[\int_0^{\kappa_c^+ \wedge T(1)} e^{-qt} h(U^{(a)}(t))\, \diff t\right] = \frac{Z^{(q + \lambda)}(b - a)}{Z^{(q + \lambda)}(c - a)} \rho_a^{(q + \lambda)}(c; h) - \rho_a^{(q + \lambda)}(b; h),
    \end{split} \\
\label{16}
\begin{split}
    B_2 &\coloneqq \E_b \left[e^{-qT(1)}; U^{(a)}(T(1)) < b, T(1) < \kappa_c^+\right] =\lambda\E_b \left[\int_0^{\kappa_c^+} e^{-(q + \lambda)t} 1_{\{U^{(a)}(t) < b\}}\, \diff t\right]\\
    &= \frac{Z^{(q + \lambda)}(b - a)}{Z^{(q + \lambda)}(c - a)} \frac{\lambda}{\lambda + q} \left(Z^{(q + \lambda)}(c - a) - Z^{(q + \lambda)}(c - b)\right) - \frac{\lambda}{\lambda + q} \left(Z^{(q + \lambda)}(b - a) - 1\right),
    \end{split}\\
    \begin{split}\label{diff}
        B_3 &\coloneqq \E_b \left[e^{-qT(1)} g(U^{(a)}(T(1)); c); U^{(a)}(T(1)) > b, T(1) < \kappa_c^+ \right]\\
        &= \lambda\E_b\left[\int_0^{\kappa_c^+} e^{-(q + \lambda)t} g(U^{(a)}(t); c) 1_{\{U^{(a)}(t) > b\}}\, \diff t\right] = \lambda\frac{Z^{(q + \lambda)}(b - a)}{Z^{(q + \lambda)}(c - a)} \int_b^c g(y; c) W^{(q + \lambda)}(c - y)\, \diff y.
    \end{split}
\end{align}
Substituting \eqref{G_expression_recursive} into \eqref{g(b; c)}, we obtain
\begin{align} \label{G_expression_recursive_new}
    \mathcal{G}^h(b;c)&=  B_1 +  B_2  \left(  \mathcal{G}^h(b;c) -  \frac {\lambda + q}  q \rho_a^{(q + \lambda)}(b; h) \right) + B_3 +  \frac {\lambda + q}  q \rho_a^{(q + \lambda)}(b; h).
\end{align}
Using \eqref{G_expression_recursive_new} together with \eqref{15}--\eqref{diff}, we compute $\mathcal{G}^h(b; c)$. We first compute the integral $\lambda\int_b^c g(y; c) W^{(q + \lambda)}(c - y)\, \diff y$ in \eqref{diff}. Recall \eqref{9}, we apply \eqref{LRZ}, \eqref{Eq: g(x; c) misc 1}, \eqref{Eq: P1 misc 1}, and \eqref{Eq: P1 misc 3} to obtain the following simplifications:
\begin{align*}
    \lambda\int_b^cW^{(q + \lambda)}(c - y) 
    g_1(y;c) \diff y
    &= \frac{W^{(q + \lambda)}(c - b)}{W^{(q)}(c - b)} \rho_b^{(q)}(c; h) - \rho_b^{(q + \lambda)}(c; h), \\
    \lambda\int_b^c W^{(q + \lambda)}(c-y) 
    g_2(y;c) \diff y
    &= Z^{(q + \lambda)}(c - b) - \frac{W^{(q + \lambda)}(c - b)}{W^{(q)}(c - b)} Z^{(q)}(c - b), \\
    \lambda\int_b^c W^{(q + \lambda)}(c - y) 
    g_3(y;c) \diff y &
    = Z^{(q + \lambda)}(c - a) - \frac{W^{(q + \lambda)}(c - b)}{W^{(q)}(c - b)} \mathscr{Z}^{(q,\lambda)}_{b}(c, a),\\
    \lambda\int_b^c W^{(q + \lambda)}(c - y) 
    g_4(y;c)
    \diff y &= \rho_a^{(q + \lambda)}(c; h) - \rho_b^{(q + \lambda)}(c; h) -\frac{W^{(q + \lambda)}(c - b)}{W^{(q)}(c - b)} \int_a^b h(z) \mathscr{W}^{(q,\lambda)}_{b}(c, z)\, \diff z.
\end{align*}
Putting these together, we have
\begin{align} \label{g_integral}
\begin{split}
    &\lambda \int_b^c g(u; c) W^{(q + \lambda)}(c - u)\, \diff u =  \frac{W^{(q + \lambda)}(c - b)}{W^{(q)}(c - b)} \left(\rho_b^{(q)}(c; h) + \int_a^b h(z) \mathscr{W}^{(q,\lambda)}_{b}(c, z)\, \diff z\right)\\
    & + \left(\frac{\lambda}{\lambda + q}\mathcal{G}^h(b;c) - \frac{\lambda}{q} \rho_a^{(q + \lambda)}(b; h)\right) \left(Z^{(q + \lambda)}(c - b) - \frac{W^{(q + \lambda)}(c - b)}{W^{(q)}(c - b)} Z^{(q)}(c - b) \right) \\
    &+ \frac {\frac q {\lambda+q} \mathcal{G}^h(b;c)} {Z^{(q + \lambda)}(b - a)}  \left(  Z^{(q + \lambda)}(c - a) - \frac{W^{(q + \lambda)}(c - b)}{W^{(q)}(c - b)} \mathscr{Z}^{(q,\lambda)}_{b}(c, a) \right)  - \rho_a^{(q + \lambda)}(c; h),
\end{split}
\end{align}
where, in the second line, we expressed $g(b; c)$ in terms of $\mathcal{G}^h(b;c)$ using \eqref{G_expression_recursive}. Upon substituting the expressions for $B_i$, $i = 1, 2, 3$, and \eqref{g_integral} into \eqref{G_expression_recursive_new}, we obtain $\mathcal{G}^h(b;c) = A + B  \mathcal{G}^h(b;c)$, where
\begin{equation*}
    A \coloneqq \frac{Z^{(q + \lambda)}(b - a)}{Z^{(q + \lambda)}(c - a)} \frac{W^{(q + \lambda)}(c - b)}{W^{(q)}(c - b)} \mathcal{H}^{(q, \lambda)}_{b}(c, a; h), \quad B \coloneqq 1 -  \frac {Z^{(q + \lambda)}(b - a)} {Z^{(q + \lambda)}(c - a)} \frac{W^{(q + \lambda)}(c - b)}{W^{(q)}(c - b)} \mathcal{K}_{b}^{(q, \lambda)}(c, a).
\end{equation*}
Hence, $\mathcal{G}^h(b;c) = \mathcal{H}^{(q, \lambda)}_{b}(c, a; h)/\mathcal{K}_{b}^{(q, \lambda)}(c, a)$. Substituting $\mathcal{G}^h(b;c)$ into \eqref{9}, we obtain \eqref{9.1} for $b < x \leq c$. In fact, \eqref{9.1} also holds for all $x \leq b$, since in this case it reduces to \eqref{4}.
\end{proof}

\begin{proposition}\label{Prop: resolvent computation}
    Suppose $h$ satisfies the assumptions of Proposition \ref{prop_resolvent_capped}. For $a < b$ and $x \in \mathbb{R}$,
    \begin{align}\label{Eq: v^h(x)}
        v_{a, b}^{h}(x) \coloneqq \mathbb{E}_x\left[\int^\infty_0e^{-qt} h(Y^{a, b}(t))\, \diff t\right] = 
        - \mathcal{H}_{b}^{(q, \lambda)}(x, a; h) + \mathcal{G}^h(b) \mathcal{K}_{b}^{(q, \lambda)}(x, a),
    \end{align}
    where $\mathcal{G}^h(b)$ is given in \eqref{Eq: v^h(b)} with $f = h$.
\end{proposition}
\begin{proof}
In view of \eqref{9.1}, we compute $\lim_{c \to \infty} {\mathcal{H}^{(q, \lambda)}_{b}(c, a; h)} / {\mathcal{K}_{b}^{(q, \lambda)}(c, a)}$. By \eqref{Eq: W, Z limits} and \eqref{Eq: scale identity}, the following limit holds:
\begin{equation*}
    \lim_{c \to \infty} \int_a^b h(y) \frac{\mathscr{W}^{(q,\lambda)}_{b}(c, y)}{Z^{(q)}(c-b)}\, \diff y = \frac{\Phi(q)}{q} \left( \int_a^b e^{-\Phi(q)(y - b)} h(y) \, \diff y + \lambda \int_a^b e^{-\Phi(q)(y - b)}\rho_a^{(q + \lambda)}(y; h)\, \diff y \right).
\end{equation*}
By analogous computations, we obtain $\lim_{c \to \infty} \rho_b^{(q)}(c; h)/Z^{(q)}(c-b) = q^{-1}\Phi(q)\int^\infty_b e^{-\Phi(q)(y - b)} h(y)\, \diff y$. Hence, 
\begin{equation*}
    \lim_{c \to \infty} \frac {\mathcal{H}^{(q, \lambda)}_{b}(c, a; h)}  {Z^{(q)}(c - b)}    
    = \frac{\Phi(q)}{q} \int_a^\infty e^{-\Phi(q)(y - b)} h(y) \, \diff y + \lambda\frac{\Phi(q)}{q}\int_a^b e^{-\Phi(q)(y - b)}\rho_a^{(q + \lambda)}(y, h)\, \diff y + \frac{\lambda}{q}\rho_a^{(q + \lambda)}(b; h).
\end{equation*}
On the other hand, by Lemma \ref{Lemma: B2} and \eqref{Eq: W, Z limits},
\begin{align*}
 \lim_{c \to \infty} \frac{\mathcal{K}_{b}^{(q, \lambda)}(c, a)}{Z^{(q)}(c - b)} = \lim_{c \to \infty} \frac{\mathcal{K}_{b}^{(q, \lambda)}(c, a)}{Z^{(q)}(c - a)} \frac{Z^{(q)}(c - a)}{Z^{(q)}(c - b)} &= \frac{\Theta^{(q + \lambda)}(b - a, \Phi(q))}{Z^{(q + \lambda)}(b - a)} e^{\Phi(q)(b-a)}. 
\end{align*}
Thus, $\lim_{c \to \infty} \mathcal{H}^{(q, \lambda)}_{b}(c, a; h) / \mathcal{K}_{b}^{(q, \lambda)}(c, a) = \mathcal{G}^h(b)$.
\end{proof}

To complete the proof of Lemma \ref{Lemma: inventory cost}, we first note that Proposition \ref{Prop: resolvent computation} remains valid for $h = f: \mathbb{R} \to \mathbb{R}$, given that $f$ satisfies Assumption \ref{asm: f convexity} and $X$ satisfies Assumption \ref{asm: on X}. Moreover, applying integration by parts, we find that \eqref{Eq: v^h(b)} is equivalent to \eqref{Eq: v^f(b)}. 

\subsection{Proof of Lemma \ref{Lemma: control costs}}\label{Appendix: proof of control cost}
To obtain \eqref{Eq: control costs}, we use \cite[Proposition 5.1]{perez_dual_2020}, which derives fluctuation identities for spectrally positive L\'evy processes, the dual of the spectrally negative L\'evy processes studied in this work.

(i) 
By \cite[Proposition 5.1]{perez_dual_2020} and the spatial homogeneity, for $x < c$ and $a < b < c$, 
\begin{equation*}
    \mathbb{E}_x \left[\int_{[0, \eta_c]} e^{-qt}\, \diff R_p^{a, b}(t)\right] = \mathbb{E}_{x-c} \left[\int_{[0, \eta_0]} e^{-qt}\, \diff R_p^{a-c, b-c}(t)\right] = \mathcal{K}_{b}^{(q, \lambda)}(x, a) \frac{k_{b}^{(q, \lambda)}(c, a)}{\mathcal{K}_{b}^{(q, \lambda)}(c, a)} - k_{b}^{(q, \lambda)}(x, a),
\end{equation*}
with $\eta_c$ as defined in \eqref{def_eta} and $\mathcal{K}_{b}^{(q, \lambda)}(c, a)$ as defined in \eqref{Eq: script K}.  For $y \in \mathbb{R}$,
\begin{align}
    k_{b}^{(q, \lambda)}(y, a) &\coloneqq \frac{\lambda}{\lambda + q} \left(\overline{Z}^{(q)}(y - b) - \frac{\lambda}{q}\overline{Z}^{(q + \lambda)}(b - a) Z^{(q)}(y - b) - \overline{\mathscr{Z}}^{(q,\lambda)}_{b}(y, a)\right). \label{Eq: k}
\end{align}
We seek $\lim_{c \to \infty} k_{b}^{(q, \lambda)}(c, a)/\mathcal{K}_{b}^{(q, \lambda)}(c, a)$. First, we compute
\begin{align}
\begin{split}\label{Eq: control misc 2}
    \lim_{c \to \infty} \frac{k_{b}^{(q, \lambda)}(c, a)}{Z^{(q)}(c - a)} &= \lim_{c \to \infty} \frac{\lambda}{\lambda + q} \left(\frac{\overline{Z}^{(q)}(c - b)}{Z^{(q)}(c - a)} - \frac{\lambda}{q}\overline{Z}^{(q + \lambda)}(b - a) \frac{Z^{(q)}(c - b)}{Z^{(q)}(c - a)} - \frac{\overline{\mathscr{Z}}^{(q,\lambda)}_{b}(c, a)}{{Z^{(q)}(c - a)}}\right) \\
    &= \lim_{c \to \infty} \frac{\lambda}{\lambda + q} \left(\frac{e^{-\Phi(q)(b - a)}}{\Phi(q)} - \frac{\lambda}{q} \overline{Z}^{q + \lambda}(b - a) e^{-\Phi(q)(b - a)} - \frac{\overline{\mathscr{Z}}^{(q,\lambda)}_{b}(c, a)}{{Z^{(q)}(c - a)}}\right). 
\end{split}
\end{align}
Using identity \eqref{Eq: scale identity}, we obtain
\begin{equation}
    \lim_{c \to \infty} \frac{\overline{\mathscr{Z}}^{(q,\lambda)}_{b}(c, a)}{{Z^{(q)}(c - a)}} = \frac{1}{\Phi(q)} + \frac{\lambda}{q}\int_{0}^{b - a} \Phi(q) e^{-\Phi(q)y} \overline{Z}^{(q + \lambda)}(y)\, \diff y. \label{Eq: control misc 3}
\end{equation}
Substituting \eqref{Eq: control misc 3} into \eqref{Eq: control misc 2} and simplifying with integration by parts, we obtain
\begin{equation}\label{Eq: control misc 4}
    \lim_{c \to \infty} \frac{k_{b}^{(q, \lambda)}(c, a)}{Z^{(q)}(c - a)} = 
    \frac{\lambda}{\lambda + q}\frac{e^{-\Phi(q)(b - a)}}{\Phi(q)} \left(1 + \frac{\lambda}{q} Z^{(q + \lambda)}(b - a)\right) - \frac{\lambda}{q}\frac{1}{\Phi(q)} \Theta^{(q + \lambda)}(b - a, \Phi(q)).
\end{equation}
Combining \eqref{Eq: control misc 1} and \eqref{Eq: control misc 4}, we obtain, for  $k_1$ as defined in \eqref{Eq: k_1},
\begin{align}\label{Eq: dRp}
    \mathbb{E}_x \left[\int_{[0, \infty)} e^{-qt}\, \diff R_p^{a, b}(t)\right] &= \mathcal{K}_{b}^{(q, \lambda)}(x, a) k_1 - k_{b}^{(q, \lambda)}(x, a).
\end{align}

(ii) By \cite[Proposition 5.2]{perez_dual_2020} and the spatial homogeneity of L\'evy processes, for $a < b < c$ and $x < c$, 
\begin{equation*}
    \mathbb{E}_x \left[\int_{[0, \eta_c]} e^{-qt}\, \diff R_c^{a, b}(t)\right] = \mathbb{E}_{x - c} \left[\int_{[0, \eta_0]} e^{-qt}\, \diff R_c^{a - c, b - c}(t)\right] = \mathcal{K}_{b}^{(q, \lambda)}(x, a) \frac{i_{b}^{(q, \lambda)}(c, a)}{\mathcal{K}_{b}^{(q, \lambda)}(c, a)} - i_{b}^{(q, \lambda)}(x, a),
\end{equation*}
where, for $y \in \mathbb{R}$,
\begin{equation}\label{Eq: i}
    i_{b}^{(q, \lambda)}(y, a) \coloneqq \overline{\mathscr{Z}}^{(q,\lambda)}_{b}(y, a) - \psi'(0+)\overline{\mathscr{W}}^{(q,\lambda)}_{b}(y, a) + \frac{\lambda}{q} Z^{(q)}(y - b) \left(\overline{Z}^{(q + \lambda)}(b - a) - \psi'(0+) \overline{W}^{(q + \lambda)}(b - a)\right).
\end{equation}
We seek $\lim_{c \to \infty} i_{b}^{(q, \lambda)}(c, a)/\mathcal{K}_{b}^{(q, \lambda)}(c, a)$. Using \eqref{Eq: scale identity} together with the definitions of $Z^{(q)}$ and $Z^{(q + \lambda)}$, and the limits in \eqref{Eq: W, Z limits}, we have 
\begin{equation*}
    \lim_{c \to \infty} \frac{\overline{\mathscr{W}}^{(q,\lambda)}_{b}(c, a)}{Z^{(q)}(c - a)} = \frac{1}{q} + \frac{\lambda}{q}\int_{0}^{b - a} \Phi(q) e^{-\Phi(q)y} \overline{W}^{(q + \lambda)}(y)\, \diff y.
\end{equation*}
Using \eqref{Eq: control misc 3} and the above limit, and simplifying by integration by parts, we obtain
\begin{equation*}
    \lim_{c \to \infty} \frac{i_{b}^{(q, \lambda)}(c, a)}{Z^{(q)}(c - a)}
    = -\frac{\lambda}{q}\frac{1}{\Phi(q)} e^{-\Phi(q)(b - a)} Z^{(q + \lambda)}(b - a) + \left(\frac{\lambda + q}{q} \frac{1}{\Phi(q)} - \frac{\psi'(0+)}{q}\right)\Theta^{(q + \lambda)}(b - a, \Phi(q)).
\end{equation*}
Combining $\lim_{c \to \infty} i_{b}^{(q, \lambda)}(c, a)/Z^{(q)}(c - a)$ and \eqref{Eq: control misc 1}, we obtain
\begin{align}\label{Eq: dRc}
    \mathbb{E}_x \left[\int_{[0, \infty)} e^{-qt}\, \diff R_c^{a, b}(t)\right] &= \mathcal{K}_{b}^{(q, \lambda)}(x, a) k_2 - i_{b}^{(q, \lambda)}(x, a).
\end{align}
Finally, by substituting \eqref{Eq: script K}, \eqref{Eq: k_1}, \eqref{Eq: k_2}, \eqref{Eq: k}, and \eqref{Eq: i} into \eqref{Eq: dRp} and \eqref{Eq: dRc}, we obtain \eqref{Eq: control costs}.

\subsection{Proof of Lemma \ref{Lemma: NPV of costs}}\label{Appendix: proof of NPV of costs}
By combining \eqref{Eq: inventory cost} and \eqref{Eq: control costs}, we have
\begin{align}\label{Eq: NPV of costs misc 1}
\begin{split}
    v_{a, b}(x) &=
    - \mathcal{H}^{(q, \lambda)}_{b}(x, a; f)
    + Z^{(q + \lambda)}(b-a) \mathcal{K}^{(q,\lambda)}_{b}(x, a)
    A - \frac{\lambda}{\lambda + q} K_p \overline{Z}^{(q)}(x - b)\\
    &~~ + K_c \psi'(0+) \left(\overline{\mathscr{W}}^{(q,\lambda)}_{b}(x, a) + \frac{\lambda}{q} Z^{(q)}(x - b) \overline{W}^{(q + \lambda)}(b - a) \right)\\
    &~~ + \left(\frac{\lambda}{\lambda + q} K_p - K_c\right) \left(\overline{\mathscr{Z}}^{(q,\lambda)}_{b}(x, a) + \frac{\lambda}{q} Z^{(q)}(x - b) \overline{Z}^{(q + \lambda)}(b - a)\right),
\end{split}
\end{align}
where
\begin{equation}\label{Eq: NPV of costs misc 2}
    A \coloneqq \frac{\mathcal{G}^f(b)
    }{Z^{(q + \lambda)}(b - a)} +\frac{K_p k_1 + K_c k_2}{Z^{(q + \lambda)}(b - a)}.
\end{equation}
We proceed to simplify $A$. By \eqref{Eq: v^f(b)},
\begin{equation}\label{Eq: misc 3}
    \frac{
    \mathcal{G}^f(b)
    }{Z^{(q + \lambda)}(b - a)} = \frac{1}{q} \left(\frac{ \int_0^\infty e^{-\Phi(q)y} f'(y + a) \diff y + \lambda \int_0^{b - a} e^{-\Phi(q)y}\rho_a^{(q + \lambda)}(y + a; f') \, \diff y}{\Theta^{(q + \lambda)}(b - a, \Phi(q))} + f(a)\right).
\end{equation}
We rewrite the above expression using $\tilde{f}'$ as defined in \eqref{Eq: tilde f}. Note that
\begin{align}
    \int_0^\infty e^{-\Phi(q)y} f'(y + a) \diff y &= \int_0^\infty e^{-\Phi(q)y} \tilde{f}'(y + a) \diff y + \left(\frac{\lambda}{\lambda + q} K_p - K_c\right)\frac{\lambda + q}{\Phi(q)} \label{Eq: misc 4}, \\
    \rho_a^{(q + \lambda)}(x; f') &= \rho_a^{(q + \lambda)}(x; \tilde{f}') + \left(\frac{\lambda}{\lambda + q} K_p - K_c\right)(\lambda + q) \overline{W}^{(q + \lambda)}(x - a), \quad x \in \mathbb{R}. \label{Eq: misc 5}
\end{align}
By \eqref{Eq: misc 5} and integration by parts, we have
\begin{multline}\label{Eq: misc 6}
    \int_0^{b - a} e^{-\Phi(q)y}\rho_a^{(q + \lambda)}(y + a; f') \, \diff y = \int_0^{b - a} e^{-\Phi(q)y} \rho_a^{(q + \lambda)}(y + a; \tilde{f}') \, \diff y\\
    - \left(\frac{\lambda}{\lambda + q} K_p - K_c\right) \frac{\lambda + q}{\Phi(q)} \left(e^{-\Phi(q)(b - a)} \overline{W}^{(q + \lambda)}(b - a) -  \int_0^{b - a} e^{-\Phi(q)y} W^{(q + \lambda)}(y) \, \diff y\right).
\end{multline}
Substituting \eqref{Eq: misc 4} and \eqref{Eq: misc 6} into \eqref{Eq: misc 3}, we have
\begin{multline*}
    \frac{\mathcal{G}^f(b)}{Z^{(q + \lambda)}(b - a)} = \frac{1}{q}\frac{\int_0^\infty e^{-\Phi(q)y} \tilde{f'}(y + a) \diff y + \lambda\int_0^{b - a} e^{-\Phi(q)y}\rho_a^{(q + \lambda)}(y + a; \tilde{f}') \, \diff y}{\Theta^{(q + \lambda)}(b - a, \Phi(q))}\\
    -  \frac {\lambda+q} q \frac{\frac{\lambda}{\Phi(q)} \left(\frac{\lambda}{\lambda + q} K_p - K_c\right) e^{-\Phi(q)(b - a)} \overline{W}^{(q + \lambda)}(b - a)}{\Theta^{(q + \lambda)}(b - a, \Phi(q))} +  \frac {\lambda+q} q \frac{1}{\Phi(q)} \left(\frac{\lambda}{\lambda + q} K_p - K_c\right) +  \frac{f(a)}{q}.
\end{multline*}
Now, using \eqref{Eq: k_1} and \eqref{Eq: k_2}, we obtain
\begin{multline*}
    \frac{K_p k_1 + K_c k_2}{Z^{(q + \lambda)}(b - a)} = \frac{\lambda + q}{q} \frac{\frac{\lambda}{\Phi(q)} \left(\frac{\lambda}{\lambda + q} K_p - K_c\right) e^{-\Phi(q)(b - a)} \overline{W}^{(q + \lambda)}(b - a) + \frac{1}{\lambda + q}\frac{\lambda}{\Phi(q)} \left(K_p - K_c\right) e^{-\Phi(q)(b - a)}}{\Theta^{(q + \lambda)}(b - a, \Phi(q))}\\
    - \frac{\lambda + q}{q}\frac{1}{\Phi(q)} \left(\frac{\lambda}{\lambda + q} K_p - K_c\right) - \frac{K_c\psi'(0+)}{q}.
\end{multline*}
Hence, after simplification, 
\begin{equation}\label{Eq: NPV of costs misc 3}
    A = \frac{1}{q} \left(\frac{\Gamma(a, b)}{\Theta^{(q + \lambda)}(b - a, \Phi(q))} + f(a) - K_c\psi'(0+)\right),
\end{equation}
where $\Gamma$ is defined in \eqref{Eq: Gamma}. Substituting $A$ into \eqref{Eq: NPV of costs misc 1}, we obtain \eqref{Eq: NPV of costs}.

\subsection{Proof of Lemma \ref{Lemma: NPV derivatives}}\label{Appendix: proof of NPV derivatives}
It suffices to establish \eqref{Eq: NPV derivative}, since \eqref{Eq: NPV derivative 2} follows from \eqref{Eq: gamma} together with \eqref{Eq: NPV derivative}, and \eqref{Eq: NPV 2nd derivative} follows by differentiation of \eqref{Eq: NPV derivative}. By the smoothness of the scale functions, for $x \in \mathbb{R}\backslash \{a, b\}$:
\begin{align}
    \overline{\mathscr{W}}^{(q,\lambda)'}_{b}(x, a) &= -\lambda W^{(q)}(x - b) \overline{W}^{(q + \lambda)}(b - a) + \mathscr{W}^{(q,\lambda)}_{b}(x, a), \label{Eq: bar W scr derivative} \\
    \mathscr{Z}^{(q,\lambda)'}_{b}(x, a) &= -\lambda W^{(q)}(x - b) Z^{(q + \lambda)}(b - a) + (q + \lambda)\mathscr{W}^{(q,\lambda)}_{b}(x, a), \label{Eq: W scr derivative} \\
    \overline{\mathscr{Z}}^{(q,\lambda)'}_{b}(x, a) &= -\lambda W^{(q)}(x - b) \overline{Z}^{(q + \lambda)}(b - a) + \mathscr{Z}^{(q,\lambda)}_{b}(x, a), \label{Eq: bar Z scr derivative}
\end{align}
where the derivatives are taken with respect to the first argument. Using \eqref{Eq: bar W scr derivative}--\eqref{Eq: bar Z scr derivative} to take the derivative of \eqref{Eq: control costs}, 
\begin{multline}\label{Eq: control cost derivative}
    \frac{\diff}{\diff x} v^{r}_{a, b}(x) = q\frac{\mathscr{W}^{(q,\lambda)}_{b}(x, a)}{Z^{(q + \lambda)}(b - a)} (K_p k_1 + K_c k_2) - \frac{\lambda}{\lambda + q} K_p Z^{(q)}(x - b) + K_c \psi'(0+) \mathscr{W}^{(q,\lambda)}_{b}(x, a)\\
    + \left(\frac{\lambda}{\lambda + q} K_p - K_c\right) \mathscr{Z}^{(q,\lambda)}_{b}(x, a), \quad x \in \mathbb{R}\backslash \{a\}.
\end{multline}
Similarly, for $x \in \mathbb{R}\backslash \{a, b\}$, by differentiating \eqref{Eq: inventory cost}, we have
\begin{equation}\label{Eq: inventory derivative misc 1}
    \frac{\diff }{\diff x} v_{a, b}^{f}(x) = -(f(b) + \lambda \rho_a^{(q + \lambda)}(b; f) ) W^{(q)}(x - b) - \rho_b^{(q)}(x; f') - \frac{\diff }{\diff x} \int_a^b f(y) \mathscr{W}^{(q,\lambda)}_{b}(x, y)\, \diff y + q\mathcal{G}^f(b)\frac{\mathscr{W}^{(q,\lambda)}_{b}(x, a)}{Z^{(q + \lambda)}(b - a)}.
\end{equation}
By integration by parts, for $x \in \mathbb{R}\backslash \{a, b\}$, 
\begin{equation}\label{Eq: derivative misc 1}
    \frac{\diff }{\diff x}\int_a^b f(y) W^{(q + \lambda)}(x - y)\, \diff y = f(a) W^{(q + \lambda)}(x - a) - f(b) W^{(q + \lambda)}(x - b) + \int_a^b f'(y) W^{(q + \lambda)}(x - y)\, \diff y
\end{equation}
and
\begin{multline*}
    \frac{\diff }{\diff x}\int_a^b f(y)\left(\lambda \int_b^x W^{(q)}(x - z) W^{(q + \lambda)}(z - y)\, \diff z\right)\, \diff y\\
    = \lambda W^{(q)}(x - b) \rho_a^{(q + \lambda)}(b; f) + \int_a^b f(y) \left(\lambda \int_b^x W^{(q)}(x - z) W^{(q + \lambda)'}(z - y)\, \diff z\right)\, \diff y.
\end{multline*}
Applying integration by parts once more, together with \eqref{LRZ}, the right-hand side can be written as
\begin{multline}\label{Eq: derivative misc 2}
    \lambda W^{(q)}(x - b) \rho_a^{(q + \lambda)}(b; f) + f(a) \lambda \int_b^x W^{(q)}(x - z) W^{(q + \lambda)}(z - a)\, \diff z \\
    - f(b)\left(W^{(q + \lambda)}(x - b)- W^{(q)}(x - b)\right) + \int_a^b f'(y) \left(\lambda \int_b^x W^{(q)}(x - z) W^{(q + \lambda)}(z - y)\, \diff z\right)\, \diff y.
\end{multline}
Substituting \eqref{Eq: derivative misc 1} and \eqref{Eq: derivative misc 2} into \eqref{Eq: inventory derivative misc 1}, we obtain the following simplification
\begin{equation}\label{Eq: inventory derivative misc 2}
    \frac{\diff}{\diff x} v_{a, b}^{f}(x) = -\rho_b^{(q)}(x; f') - \int_a^b f'(y) \mathscr{W}^{(q,\lambda)}_{b}(x, y)\, \diff y - f(a) \mathscr{W}^{(q,\lambda)}_{b}(x, a) + q\mathcal{G}^f(b) \frac{\mathscr{W}^{(q,\lambda)}_{b}(x, a)}{Z^{(q + \lambda)}(b - a)}.
\end{equation}
Now, we rewrite \eqref{Eq: inventory derivative misc 2} using $\tilde{f}'$ as defined in \eqref{Eq: tilde f}. Equation \eqref{Eq: misc 5} together with \eqref{Eq: scale identity} and Fubini's theorem yields
\begin{align}
\begin{split}
    &\rho_b^{(q)}(x; f') + \int_a^b f'(y) \mathscr{W}^{(q,\lambda)}_{b}(x, y)\, \diff y \\
    &= \rho_b^{(q)}(x; \tilde{f}') + \int_a^b \tilde{f}'(y) \mathscr{W}^{(q,\lambda)}_{b}(x, y)\, \diff y + \left(\frac{\lambda}{\lambda + q} K_p - K_c\right)\left(\mathscr{Z}^{(q,\lambda)}_{b}(x, a) - \left(1 - \lambda\overline{W}^{(q)}(x - b)\right)\right).
\end{split}\label{Eq: inventory derivative misc 3}
\end{align}

Substituting \eqref{Eq: inventory derivative misc 3} into \eqref{Eq: inventory derivative misc 2}, then combining with \eqref{Eq: control cost derivative}, on $\mathbb{R}\backslash \{a\}$, we have 
\begin{multline}\label{Eq: NPV derivative misc 1}
    v_{a, b}'(x) = -\rho_b^{(q)}(x; \tilde{f}') - \int_a^b \tilde{f}'(y) \mathscr{W}^{(q,\lambda)}_{b}(x, y)\, \diff y + \lambda (K_c-K_p) \overline{W}^{(q)}(x - b) - K_c\\
    + \mathscr{W}^{(q,\lambda)}_{b}(x, a)\left(
    qA + K_c \psi'(0+) - f(a)\right),
\end{multline}
where $A$ is defined in \eqref{Eq: NPV of costs misc 2}, and it is helpful to note that 
\begin{equation*}
    \lambda (K_c-K_p) \overline{W}^{(q)}(x - b) - K_c = -\frac{\lambda}{\lambda + q} K_p Z^{(q)}(x - b) + \left(\frac{\lambda}{\lambda + q} K_p - K_c\right) \left(1 - \lambda\overline{W}^{(q)}(x - b)\right). 
\end{equation*}
Substituting \eqref{Eq: NPV of costs misc 3} into \eqref{Eq: NPV derivative misc 1}, we obtain \eqref{Eq: NPV derivative}.

\subsection{Proof of Lemma \ref{Lemma: auxiliary result}}\label{Appendix: proof of auxiliary result} 
By the resolvent measure obtained in \eqref{Eq: v^h(x)}, we have
\begin{equation}\label{Eq: v^f' misc 0}
    v_{a, b}^{f'}(x) = \mathbb{E}_x\left[\int^\infty_0 e^{-qt} f'(Y^{a, b}(t))\, \diff t\right] = 
    - \mathcal{H}_{b}^{(q, \lambda)}(x, a; f') + \mathcal{G}^{f'}(b) \mathcal{K}_{b}^{(q, \lambda)}(x, a).
\end{equation}
We write the above in terms of $\tilde{f}'$. First, using \eqref{Eq: misc 5} and \eqref{Eq: inventory derivative misc 3}, we have 
\begin{equation}\label{Eq: H^f'}
    \mathcal{H}_{b}^{(q, \lambda)}(x, a; f') = \mathcal{H}_{b}^{(q, \lambda)}(x, a; \tilde{f}') + \left(\frac{\lambda}{\lambda + q} K_p - K_c\right)  \frac{\lambda + q}{q} Z^{(q + \lambda)}(b - a) \mathcal{K}_{b}^{(q, \lambda)}(x, a) - \left(\frac{\lambda}{\lambda + q} K_p - K_c\right) \frac{\lambda + q}{q}.
\end{equation}
Using \eqref{Eq: misc 4}--\eqref{Eq: misc 6} and \eqref{Eq: v^h(b)} with $f = f'$, we write $\mathcal{G}^{f'}(b)$ in terms of $\tilde{f}'$,
\begin{align*}
    \mathcal{G}^{f'}(b) &= \frac{\frac{1}{q}\int_0^\infty \Phi(q)e^{-\Phi(q)y} \tilde{f}'(y + a) \diff y + \frac{\lambda}{q} \int_0^{b - a} \Phi(q) e^{-\Phi(q)y} \rho_a^{(q + \lambda)}(y + a; \tilde{f}') \, \diff y}{\Theta^{(q + \lambda)}(b - a, \Phi(q))}Z^{(q + \lambda)}(b - a)\\
    &+ \frac{e^{-\Phi(q)(b - a)}\frac{\lambda}{q} \rho_a^{(q + \lambda)}(b; \tilde{f}')}{\Theta^{(q + \lambda)}(b - a, \Phi(q))} Z^{(q + \lambda)}(b - a) + \frac{\lambda + q}{q} \left(\frac{\lambda}{\lambda + q} K_p - K_c\right) Z^{(q + \lambda)}(b - a).
\end{align*}
Recall $\Gamma$ as defined in \eqref{Eq: Gamma} and $\gamma$ as defined in \eqref{Eq: gamma}. Adding and subtracting to $\mathcal{G}^{f'}(b)$ the following, 
\begin{equation*}
    \frac{\frac{\Phi(q)}{q}\frac{\lambda}{\Phi(q)} e^{-\Phi(q)(b - a)} (K_p - K_c)}{\Theta^{(q + \lambda)}(b - a, \Phi(q))}Z^{(q + \lambda)}(b - a) = \frac{\frac{\lambda}{q} e^{-\Phi(q)(b - a)} (K_p - K_c)}{\Theta^{(q + \lambda)}(b - a, \Phi(q))}Z^{(q + \lambda)}(b - a),
\end{equation*}
we obtain 
\begin{equation}\label{Eq: G^f'}
    \mathcal{G}^{f'}(b) = \frac{\frac{\Phi(q)}{q} \Gamma(a, b) - \frac{\lambda}{q} e^{-\Phi(q)(b - a)} \gamma(a, b)}{\Theta^{(q + \lambda)}(b - a, \Phi(q))} Z^{(q + \lambda)}(b - a) + \frac{\lambda + q}{q} \left(\frac{\lambda}{\lambda + q} K_p - K_c\right) Z^{(q + \lambda)}(b - a). 
\end{equation}
Substituting \eqref{Eq: H^f'} and \eqref{Eq: G^f'} into \eqref{Eq: v^f' misc 0} yields \eqref{Eq: v^f'}. Evaluating \eqref{Eq: v^f'} at $x = a$ and $x = b$, and using identity \eqref{Eq: gamma}, we obtain \eqref{Eq: v^f' at a} and \eqref{Eq: v^f' at b}, respectively.

\subsection{Proof of Lemma \ref{Lemma: Gamma asymptotics}}\label{Appendix: proof of gamma asymptotics}
As $b \mapsto \Gamma(a, b)$ is continuous on $(a, \infty)$, Equation \eqref{Eq: Gamma 1} holds by setting $b = a+$ in \eqref{Eq: Gamma}. To establish \eqref{Eq: Gamma 2}, applying L'H\^opital's rule, we have
\begin{equation}\label{Eq: end value misc 1}
    \Gamma_2(a) = \lim_{b \to \infty} \frac{\rho_a^{(q + \lambda)}(b; \tilde{f}')}{W^{(q + \lambda)}(b - a)} = \lim_{b \to \infty} \int_a^b \tilde{f}'(y) \frac{W^{(q + \lambda)}(b - y)}{W^{(q + \lambda)}(b - a)} \, \diff y = \int_0^\infty e^{-\Phi(q + \lambda)y} \tilde{f}'(y + a)\, \diff y,
\end{equation}
where, for the last equality, we use the fact that $\frac{W^{(q + \lambda)}(b - y)}{W^{(q + \lambda)}(b - a)}$ monotonically increases to $e^{-\Phi(q + \lambda)(y - a)}$ as $b \to \infty$ for $a < y$. This follows from \eqref{Eq: W, Z limits} together with the fact that $\frac{\partial}{\partial b}\frac{W^{(q + \lambda)}(b - y)}{W^{(q + \lambda)}(b - a)} \geq 0$ for $a < y < b$ (see Equation (5.2) and Remark 2.1(3) of \cite{baurdoux_optimality_2015}). 
Finally, the strict monotonicity of $\Gamma_1$ and $\Gamma_2$ on $(-\infty, \bar{\bar{a}}]$ follows from the convexity of $\tilde{f}$ together with Assumption \ref{asm: f slope}. Moreover, the continuity of $\Gamma_1$ and $\Gamma_2$ follows from Assumption \ref{asm: f convexity} and dominated convergence.

\subsection{Proof of Lemma \ref{Lemma: bounds for a underline}}\label{Appendix: proof of bounds for a underline}
To verify the bound for $\underline{a}_2$, first, by monotone convergence and the convexity of $\tilde{f}$, we have $\Gamma_2(a) \xrightarrow{a \downarrow -\infty}\int^{\infty}_0 e^{-\Phi(q + \lambda) z} \tilde{f}'(-\infty)\, \diff z \in [-\infty, 0)$. Hence, $\underline{a}_2 > -\infty$. On the other hand, $\tilde{f}'(y) \geq 0$ for all $y \geq \bar{a}$, and with strict inequality for $y \geq \bar{\bar{a}}$. Therefore, $\Gamma_2(\bar{a}) > 0$, which, by the monotonicity of $\Gamma_2$, implies $\underline{a}_2 < \bar{a}$. 

To prove the bounds for $\underline{a}_1$, we apply monotone convergence and the inequality $K_p < K_c$ to show that $\underline{a}_1 > -\infty$. By \eqref{Eq: Gamma 1} and \eqref{Eq: tilde f}, we have $\Gamma_1(a) = \int_0^\infty e^{-\Phi(q)y} \left(f'(y + a) + qK_c\right) \diff y > \int_0^\infty e^{-\Phi(q)y} \left(f'(y + a) + qK_p\right) \diff y$. Then, it is clear that $\Gamma_1(\bar{\bar{a}}) > 0$. By the monotonicity of $\Gamma_1$, we have $\underline{a}_1 \leq \bar{\bar{a}}$. 

\subsection{Proof of Lemma \ref{Lemma: rho monotone}}\label{Appendix: proof of rho monotone}

We first show the statement for the case $a < \underline{a}_2$, and then for the case $a = \underline{a}_2$. Fix $a < \underline{a}_2$, taking a derivative of $b \mapsto \rho_{a}^{(q + \lambda)}(b; \tilde{f}')$ with respect to $b$, we have
\begin{multline}\label{Eq: rho prime}
  \frac{\partial}{\partial b} \rho_{a}^{(q + \lambda)}(b; \tilde{f}') = \int_{a}^b \tilde{f}'(y) W^{(q + \lambda)'}(b - y)\, \diff y + \tilde{f}'(b-) W^{(q + \lambda)}(0) \\
    = \int_{a}^b \tilde{f}''(y) W^{(q + \lambda)}(b - y) \, \diff y + \sum_{a < y < b} W^{(q + \lambda)}(b - y) \left[\tilde{f}'(y+) - \tilde{f}'(y-)\right] + \tilde{f}'(a+) W^{(q + \lambda)}(b - a).
\end{multline}
Dividing both sides by $W^{(q + \lambda)}(b - a)$ and taking another derivative with respect to $b$, we have
\begin{multline} \label{derivative_fraction_rho}
    \frac{\partial}{\partial b} \frac{\frac{\partial}{\partial b} \rho_{a}^{(q + \lambda)}(b; \tilde{f}')}{W^{(q + \lambda)}(b - a)} = \int_{a}^b \tilde{f}''(y) \frac{\partial}{\partial b}\frac{W^{(q + \lambda)}(b - y)}{W^{(q + \lambda)}(b - a)} \, \diff y + \tilde{f}''(b-) \frac{W^{(q + \lambda)}(0)}{W^{(q + \lambda)}(b - a)}\\
    + \sum_{a < y < b} \frac{\partial}{\partial b}\frac{W^{(q + \lambda)}(b - y)}{W^{(q + \lambda)}(b - a)} \left[\tilde{f}'(y+) - \tilde{f}'(y-)\right] \geq 0,
\end{multline}
where the inequality follows from $\tilde{f}'' \geq 0$ almost everywhere and the fact that $\frac{\partial}{\partial b}\frac{W^{(q + \lambda)}(b - y)}{W^{(q + \lambda)}(b - a)} \geq 0$ for $a < y < b$. This implies that $b \mapsto \frac{\partial}{\partial b} \rho_{a}^{(q + \lambda)}(b; \tilde{f}')$ is either (i) first negative and then positive, (ii) uniformly non-positive, or (iii) uniformly non-negative. We can immediately eliminate (iii)  since $\tilde{f}'(a) < 0$, which implies that, in \eqref{Eq: rho prime}, $\frac{\partial}{\partial b} \rho_{a}^{(q + \lambda)}(a + \varepsilon; \tilde{f}') < 0$ for every sufficiently small $\varepsilon > 0$, where $\tilde{f}' (\cdot) < 0$ on $[a, a+ \varepsilon]$. Additionally, (i) is not possible. Indeed, if (i) holds then $\rho_{a}^{(q + \lambda)}(y + a; \tilde{f}') \geq \rho_{a}^{(q + \lambda)}(\tilde{b}; \tilde{f}')$ for some $\tilde{b} > a$ and every $y \geq 0$. This implies 
\begin{equation*}
    \liminf_{b \to \infty}\Gamma(a, b) \geq  \liminf_{b \to \infty} \left(\int_0^\infty e^{-\Phi(q)y} \tilde{f'}(y + a) \, \diff y + \lambda \int_0^{b - a} e^{-\Phi(q)y}\rho_{a}^{(q + \lambda)} (\tilde{b}; \tilde{f}') \, \diff y  - \frac{\lambda}{\Phi(q)} e^{-\Phi(q)(b - a)} \left(K_c - K_p\right) \right).
\end{equation*}
Note that the right-hand side is finite, which contradicts \eqref{Gamma_b_asymp}.
Thus, (ii) always holds, i.e., $\frac{\partial}{\partial b} \rho_{a}^{(q + \lambda)}(b; \tilde{f}') \leq 0$ for $b > a$. We further show that  $\frac{\partial}{\partial b} \rho_{a}^{(q + \lambda)}(b; \tilde{f}') < 0$ for $b > a$. 

Suppose that $\frac{\partial}{\partial b} \rho_{a}^{(q + \lambda)}(b_0; \tilde{f}') = 0$ for some $b_0 > a$. Since $\frac{\partial}{\partial b} \rho_{a}^{(q + \lambda)}(b_0; \tilde{f}')/W^{(q + \lambda)}(b_0 - a) = 0$, the inequality established in \eqref{derivative_fraction_rho} yields
\begin{equation*}
    \frac{\frac{\partial}{\partial b} \rho_{a}^{(q + \lambda)}(b; \tilde{f}')}{W^{(q + \lambda)}(b - a)} \geq 0, \quad b > b_0 \implies \frac{\frac{\partial}{\partial b} \rho_{a}^{(q + \lambda)}(b; \tilde{f}')}{W^{(q + \lambda)}(b - a)} = 0, \quad b > b_0,
\end{equation*}
where the implication follows from the standing assumption of case (ii), i.e., $b \mapsto \frac{\partial}{\partial b} \rho_{a}^{(q + \lambda)}(b; \tilde{f}')$ is uniformly non-positive. Note that we must have
\begin{equation*}
    \frac{\partial}{\partial b} \rho_{a}^{(q + \lambda)}(b; \tilde{f}') = 0, \quad b > b_0.
\end{equation*}
Hence, it holds that $\lim_{b \to \infty} \rho_{a}^{(q + \lambda)}(b; \tilde{f}') = \rho_{a}^{(q + \lambda)}(b_0; \tilde{f}') \in \mathbb{R}$. In view of \eqref{Eq: Gamma}, this implies $\lim_{b \to \infty} \Gamma(a, b) \in \mathbb{R}$, which contradicts \eqref{Gamma_b_asymp}. Thus, we conclude that $\frac{\partial}{\partial b} \rho_{a}^{(q + \lambda)}(b; \tilde{f}') < 0$ for all $b > a$, and hence $b \mapsto \rho_{a}^{(q + \lambda)}(b; \tilde{f}')$ is strictly decreasing. This observation, in conjunction with $\rho_{a}^{(q + \lambda)}(a+; \tilde{f}') = 0$, implies that $\rho_{a}^{(q + \lambda)}(b; \tilde{f}')$ is negative for all $b > a$.

Now, set $a = \underline{a}_2$. An analogous argument to that used for the case $a < \underline{a}_2$ applies to eliminate case (iii) for the function $b \mapsto \frac{\partial}{\partial b} \rho_{\underline{a}_2}^{(q + \lambda)}(b; \tilde{f}')$. To further eliminate case (i), suppose that it holds for the sake of contradiction. Since $\frac{\partial}{\partial b} \rho_{\underline{a}_2}^{(q + \lambda)}(b_1; \tilde{f}') > c_0$ for some $b_1 > \underline{a}_2$ and $c_0 > 0$, by the inequality established in \eqref{derivative_fraction_rho}, there exists some $c_1 > 0$ such that
\begin{equation*}
    \frac{\frac{\partial}{\partial b} \rho_{\underline{a}_2}^{(q + \lambda)}(b; \tilde{f}')}{W^{(q + \lambda)}(b - \underline{a}_2)} \geq \frac{\frac{\partial}{\partial b} \rho_{\underline{a}_2}^{(q + \lambda)}(b_1; \tilde{f}')}{W^{(q + \lambda)}(b_1 - \underline{a}_2)} > c_1, \quad b > b_1. 
\end{equation*}
Since $W^{(q + \lambda)}(x) \sim \frac{e^{\Phi(q + \lambda)x}}{\psi'(\Phi(q + \lambda))}$ as $x \to \infty$, it follows from the above inequality and L'H\^opital's rule that $\lim_{b \to \infty} \rho_{\underline{a}_2}^{(q + \lambda)}(b; \tilde{f}')/ W^{(q + \lambda)}(b - \underline{a}_2) > 0$. In view of \eqref{Eq: end value misc 1}, this contradicts the definition of $\underline{a}_2$, which satisfies $\Gamma_2(\underline{a}_2) = 0$. Hence, case (i) can also be eliminated. Since case (ii) always holds, it follows that $b \mapsto \rho_{\underline{a}_2}^{(q + \lambda)}(b; \tilde{f}')$ is non-increasing and non-positive for all $b > \underline{a}_2$.

\subsection{Proof of Lemma \ref{Lemma: alt form gamma}}\label{Appendix: proof of alt form gamma}
By Fubini's theorem, we have
\begin{equation*}
    \int_0^{b - a} e^{-\Phi(q)y} \rho_a^{(q + \lambda)}(y + a; \tilde{f}') \, \diff y = \int_0^{b - a}e^{-\Phi(q)z} \tilde{f}'(z + a)\int_0^{b - a - z}e^{-\Phi(q)u} W^{(q + \lambda)}(u)\,\diff u\,\diff z.
\end{equation*}
The above equation together with $\Gamma$ as in \eqref{Eq: Gamma} yields
\begin{align*}
    \Gamma(a, b) - \frac{\lambda}{\Phi(q)} e^{-\Phi(q)(b - a)} \left(K_p - K_c\right) &= \int_0^{\infty}e^{-\Phi(q)z}\tilde{f}'(z+a)\left(1+\lambda\int_0^{b-a-z}e^{-\Phi(q)u}W^{(q+\lambda)}(u)\,\diff u\right)\,\diff z\\
    &= e^{-\Phi(q)(b-a)} \int_0^{\infty} \tilde{f}'(z + a) Z^{(q + \lambda)}(b - a - z, \Phi(q))\,\diff z,
\end{align*}
where $Z^{(q + \lambda)}(\cdot, \Phi(q))$ is defined in \eqref{Eq: Z phi}. After rearranging, we get \eqref{Eq: Gamma alt 2}. In \eqref{Eq: gamma alt 2}, the first equality follows from \eqref{Eq: gamma} and the second follows by differentiating \eqref{Eq: Gamma tilde} with respect to $b$. 

\subsection{Proof of Lemma \ref{Lemma: generator}}\label{Appendix: proof of generator}
(i): By setting $a = a^*$ and $b = b^*$ in \eqref{Eq: NPV of costs} and simplification, we obtain, for $ a^* < x < b^*$, 
\begin{multline}\label{Eq: vi misc 1}
    v_{a^*, b^*}(x) = - \left(\rho_{a^*}^{(q + \lambda)}(x; f) + \frac{\lambda}{q} \rho_{a^*}^{(q + \lambda)}(b^*; f)\right) + \frac{f(a^*)}{q + \lambda} \left(Z^{(q + \lambda)}(x - a^*) + \frac{\lambda}{q}Z^{(q + \lambda)}(b^* - a^*)\right)\\
    - \frac{\lambda}{\lambda + q} K_p (x - b^*) - \frac{K_c}{q} \psi'(0+) + \left(\frac{\lambda}{\lambda + q} K_p - K_c\right) \left(\overline{Z}^{(q + \lambda)}(x - a^*) + \frac{\lambda}{q} \overline{Z}^{(q + \lambda)}(b^* - a^*)\right).
\end{multline}
We now apply the operator $(\mathcal{L} - q)$ to \eqref{Eq: vi misc 1}. By the proof of Lemma 4.5 in \cite{egami_precautionary_2013}, we have $(\mathcal{L} - (q + \lambda)) \rho_{a^*}^{(q + \lambda)}(x; f) = f(x)$ for $a^* < x < b^*$. Therefore, 
\begin{equation}\label{Eq: vi misc 2}
    (\mathcal{L} - q) \rho_{a^*}^{(q + \lambda)}(x; f) = f(x) + \lambda\rho_{a^*}^{(q + \lambda)}(x; f).
\end{equation}
Similarly, by the proof of Theorem 2.1 in \cite{bayraktar_optimal_2014},   
\begin{equation*}
    (\mathcal{L} - (q + \lambda)) \left(\overline{Z}^{(q + \lambda)}(x - a^*) + \frac{\psi'(0+)}{q + \lambda}\right) = (\mathcal{L} - (q + \lambda)) Z^{(q + \lambda)}(x - a^*) =  0, \quad x > a^*,
\end{equation*}
from which it follows that 
\begin{equation}
    (\mathcal{L} - q) Z^{(q + \lambda)}(x - a^*) = \lambda Z^{(q + \lambda)}(x - a^*), \quad (\mathcal{L} - q) \overline{Z}^{(q + \lambda)}(x - a^*) = \lambda \overline{Z}^{(q + \lambda)}(x - a^*) + \psi'(0+) \label{Eq: vi misc 3}.
\end{equation}
Furthermore, it is known that
\begin{equation}
    (\mathcal{L} - q) (y \mapsto y)(x) = \psi'(0+) - qx, \quad x \in \mathbb{R}. \label{Eq: vi misc 5}
\end{equation}
Using \eqref{Eq: vi misc 1} and combining \eqref{Eq: vi misc 2}--\eqref{Eq: vi misc 5}, the following is immediate,
\begin{align*}
    (\mathcal{L} - q) v_{a^*, b^*}(x) + f(x) &= \lambda \left(-\rho_{a^*}^{(q + \lambda)}(x; f) + \rho_{a^*}^{(q + \lambda)}(b^*; f) \right)\\
    &+ \lambda\left(\frac{f(a^*)}{q + \lambda} \left(Z^{(q + \lambda)}(x - a^*) - Z^{(q + \lambda)}(b^* - a^*) \right) - \frac{q}{\lambda + q} K_p(b^* - x) \right) \\
    &+ \lambda\left(\frac{\lambda}{\lambda + q} K_p - K_c\right) \left(\overline{Z}^{(q + \lambda)}(x - a^*)- \overline{Z}^{(q + \lambda)}(b^* - a^*) \right).
\end{align*}
It can be verified that the above matches the expression for $-\lambda\left(v_{a^*, b^*}(b^*) - v_{a^*, b^*}(x) + K_p(b^* - x)\right)$.

(ii): For $x > b^*$, 
\begin{align}\label{Eq: vi misc 6}
\begin{split}
    v_{a^*, b^*}(x) &= -\int_{b^*}^x f(y) W^{(q)}(x - y) \, \diff y - \int_{a^*}^{b^*} f(y) \mathscr{W}^{(q,\lambda)}_{b^*}(x, y)\, \diff y - \frac{\lambda}{q} \rho_{a^*}^{(q + \lambda)}(b^*; f) Z^{(q)}(x - b^*) \\
    &+ \frac{f(a^*) - K_c \psi'(0+)}{q + \lambda} \left(\mathscr{Z}^{(q,\lambda)}_{b^*}(x, a^*) + \frac{\lambda}{q}Z^{(q)}(x - b^*)Z^{(q + \lambda)}(b^* - a^*)\right)\\
    &- \frac{\lambda}{\lambda + q} K_p \overline{Z}^{(q)}(x - b^*) + K_c \psi'(0+) \left(\overline{\mathscr{W}}^{(q,\lambda)}_{b^*}(x, a^*) + \frac{\lambda}{q} Z^{(q)}(x - b^*) \overline{W}^{(q + \lambda)}(b^* - a^*) \right)\\
    &+ \left(\frac{\lambda}{\lambda + q} K_p - K_c\right) \left(\overline{\mathscr{Z}}^{(q,\lambda)}_{b^*}(x, a^*) + \frac{\lambda}{q} Z^{(q)}(x - b^*) \overline{Z}^{(q + \lambda)}(b^* - a^*)\right).
\end{split}
\end{align}
By the proof of Lemma 4.5 in \cite{egami_precautionary_2013} and the proof of Lemma 5.1 in \cite{noba_optimal_2018}, we obtain
\begin{alignat}{2}\label{Eq: vi block 1}
    &(\mathcal{L} - q) \overline{\mathscr{W}}^{(q,\lambda)}_{b^*}(x, a^*) = 1, \quad & (\mathcal{L} - q) \mathscr{Z}^{(q,\lambda)}_{b^*}(x, a^*) = 0,\notag\\
    &(\mathcal{L} - q) \overline{\mathscr{Z}}^{(q,\lambda)}_{b^*}(x, a^*) = \psi'(0+), \quad & (\mathcal{L} - q) \left(\int_{b^*}^x f(y) W^{(q)}(x - y) \, \diff y\right) = f(x).
\end{alignat}
By the proof of Lemma 4 in \cite{avram_exit_2004}, $(\mathcal{L} - (q + \lambda)) W^{(q + \lambda)}(x - y) = 0$ for $x > y$. Using this result and dominated convergence, we obtain
\begin{equation}\label{Eq: vi misc 7}
    (\mathcal{L} - q) \left(\int_{a^*}^{b^*} f(y) W^{(q + \lambda)}(x - y)\, \diff y\right) = \lambda\int_{a^*}^{b^*} f(y) W^{(q + \lambda)}(x - y)\, \diff y. 
\end{equation}
Moreover, by Fubini's theorem and the bottom-right identity of \eqref{Eq: vi block 1}, we have
\begin{multline}\label{Eq: vi misc 8}
    (\mathcal{L} - q) \left(\int_{a^*}^{b^*} f(y) \int_{b^*}^{x} W^{(q)}(x - z)W^{(q + \lambda)}(z - y) \, \diff z\, \diff y\right) \\
    = (\mathcal{L} - q) \left(\int_{b^*}^{x} W^{(q)}(x - z) \int_{a^*}^{b^*} f(y) W^{(q + \lambda)}(z - y) \, \diff y\, \diff z\right) = \int_{a^*}^{b^*} f(y) W^{(q + \lambda)}(x - y) \, \diff y. 
\end{multline}
Combining \eqref{Eq: vi misc 7} and \eqref{Eq: vi misc 8}, we obtain $(\mathcal{L} - q) \left(\int_{a^*}^{b^*} f(y) \mathscr{W}^{(q,\lambda)}_{b^*}(x, y)\, \diff y\right) = 0$.

The expression for $v_{a^*, b^*}$ as given in \eqref{Eq: vi misc 6} together with the computations above yields $(\mathcal{L} - q) v_{a^*, b^*}(x) = - f(x)$ for $x > b^*$, which implies $(\mathcal{L} - q) v_{a^*, b^*}(x) + f(x) = 0$ for $x > b^*$. To show that this statement remains valid at $x = b^*$, it suffices to observe that as $v_{a^*, b^*}$ is sufficiently smooth, $x \mapsto (\mathcal{L} - q) v_{a^*, b^*}(x)$ is a continuous function (see \eqref{Eq: generator}). Since $f$ is also continuous, we further conclude that $(\mathcal{L} - q) v_{a^*, b^*} + f$ is continuous, and 
\begin{equation*}
     (\mathcal{L} - q) v_{a^*, b^*}(b^*) + f(b^*)=(\mathcal{L} - q) v_{a^*, b^*}(b^*+) + f(b^*+)  =  0. 
\end{equation*}
(iii): For $x \leq a^*$,
\begin{multline}\label{Eq: vi misc 9}
    v_{a^*, b^*}(x) = - \frac{\lambda}{q} \rho_{a^*}^{(q + \lambda)}(b^*; f) + \frac{f(a^*)}{q + \lambda} \left(1 + \frac{\lambda}{q} Z^{(q + \lambda)}(b^* - a^*)\right) - \frac{\lambda}{\lambda + q} K_p (x - b^*)\\
    - \frac{K_c}{q} \psi'(0+) + \left(\frac{\lambda}{\lambda + q} K_p - K_c\right) \left(x - a^* + \frac{\lambda}{q} \overline{Z}^{(q + \lambda)}(b^* - a^*)\right).
\end{multline}
Applying $(\mathcal{L} - q)$ to \eqref{Eq: vi misc 9} and simplifying, we obtain
\begin{multline*}
    (\mathcal{L} - q) v_{a^*, b^*}(x) = \lambda \rho_{a^*}^{(q + \lambda)}(b^*; f) - f(a^*) \left(1 + \lambda \overline{W}^{(q + \lambda)}(b^* - a^*)\right) - \lambda \left(\frac{\lambda}{\lambda + q} K_p - K_c\right) \overline{Z}^{(q + \lambda)}(b^* - a^*)\\
    + \frac{\lambda}{\lambda + q} K_p (qx - qb^*) - \left(\frac{\lambda}{\lambda + q} K_p - K_c\right) \left(qx - qa^*\right) .
\end{multline*}
On the other hand, by Corollary \ref{Corollary: generator}, we have
\begin{multline*}
    \mathcal{M}v_{a^*, b^*}(x) - v_{a^*, b^*}(x) 
    = - \rho_{a^*}^{(q + \lambda)}(b^*; f) + f(a^*) \overline{W}^{(q + \lambda)}(b^* - a^*) + \left(\frac{\lambda}{\lambda + q} K_p - K_c\right) \overline{Z}^{(q + \lambda)}(b^* - a^*)\\
    - \frac{q}{\lambda + q} K_p (x - b^*) - \left(\frac{\lambda}{\lambda + q} K_p - K_c\right) \left(x - a^*\right).
\end{multline*}
Thus, for $x \leq a^*$, 
\begin{equation*}
    (\mathcal{L}-q)v_{a^*, b^*}(x) + \lambda(\mathcal{M}v_{a^*, b^*}(x) - v_{a^*, b^*}(x)) + f(x) = \tilde{f}(x) - \tilde{f}(a^*).
\end{equation*}
Recalling Assumption \ref{asm: f slope}(1) and noting that $a^* < \bar{a}$, we have $\tilde{f}(x) - \tilde{f}(a^*) \geq 0$ for $x \leq a^*$. 

\subsection{Proof of Lemma \ref{Lemma: verification limit}.}\label{Appendix: proof of verification limit}
By It\^o's formula, for $n > 0$ and $t > 0$,
    \begin{align*}
    v_{a^*, b^*}(x) &= -\mathbb{E}_x\left[\int^{t \wedge \tau_n}_0 e^{-qs} (\mathcal{L}-q) v_{a^*, b^*}(X(s))\, \diff s \right] + \mathbb{E}_x\left[e^{-q(t\wedge \tau_n)} v_{a^*, b^*}(X(t\wedge \tau_n))\right]\\
    &\leq \mathbb{E}_x\left[\int^{t \wedge \tau_n}_0 e^{-qs} f(X(s))\, \diff s \right] + \mathbb{E}_x\left[e^{-q(t\wedge \tau_n)} v_{a^*, b^*}(X(t\wedge \tau_n))\right],
\end{align*}
where $\tau_n \coloneqq \inf\{t \geq 0: |X(t)| > n\}$ and the inequality follows from Corollary \ref{Corollary: generator}, Lemma \ref{Lemma: generator}, and the non-positivity of $\mathcal{M}v_{a^*, b^*}-v_{a^*, b^*}$. Then, by Lemma \ref{Lemma: polynomial growth} and Assumption \ref{asm: on X}, the argument of Lemma 11 in \cite{yamazaki_inventory_2017} applies and yields $\mathbb{E}_x\left[e^{-q(t\wedge \tau_n)} v_{a^*, b^*}(X(t\wedge \tau_n))\right] \to 0$ as $n, t \uparrow \infty$, and therefore $v_{a^*, b^*}(x) \leq \mathbb{E}_x\left[\int^\infty_0 e^{-qs} f(X(s))\, \diff s\right]$ for $x \in \mathbb{R}$. Thus, it follows that
    \begin{align}\label{Eq: upper bound}
    \begin{split}
        \mathbb{E}_x\left[e^{-q(t\wedge \tau_n)} v_{a^*, b^*}(Y^\pi(t\wedge \tau_n))\right] &\leq \mathbb{E}_x\left[ e^{-q(t \wedge \tau_n)} \mathbb{E}_{Y^\pi(t \wedge \tau_n)} \left[\int^\infty_{0}e^{-qs}f(X(s))\, \diff s\right]\right]\\
        &= \mathbb{E}_x\left[ \int^\infty_{t \wedge \tau_n}e^{-qs}f(X(s) + R^\pi_c(t \wedge \tau_n) + R^\pi_p(t \wedge \tau_n))\, \diff s\right],
        \end{split}
    \end{align}
    where the equality is due to the strong Markov property and the spatial homogeneity of $X$. As both $R_c^\pi$ and $R_p^\pi$ are non-negative, for $s \geq t \wedge \tau_n$, $X(s) \leq X(s) + R_c^\pi(t \wedge \tau_n) + R_p^\pi(t \wedge \tau_n) \leq Y^\pi(s)$. This, together with the convexity of the function $f$, gives the following bound,
    \begin{multline*}
        \mathbb{E}_x\left[\int^\infty_{t \wedge \tau_n} e^{-qs} f(X(s) + R_c^\pi(t \wedge \tau_n) + R_p^\pi(t \wedge \tau_n))\, \diff s\right] \\
        \leq \max \left(\mathbb{E}_x\left[\int^\infty_{t \wedge \tau_n} e^{-qs} f(X(s))\, \diff s\right] , \mathbb{E}_x\left[\int^\infty_{t \wedge \tau_n} e^{-qs} f(Y^\pi(s))\, \diff s\right]\right),
    \end{multline*}
    which vanishes as $n, t \uparrow \infty$ by Assumption \ref{asm: on X} and the assumption that $\pi$ is admissible. Hence, the limit supremum of \eqref{Eq: upper bound} is zero. 
\bibliographystyle{abbrv}
\bibliography{main}

\begin{thebibliography}{10}

\bibitem{ivanovs_exit_2016}
H.~Albrecher, J.~Ivanovs, and X.~Zhou.
\newblock Exit identities for {L}\'evy processes observed at {P}oisson arrival times.
\newblock {\em Bernoulli}, 22(3):1364--1382, 2016.

\bibitem{avanzi_dividend_2014}
B.~Avanzi, V.~Tu, and B.~Wong.
\newblock On optimal periodic dividend strategies in the dual model with diffusion.
\newblock {\em Insurance Math. Econom.}, 55:210--224, 2014.

\bibitem{avanzi_hybrid_2016}
B.~Avanzi, V.~Tu, and B.~Wong.
\newblock On the interface between optimal periodic and continuous dividend strategies in the presence of transaction costs.
\newblock {\em Astin Bull.}, 46(3):709--746, 2016.

\bibitem{avram_exit_2004}
F.~Avram, A.~E. Kyprianou, and M.~R. Pistorius.
\newblock Exit problems for spectrally negative {L}\'evy processes and applications to ({C}anadized) {R}ussian options.
\newblock {\em Ann. Appl. Probab.}, 14(1):215--238, 2004.

\bibitem{avram_parisian_2018}
F.~Avram, J.-L. P\'erez, and K.~Yamazaki.
\newblock Spectrally negative {L}\'evy processes with {P}arisian reflection below and classical reflection above.
\newblock {\em Stochastic Process. Appl.}, 128(1):255--290, 2018.

\bibitem{baurdoux_optimality_2015}
E.~J. Baurdoux and K.~Yamazaki.
\newblock Optimality of doubly reflected {L}\'evy processes in singular control.
\newblock {\em Stochastic Process. Appl.}, 125(7):2727--2751, 2015.

\bibitem{bayraktar_optimal_2014}
E.~Bayraktar, A.~E. Kyprianou, and K.~Yamazaki.
\newblock Optimal dividends in the dual model under transaction costs.
\newblock {\em Insurance Math. Econom.}, 54:133--143, 2014.

\bibitem{beckmann_compoundpoisson_1961}
M.~Beckmann.
\newblock An inventory model for arbitrary interval and quantity distributions of demand.
\newblock {\em Manag. Sci.}, 8(1):35--57, 1961.

\bibitem{benkherouf_optimality_2009}
L.~Benkherouf and A.~Bensoussan.
\newblock Optimality of an {$(s,S)$} policy with compound {P}oisson and diffusion demands: a quasi-variational inequalities approach.
\newblock {\em SIAM J. Control Optim.}, 48(2):756--762, 2009.

\bibitem{bensoussan_QVI_2005}
A.~Bensoussan, R.~H. Liu, and S.~P. Sethi.
\newblock Optimality of an {$(s,S)$} policy with compound {P}oisson and diffusion demands: a quasi-variational inequalities approach.
\newblock {\em SIAM J. Control Optim.}, 44(5):1650--1676, 2005.

\bibitem{berling_stochastic_2011}
P.~Berling and V.~M. de~Albéniz.
\newblock Optimal inventory policies when purchase price and demand are stochastic.
\newblock {\em Oper. Res.}, 59(1):109--124, 2011.

\bibitem{egami_precautionary_2013}
M.~Egami and K.~Yamazaki.
\newblock Precautionary measures for credit risk management in jump models.
\newblock {\em Stochastics}, 85(1):111--143, 2013.

\bibitem{egami_phase-type_2014}
M.~Egami and K.~Yamazaki.
\newblock Phase-type fitting of scale functions for spectrally negative {L}\'evy processes.
\newblock {\em J. Comput. Appl. Math.}, 264:1--22, 2014.

\bibitem{yam_lostsales_2024}
J.~Han, X.~Li, S.~P. Sethi, C.~C. Siu, and S.~C.~P. Yam.
\newblock Technical note—production management with general demands and lost sales.
\newblock {\em Oper. Res.}, 72(5):1751--1764, 2024.

\bibitem{hurter_regenerative_1967}
A.~P. Hurter and F.~C. Kaminsky.
\newblock An application of regenerative stochastic processes to a problem in inventory control.
\newblock {\em Oper. Res.}, 15(3):467--472, 1967.

\bibitem{hurter_discount_1968}
A.~P. Hurter and F.~C. Kaminsky.
\newblock Inventory control with a randomly available discount purchase price.
\newblock {\em OR}, 19(4):433--444, 1968.

\bibitem{kalymon_stochastic_1971}
B.~A. Kalymon.
\newblock Stochastic prices in a single-item inventory purchasing model.
\newblock {\em Oper. Res.}, 19(6):1434--1458, 1971.

\bibitem{kuznetsov_theory_2013}
A.~Kuznetsov, A.~E. Kyprianou, and V.~Rivero.
\newblock The theory of scale functions for spectrally negative {L}\'evy processes.
\newblock In {\em L\'evy matters {II}}, volume 2061 of {\em Lecture Notes in Math.}, pages 97--186. Springer, Heidelberg, 2012.

\bibitem{kyprianou_fluctuations_2014}
A.~E. Kyprianou.
\newblock {\em Fluctuations of {L}\'evy processes with applications}.
\newblock Universitext. Springer, Heidelberg, second edition, 2014.

\bibitem{landriault_potential_2018}
D.~Landriault, B.~Li, J.~T.~Y. Wong, and D.~Xu.
\newblock Poissonian potential measures for {L}\'evy risk models.
\newblock {\em Insurance Math. Econom.}, 82:152--166, 2018.

\bibitem{leung_analytic_2014}
T.~Leung, K.~Yamazaki, and H.~Zhang.
\newblock An analytic recursive method for optimal multiple stopping: {C}anadization and phase-type fitting.
\newblock {\em Int. J. Theor. Appl. Finance}, 18(5):1550032, 31, 2015.

\bibitem{loeffen_occupation_2014}
R.~L. Loeffen, J.-F.~c. Renaud, and X.~Zhou.
\newblock Occupation times of intervals until first passage times for spectrally negative {L}\'evy processes.
\newblock {\em Stochastic Process. Appl.}, 124(3):1408--1435, 2014.

\bibitem{moinzadeh_discount_1997}
K.~Moinzadeh.
\newblock Replenishment and stocking policies for inventory systems with random deal offerings.
\newblock {\em Manag. Sci.}, 43(3):334--342, 1997.

\bibitem{noba_optimal_2018}
K.~Noba, J.-L. P\'erez, K.~Yamazaki, and K.~Yano.
\newblock On optimal periodic dividend and capital injection strategies for spectrally negative {L}\'evy models.
\newblock {\em J. Appl. Probab.}, 55(4):1272--1286, 2018.

\bibitem{noba_dividend_2018}
K.~Noba, J.-L. P\'erez, K.~Yamazaki, and K.~Yano.
\newblock On optimal periodic dividend strategies for {L}\'evy risk processes.
\newblock {\em Insurance Math. Econom.}, 80:29--44, 2018.

\bibitem{noba_singular_2023}
K.~Noba and K.~Yamazaki.
\newblock On singular control for {L}\'evy processes.
\newblock {\em Math. Oper. Res.}, 48(3):1213--1234, 2023.

\bibitem{perez_dual_2020}
J.-L. P\'erez and K.~Yamazaki.
\newblock Optimality of hybrid continuous and periodic barrier strategies in the dual model.
\newblock {\em Appl. Math. Optim.}, 82(1):105--133, 2020.

\bibitem{perez_optimal_2020}
J.-L. P\'erez, K.~Yamazaki, and A.~Bensoussan.
\newblock Optimal periodic replenishment policies for spectrally positive {L}\'evy demand processes.
\newblock {\em SIAM J. Control Optim.}, 58(6):3428--3456, 2020.

\bibitem{silver_random_1993}
E.~A. Silver, D.~J. Robb, and M.~R. Rahnama.
\newblock Random opportunities for reduced cost replenishments.
\newblock {\em IIE Transactions}, 25(2):111--120, 1993.

\bibitem{yamazaki_inventory_2017}
K.~Yamazaki.
\newblock Inventory control for spectrally positive {L}\'evy demand processes.
\newblock {\em Math. Oper. Res.}, 42(1):212--237, 2017.

\bibitem{yamazaki_optimal_2025}
K.~Yamazaki and Q.~Zhang.
\newblock Optimal periodic double-barrier strategies for spectrally negative {L}\'evy processes.
\newblock {\em ArXiv}, May 2025.

\bibitem{zhao_dividend_2017}
Y.~Zhao, P.~Chen, and H.~Yang.
\newblock Optimal periodic dividend and capital injection problem for spectrally positive {L}\'evy processes.
\newblock {\em Insurance Math. Econom.}, 74:135--146, 2017.

\end{thebibliography}
\end{document}